\documentclass[12pt]{article}

\usepackage{amsfonts, amsmath, amssymb, amsgen, amsthm, amscd,
latexsym}
\usepackage{color}
\usepackage[all]{xy}

\def\Id{\mathop{\rm Id}\nolimits}

\def\Ad{\mathop{\rm Ad}\nolimits}
\def\ad{\mathop{\rm ad}\nolimits}
\def\det{\mathop{\rm det}\nolimits}

\def\Hom{\mathop{\rm Hom}\nolimits}
\def\Tot{\mathop{\rm Tot}\nolimits}

\def\Cb{{\mathbb C}}

\def\Rb{{\mathbb R}}

\def\Zb{{\mathbb Z}}

\def\Fc{{\cal F}}

\def\Hc{{\cal H}}

\def\Lc{{\cal L}}
\def\Mc{{\cal M}}

\def\Pc{{\cal P}}

\def\Uc{{\cal U}}
\def\Vc{{\cal V}}

\def\Lc{{\cal L}}

\def\Dc{{\cal D}}
\def\Uc{{\cal U}}
\def\Vc{{\cal V}}

\def\a{\alpha}
\def\b{\beta}
\def\d{\delta}
\def\D{\Delta}
\def\g{\gamma}

\def\om{\omega}

\def\s{\sigma}

\def\t{\theta}
\def\z{\zeta}
\def\ve{\varepsilon}
\def\vp{\varphi}
\def\nr{\natural}
\def\x{\xi}

\def\fl{\forall}

\def\nb{\nabla}
\def\ot{\otimes}

\def\ra{\rightarrow}

\def\rt{\triangleright}

\def\lt{\triangleleft}
\def\cl{\blacktriangleright\hspace{-4pt} < }
\def\al{>\hspace{-4pt}\vartriangleleft}

\def\acl{\blacktriangleright\hspace{-4pt}\vartriangleleft }

\def\bD{\blacktriangledown}

\def\bi{\bowtie}
\def\hd{\overset{\ra}{\partial}}

\def \vd{\uparrow\hspace{-4pt}\partial}
\def\hs{\overset{\ra}{\sigma}}
\def \vs{\uparrow\hspace{-4pt}\sigma}
\def\hta{\overset{\ra}{\tau}}
\def \vta{\uparrow\hspace{-4pt}\tau}
\def\hb{\overset{\ra}{b}}
\def\hP{\overset{\ra}{\p}}
\def \vb{\uparrow\hspace{-4pt}b}
\def \vP{{\uparrow\hspace{-1pt}\p}}

\def\hB{\overset{\ra}{B}}
\def \vB{\uparrow\hspace{-4pt}B}
\def\p{\partial}

\def\0D{\Delta^{(0)}}
\def\1D{\Delta^{(1)}}
\def\Db{\blacktriangledown}

\def\wg{\wedge}

\def\td{\tilde}

\newcommand{\FD}{\mathfrak{D}}

\newcommand{\Fa}{\mathfrak{a}}

\newcommand{\Fd}{\mathfrak{d}}
\newcommand{\Fg}{\mathfrak{g}}
\newcommand{\Fh}{\mathfrak{h}}
\newcommand{\Fk}{\mathfrak{k}}
\newcommand{\Fl}{\mathfrak{l}}

\newcommand{\Fs}{\mathfrak{s}}

\newtheorem{theorem}{Theorem}[section]
\newtheorem{remark}[theorem]{Remark}
\newtheorem{proposition}[theorem]{Proposition}
\newtheorem{lemma}[theorem]{Lemma}
\newtheorem{corollary}[theorem]{Corollary}

\newtheorem{definition}[theorem]{Definition}

\def\build#1_#2^#3{\mathrel{
\mathop{\kern 0pt#1}\limits_{#2}^{#3}}}
\newcommand{\ps}[1]{~\hspace{-4pt}_{^{(#1)}}}
\newcommand{\ns}[1]{~\hspace{-4pt}_{_{{<#1>}}}}
\newcommand{\sns}[1]{~\hspace{-4pt}_{_{{<\overline{#1}>}}}}

\def\odots{\ot\cdots\ot}
\def\wdots{\wedge\dots\wedge}

\numberwithin{equation}{section}

 \def\cf{{\it cf.\/}\ }

\def\a{\alpha}
\def\b{\beta}

\def\d{\delta}

\def\g{\gamma}

\def\i{\iota}

\def\om{\omega}
\def\s{\sigma}
\def\t{\theta}
\def\ve{\varepsilon}

\def\vp{\varphi}

\def\z{\zeta}

\def\D{\Delta}

\def\dt{\left.\frac{d}{dt}\right|_{_{t=0}}}

\def\fl{\forall}

\def\nb{\nabla}

\def\ot{\otimes}
\def\part{\partial}

\def\wdg{\wedge}

\def\ra{\rightarrow}

\def\text{\hbox}

\def\fl{\forall}

\def\nb{\nabla}
\def\ot{\otimes}

\def\ra{\rightarrow}

\def\wdg{\wedge}

\def\Ad{\mathop{\rm Ad}\nolimits}

\def\Hom{\mathop{\rm Hom}\nolimits}

\def\Id{\mathop{\rm Id}\nolimits}

\def\build#1_#2^#3{\mathrel{
\mathop{\kern 0pt#1}\limits_{#2}^{#3}}}

\numberwithin{equation}{section}
\newcommand{\comment}[1]{\relax}

\def\rt{\triangleright}

\def\lt{\triangleleft}
\def\cl{\blacktriangleright\hspace{-4pt} < }
\def\al{>\hspace{-4pt}\vartriangleleft}

\def\acl{\blacktriangleright\hspace{-4pt}\vartriangleleft }
\def\dcp{\vartriangleright\hspace{-4pt}\vartriangleleft }

\def\dt{\left.\frac{d}{dt}\right|_{_{t=0}}}
\def\ds{\left.\frac{d}{ds}\right|_{_{s=0}}}

\newcommand{\mdt}[1]{\left.\frac{d}{dt_{#1}}\right|_{_{t_{#1}=0}}}

\begin{document}
\title{\bf  Lie-Hopf  algebras and  their  Hopf cyclic cohomology}
\author{
\begin{tabular}{cc}
Bahram Rangipour \thanks{Department of Mathematics  and   Statistics,
     University of New Brunswick, Fredericton, NB, Canada}\quad and \quad  Serkan S\"utl\"u $~^\ast$
      \end{tabular}}

\maketitle
\begin{abstract}
The correspondence between Lie algebras, Lie groups, and algebraic groups, on one side and commutative Hopf algebras on the other side  are known for a long time by works of Hochschild-Mostow and others. We extend  this correspondence by associating a  noncommutative noncocommutative  Hopf algebra to any matched pair of Lie algebras, Lie groups, and affine algebraic groups. We canonically associate a  modular pair in involution to any of these Hopf algebras. More precisely, to any locally finite representation of a matched pair object as above we associate a  SAYD module to the corresponding Hopf algebra. At the end, we compute the   Hopf cyclic cohomology of the associated Hopf algebra with coefficients in the aforementioned SAYD module in terms of Lie algebra cohomology of the Lie algebra associated to the matched pair object relative to an appropriate Levi subalgebra with coefficients induced by the original representation.
\end{abstract}

\section{Introduction}
 Lie groups and Lie algebras are well-known for a long time and their  correspondence has been  adopted to many other algebraic structures such as  algebraic groups and their Lie algebras.  This correspondence  was advanced  by the work of Hochschild and Mostow  when they defined the notion of representative functions for Lie groups, Lie algebras, and algebraic groups.  They showed that  the algebra of representative functions on these objects forms a commutative Hopf algebra whose coalgebra structure  encodes the  group structure of the group  or the Lie bracket of the Lie algebra  respectively \cite{Hochschild-Mostow-57,Hochschild-Mostow-61}. They also defined a cohomology theory,  for Lie algebras , Lie groups, and affine algebraic groups, called  the representative cohomology due to the fact that it is based on the representatively injective resolutions \cite{Hochschild-Mostow-62}.  They proved their  van Est type  theorem, see \cite{vEst}, by computing the representative  cohomology  in terms of relative Lie algebra cohomology  for suitable pair of Lie algebras for each case   \cite{Hochschild-Mostow-62}. The important point for us is to see  this cohomology as the cohomology of the underlaying coalgebra of the commutative Hopf algebra with coefficients in the comodule induced by the representation in question.

\medskip

We extend the work of Hochschild-Mostow by  associating a not necessarily commutative or cocommutative Hopf algebra   to any matched pair of Lie algebras, Lie groups, and  algebraic groups. The resulting Hopf algebra   is  a bicrossed product Hopf algebra.
Such a Hopf algebra is made of two Hopf algebras in such a way that both algebra and coalgebra structures interact with each other.   We refer the reader to  \cite{Majid-book} for a comprehensive account on  these  Hopf algebras.  One of the interesting examples of bicrossed product Hopf algebras is $\Hc_n$, the Hopf algebra of general transverse symmetry on $\Rb^n$, defined by Connes and  Moscovici  \cite{Connes-Moscovici-98}. It is shown  in \cite{Khalkhali-Rangipour-06,Hadfield-Majid-07} that $\Hc_1$ is a  bicrossed product Hopf algebra. In \cite{Moscovici-Rangipour-09}, Moscovici and the first  author associated  to each infinite primitive Lie pseudogroup a bicrossed product Hopf algebra by means of  transverse symmetries.

One of our minor aims in this paper is to develop a short method for proving that a Hopf algebra is of the form of bicrossed product. We refer the reader to \cite{Moscovici-Rangipour-09,Hadfield-Majid-07}  for other solutions to this  problem. We introduce the notion of Lie-Hopf algebras as the  responsible objects for this method.

We extend the whole theory developed in \cite{Moscovici-Rangipour-09,Moscovici-Rangipour-011} including coefficients,  the Chevalley-Eilenberg    bicomplex, and van Est isomorphism  to the case of abstract Lie-Hopf algebras. In \cite{Moscovici-Rangipour-09} the  total Hopf algebra was first constructed by  means of an algebra of operators on a crossed product algebra and then it was shown to be of the form of a  bicrossed product Hopf algebra. However  in this paper we do not assume to have the total Hopf algebra from the beginning. This forces us to glue  two different Hopf algebras together by  dealing  with them conceptually.

Hopf cyclic cohomology was defined by Connes and Moscovici in \cite{Connes-Moscovici-98} and then was generalized by Hajac, Khalkhali, Sommerhauser, and the first author in \cite{Hajac-Khalkhali-Rangipour-Sommerhauser-04-1,Hajac-Khalkhali-Rangipour-Sommerhauser-04-2} to include the appropriate  coefficients.

In Section \ref{S1}  we first define the notion of Lie-Hopf algebras.  In three separated subsections, we explain carefully how to construct  Lie-Hopf algebras by having a matched pair of Lie algebras, a matched pair of Lie groups, and a matched pair of affine algebraic groups.

In Section \ref{S2} we deal with coefficients. In Subsection \ref{Canonical MPI associated to Lie-Hopf algebras}  we canonically associate a modular pair in involution to any Lie-Hopf algebra. In Subsection \ref{Induced Hopf cyclic coefficients} we classify a subcategory of Yetter-Drinfeld modules over the bicrossed product Hopf algebras  under the name of {\it induced } module. 

In Section \ref{S3} we study the Hopf cyclic cohomology of commutative Hopf algebras.  We define the notion of Hopf-Levi decomposition and prove a general van Est isomorphism by computing the Hopf cyclic cohomology of such  a Hopf algebra  with coefficients in an induced module. 

The Section \ref{S4} contains  the main results  and is in fact  the main motivation of the paper. In Subsection \ref{SS-Bicocyclic module associated to Lie Hopf algebras} we advance the machinery developed by Moscovici and the first author in \cite{Moscovici-Rangipour-07,Moscovici-Rangipour-09,Moscovici-Rangipour-011}  to cover all Lie-Hopf algebras. We prove that  for any $\Fg$-Hopf algebra $\Fc$ and any induced module $M$ there is a bicomplex computing the Hopf cyclic cohomology of $\Fc\acl U(\Fg)$.   We then prove  Theorem \ref{Theorem-main} in its full generality as is stated  here.
\addtocounter{section}{4}
\addtocounter{theorem}{8}
\begin{theorem}
Let $(\Fg_1,\Fg_2)$ be a matched pair of Lie algebras and $\Fc$ be a $(\Fg_1,\Fg_2)$-related Hopf algebra. Let assume that $\Fg_2=\Fh\ltimes \Fl$ is a $\Fc$-Levi decomposition such that $\Fh$ is $\Fg_1$-invariant and the natural action of $\Fh$ on $\Fg_1$ is given by derivations. Then for any $\Fc$-comodule and $\Fg_1$-module $M$, the map $\Vc$, defined in \eqref{VE}, is a map of bicomplexes and induces an isomorphism between Hopf cyclic cohomology of $\Fc\acl U(\Fg_1)$ with coefficients in $^\s{M}_\d$ and the Lie algebra cohomology of $\Fa:=\Fg_1\bowtie \Fg_2$ relative to $\Fh$ with coefficients in the $\Fa$-module induced by $M$. In other words,
\begin{equation}
HP^\bullet(\Fc\acl U(\Fg_1), ^\s{M}_\d)\cong \bigoplus_{i=\bullet\;\text{mod 2}} H^i(\Fa, \Fh,\;M).
\end{equation}
\end{theorem}

\addtocounter{section}{-4}
 We get three corollaries by specializing  this theorem to our geometric cases. 
\bigskip

\tableofcontents

\section{Geometric noncommutative Hopf algebras}
\label{S1}
In this paper  all Lie algebras, Lie groups, and affine algebraic groups are over complex numbers.  All  groups are assumed to be connected and all Lie algebras are finite dimensional.

 \bigskip 
 
In Subsection \ref{SS-bicrossed} we recall the needed definitions and basics of bicrossed and double crossed Hopf algebras. In Subsection \ref{SS-Lie-Hopf} we introduce Lie-Hopf algebras and proves that they are equivalent to bicrossed product Hopf algebras.  In Subsection \ref{SS-Lie algebra} we consider a matched pair of Lie algebra $(\Fg_1,\Fg_2)$ and construct the bicrossed product Hopf algebra  $R(\Fg_2)\acl U(\Fg_1)$, where $R(\Fg_2)$ is the Hopf algebra of representative functions on $U(\Fg_2)$. In Subsection \ref{SS-Lie group} we associate the Hopf algebra $R(G_2)\acl U(\Fg_1)$ to a matched pair of Lie groups $(G_1,G_2)$, where $R(G_2)$ is the Hopf algebra of representative smooth functions on $G_2$ and $\Fg_1$ is the Lie algebra of $G_1$. We finish  this section by Subsection \ref{SS-algebraic group} where we construct the Hopf algebra $\Pc(G_2)\acl U(\Fg_1)$ for a matched pair of connected affine algebraic group $(G_1,G_2)$, where $\Pc(G_2)$ is the Hopf algebra of all polynomial representative functions on $G_2$ and $\Fg_1$ is the Lie algebra of $G_1$.

\subsection{Bicrossed and double crossed product Hopf algebras}
\label{SS-bicrossed}
It is always helpful to decompose complicated algebraic structures into product of less complicated objects. In the category of Hopf algebras there are many of such decompositions. The first one helping us in this paper is bicrossed product and the second one is double crossed product. We refer the interested reader to \cite{Majid-book} for a comprehensive account, however we briefly  recall below the most basic notions concerning the bicrossed and double crossed product construction.

\medskip

Let  $\Uc$ and $\Fc$ be two Hopf algebras.  A linear map $$ \Db:\Uc\ra\Uc\ot \Fc , \qquad \Db u \, = \, u\ns{0} \ot u\ns{1} \, , $$ defines a {\em right
coaction}, and thus equips
 $\Uc$ with a  {right $\Fc-$comodule coalgebra} structure,  if the
following conditions are satisfied for any $u\in \Uc$:
\begin{align}\label{cc1}
&u\ns{0}\ps{1}\ot u\ns{0}\ps{2}\ot u\ns{1}= u\ps{1}\ns{0}\ot u\ps{2}\ns{0}\ot u\ps{1}\ns{1}u\ps{2}\ns{1}\\\label{cc2} &\epsilon(u\ns{0})u\ns{1}=\epsilon(u)1.
\end{align}
One  then  forms a  cocrossed product coalgebra $\Fc\cl\Uc$, that has $\Fc\ot \Uc$ as underlying vector space and the following coalgebra structure:
\begin{align}\label{cocross}
&\Delta(f\cl u)= f\ps{1}\cl u\ps{1}\ns{0}\ot  f\ps{2}u\ps{1}\ns{1}\cl u\ps{2}, \\ &\epsilon(f\cl u)=\epsilon(f)\epsilon(u).
\end{align}
In a dual fashion,   $\Fc$ is called a {left $\Uc-$module algebra}, if $\Uc$ acts from the left on $\Fc$ via a left action $$
\rt : \Fc\ot \Uc \ra \Fc
$$ which satisfies the following conditions for any $u\in \Uc$, and $f,g\in \Fc$ :
\begin{align}\label{ma1}
&u\rt 1=\epsilon(u)1\\\label{ma2}
 &u\rt(fg)=(u\ps{1}\rt
f)(u\ps{2}\rt g).
\end{align}
This time  we can endow the underlying vector space $\Fc\ot \Uc$ with
 an algebra structure, to be denoted by $\Fc\al \Uc$, with  $1\al 1$
 as its unit and the product given by
\begin{equation}
(f\al u)(g\al v)=f \;u\ps{1}\rt g\al u\ps{2}v
\end{equation}

$\Uc$ and $\Fc$ are said to form a matched pair of Hopf algebras if they are equipped, as above,  with an action and a coaction which satisfy the following
compatibility conditions for any $u\in\Uc$, and any $f\in \Fc$.
\begin{align}\label{mp1}
&\epsilon(u\rt f)=\epsilon(u)\epsilon(f), \\  \label{mp2} &\Delta(u\rt f)=u\ps{1}\ns{0} \rt f\ps{1}\ot u\ps{1}\ns{1}(u\ps{2}\rt f\ps{2}), \\  \label{mp3}
 &\Db(1)=1\ot
1, \\ \label{mp4} &\Db(uv)=u\ps{1}\ns{0} v\ns{0}\ot u\ps{1}\ns{1}(u\ps{2}\rt v\ns{1}),\\  \label{mp5} &u\ps{2}\ns{0}\ot (u\ps{1}\rt
f)u\ps{2}\ns{1}=u\ps{1}\ns{0}\ot u\ps{1}\ns{1}(u\ps{2}\rt f).
\end{align}
One then forms  a new Hopf algebra $\Fc\acl \Uc$, called the  bicrossed product of the matched pair  $(\Fc , \Uc)$ ; it has $\Fc\cl \Uc$ as underlying
coalgebra, $\Fc\al \Uc$ as underlying algebra and the antipode is defined by
\begin{equation}\label{anti}
S(f\acl u)=(1\acl S(u\ns{0}))(S(fu\ns{1})\acl 1) , \qquad f \in \Fc , \, u \in \Uc.
\end{equation}

On the other hand, we need to recall the notion of double crossed product Hopf algebra \cite{Majid-book}.  Let $\Uc$ and $\Vc$ be two Hopf algebras such that $\Vc$ is a
right $\Uc-$module coalgebra and $\Uc$ is left  $\Vc-$module coalgebra. We call them mutual pair if   their actions  satisfy the following  conditions.
\begin{align}\label{mutual-1}
&v\rt(u^1u^2)= (v\ps{1}\rt u^1\ps{1})((v\ps{2}\lt u^1\ps{2})\rt  u^2),\quad  1\lt u=\ve(u),\\ \label{mutual-2}
&(v^1v^2)\lt u= (v^1\lt(v^2\ps{1}\rt u\ps{1}))(v^2\ps{2}\lt u\ps{2}),\quad
v\rt 1=\ve(v),\\\label{mutual-3}
 &\sum v\ps{1}\lt u\ps{1}\ot v\ps{2}\rt u\ps{2}=\sum v\ps{2}\lt u\ps{2}\ot v\ps{1}\rt u\ps{1}.
\end{align}
Having a  mutual   pair of Hopf algebras, one constructs the double crossed product Hopf algebra $\Uc\dcp \Vc$. As a coalgebra $\Uc\dcp\Vc$ is  $\Uc\ot\Vc $,
however its algebra structure is defined by the following rule together with $1\dcp 1$ as its unit.
\begin{equation}
(u^1\dcp v^1)(u^2\dcp v^2):= u^1(v^1\ps{1}\rt u^2\ps{1})\dcp (v^2\ps{2}\lt u^2\ps{2})v^2
\end{equation}
The antipode of $\Uc\dcp \Vc$ is defined by
\begin{equation}
S(u\dcp v)=(1\dcp S(v))(S(u)\dcp 1)= S(v\ps{1})\rt S(u\ps{1})\dcp S(v\ps{2})\lt S(u\ps{2}).
\end{equation}

\subsection{Lie-Hopf algebras and associated  bicrossed product Hopf algebras}
\label{SS-Lie-Hopf}
It is now clear that bicrossed product Hopf algebras play a crucial r\^ole in many places, especially when one tries to compute Hopf cyclic cohomology of nontrivial Hopf algebras \cite{Moscovici-Rangipour-07,Moscovici-Rangipour-09,Moscovici-Rangipour-011}. It is always a lengthy task to show that a particular Hopf algebra is of the form of bicrossed product. In this subsection, we introduce the minimum criteria by which one verifies a Hopf algebra is of this form.
Based on our interest, we restrict ourself  to the case that one of the building block Hopf algebra is commutative and the other one is enveloping algebra of a Lie algebra. It is clear that everything can be extended to more general situations but we do not try it here.

Let $\Fc$ be a commutative Hopf algebra on which a Lie algebra  $\Fg$ acts  by derivations. We endow the vector space $\Fg\ot \Fc$ with the following bracket.
\begin{equation}\label{bracket}
[X\ot f, Y\ot g]= [X,Y]\ot fg+ Y\ot \ve(f)X\rt g- X\ot \ve(g) Y\rt f.
\end{equation}
\begin{lemma}
Let $\Fg$ act on a commutative Hopf algebra $\Fc$ and $\ve(X\rt f)=0$ for any $X\in \Fg$ and $f\in \Fc$. Then  the bracket defined in \eqref{bracket} endows  $\Fg\ot \Fc$ with a  Lie algebra structure.
\end{lemma}
\begin{proof}
It is obvious that the bracket is antisymmetric.  We need to check the Jacobi identity. Indeed, after routine   computation we observe that,
\begin{align}
\begin{split}\label{jacobi-1}
&[[X\ot f, Y\ot g], Z\ot h]=[[X,Y],Z]\ot fgh+Z\ot \ve(fg)[X,Y]\rt h-\\
&[X,Y]\ot \ve(h)Z\rt (fg)+[Y,Z]\ot \ve(f)hX\rt g - Y\ot \ve(h)\ve(f)Z\rt (X\rt g)-\\
&[X,Z]\ot \ve(g)hY\rt f+X\ot \ve(h)\ve(g)Z\rt(Y\rt f).
\end{split}
\end{align}
\begin{align}
\begin{split}\label{jacobi-2}
&[[ Y\ot g, Z\ot h],X\ot f]= [[Y,Z],X]\ot fgh+ X\ot \ve(gh)[Y,Z]\rt f-\\
 &[Y,Z]\ot \ve(f) X\rt (gh)+[Z,X]\ot \ve(g)f Y\rt h- Z\ot \ve(f)\ve(g) X\rt(Y\rt h)-\\
&[Y,X]\ot \ve(h)fZ\rt g+ Y\ot \ve(f)\ve(h) X\rt(Z\rt g).
\end{split}
\end{align}

\begin{align}
\begin{split}\label{jacobi-3}
&[[Z\ot h,X\ot f],Y\ot g]=[[Z,X],Y]\ot fgh+ Y\ot \ve(hf)[Z,X]\rt g- \\
&[Z,X]\ot \ve(g)Y\rt(hf)+[X,Y]\ot \ve(h)g Z\rt f- X\ot \ve(h)\ve(g)Y\rt(Z\rt f)-\\
&[Z,Y]\ot \ve(f)g X\rt h+ Z\ot \ve(g)\ve(f) Y\rt(X\rt h).
\end{split}
\end{align}
Summing up \eqref{jacobi-1}, \eqref{jacobi-2}, and \eqref{jacobi-3} and using the fact that $\Fg$ is a Lie algebra  acting on $\Fc$ by derivations, we get
 \begin{equation}\label{jocobi}
 [[X\ot f, Y\ot g], Z\ot h]+[[ Y\ot g, Z\ot h],X\ot f]+ [[Z\ot h,X\ot f],Y\ot g]=0.
 \end{equation}
\end{proof}
Now let assume that $\Fc$ coacts on $\Fg$ from right via $\Db_\Fg:\Fg\ra \Fg\ot \Fc$. We define the first-order  matrix coefficients $ f^i_j\in \Fc$ of $\Db_\Fg$ by ,
\begin{equation}\label{def-Db-Fg}
\Db_\Fg(X_i)= \sum_i X_i\ot f_j^i.
\end{equation}
One use the fact that $\Db_\Fg$ is a coaction to observe
\begin{equation}\label{D-f-j-i}
\D(f_i^j)=\sum_{k=1}^n f_k^j\ot f_i^k.
\end{equation}
We define the second-order  matrix coefficients by
\begin{equation}
 f_{j,k}^i:=  X_{k}\rt f^{i}_j.
 \end{equation}
  Let $C^k_{i,j}$ stand for the structure constants of the Lie algebra $\Fg$, i.e,
  \begin{equation}
  [X_i,X_j]=\sum_kC_{i,j}^k X_k
  \end{equation}
\begin{definition}
We say that a coaction $\Db_\Fg:\Fg\ra \Fg\ot \Fc$ satisfies the structure identity of $\Fg$  if
\begin{equation}\label{Bianchi}
f_{j,i}^k-f_{i,j}^k=\sum_{s,r} C_{s,r}^k f_{i}^rf_{j}^s+\sum_lC_{i,j}^lf_l^k\,.
\end{equation}
\end{definition}
\begin{lemma}\label{Lemma-structure-identity}
The coaction  $\Db_\Fg:\Fg\ra \Fg\ot \Fc$ satisfies the structure identity of $\Fg$  if and only  if  $\Db_\Fg:\Fg\ra \Fg\ot \Fc$ is a Lie algebra map.
\end{lemma}
\begin{proof}
Indeed we  just need to check for two basis elements.
\begin{equation}
\Db_\Fg([X_i,X_j])=\Db_\Fg(C^k_{i,j}X_k)= C^k_{i,j}X_l\ot f^l_k.
\end{equation}
On the other hand, by using \eqref{Bianchi} we observe
\begin{align}
\begin{split}
&[\Db_\Fg(X_i), \Db_\Fg(X_j)]=[X_p\ot f^p_i, X_q\ot f^q_j]=\\
&[X_p,X_q]\ot f^p_if^q_j+ X_q\ot \ve(f^p_i)X_p\rt f^q_j- X_p\ot \ve(f^q_j)X_q\rt f^p_i=\\
& C^k_{r,s}X_k\ot f^r_if^s_j+  X_k\ot X_i\rt f^k_j- X_k\ot X_j\rt f^k_i=\\
&C^k_{r,s}X_k\ot f^r_if^s_j+ X_k\ot (f^k_{j,i}- f^k_{i,j})=\\
&C^k_{r,s}X_k\ot f^r_if^s_j+ C^k_{s,r}X_k\ot f^r_if^s_j+ C^k_{i,j}X_l\ot f^l_k=C^k_{i,j}X_l\ot f^l_k.
\end{split}
\end{align}
The converse argument is similar to the above.
\end{proof}
One uses the  action of $\Fg$ on $\Fc$ and the coaction of  $\Fc$  on $\Fg$ to define an action $\Fg$ on $\Fc\ot \Fc$ by
\begin{equation}
X\bullet (f^1 \ot f^2)= X\ns{0}\rt f^1\ot X\ns{1} f^2 + f^1\ot X\rt f^2.
\end{equation}
\begin{definition}\label{def-Lie-Hopf}
Let a Lie algebra $\Fg$ act on a commutative Hopf algebra $\Fc$ by derivations. We say that $\Fc$ is a $\Fg$-Hopf algebra if
 \begin{enumerate}
   \item $\Fc$ coacts on $\Fg$ and its coaction satisfies the structure identity of $\Fg$.
   \item $\D$ and $\ve$ are $\Fg$-linear, i.e,  $\D(X\rt f)=X\bullet\D(f)$, \quad $\ve(X\rt f)=0$, \quad $f\in \Fc$ and $X\in \Fg$.
    \end{enumerate}
\end{definition}
Let $\Fc$ be a  $\Fg$-Hopf algebra. Then  $U(\Fg)$ acts on $\Fc$ in the obvious way and makes it a $U(\Fg)$-module algebra. We extend the coaction $\Db_\Fg$ of $\Fc$ on $\Fg$ to a coaction $\Db$ of $\Fc$  on $U(\Fg)$ inductively via the rule \eqref{mp4} and $\Db(1)=1\ot 1$.
\begin{lemma}\label{U-coaction}
The extension of $\Db_\Fg:\Fg\ra \Fg\ot \Fc$  to $\Db:U(\Fg)\ra U(\Fg)\ot \Fc$  via \eqref{mp4} is well-defined.
\end{lemma}
\begin{proof}
 It is necessary and sufficient to prove that $\Db_\Fg([X,Y])=\Db(XY-YX)$. Using Lemma \ref{Lemma-structure-identity},    the rule \eqref{mp4}  and the fact that $\Fc$ is commutative, we see that
\begin{multline}
\Db(XY-YX)= [X\ns{0},Y\ns{0}]\ot X\ns{1}Y\ns{1}+ Y\ns{0}\ot X\rt Y\ns{1}-\\ X\ns{0}\ot Y\rt X\ns{1}=\Db_\Fg([X,Y]).
\end{multline}
\end{proof}
\begin{theorem}\label{Theorem-Lie-Hopf-matched-pair}
Let $\Fc$ be a commutative $\Fg$-Hopf algebra. Then via the coaction of $\Fc$ on $U(\Fg)$ defined above and the natural action of $U(\Fg)$ on $\Fc$, the pair $(U(\Fg), \Fc)$ becomes a matched pair of Hopf algebras. Conversely, for a commutative Hopf algebra $\Fc$, if $(U(\Fg), \Fc)$ is a matched pair of Hopf algebras then $\Fc$ is a $\Fg$-Hopf algebra .
\end{theorem}

\begin{proof}
We need to verify that the matched pair conditions are satisfied. The axioms \eqref{mp1} and \eqref{mp3} are held by the definition. We prove the other two.

By definition of the coaction $\Db:U(\Fg)\ra U(\Fg)\ot \Fc$, the axiom \eqref{mp4} is held for any $u,v\in U(\Fg)$.

 Next we check \eqref{mp2}. By Definition \ref{def-Lie-Hopf} (2) we observe that \eqref{mp2} is satisfied  for any $X\in \Fg$ and $f\in \Fc$. Let assume that  it  is held for $u,v\in U(\Fg)$, and any $f\in \Fc$. By using \eqref{mp4}  we see that
\begin{align}\label{proof-label-1}
\begin{split}
&(uv)\ps{1}\ns{0}\ot (uv)\ps{1}\ns{1}\ot (uv)\ps{2}=\\
&u\ps{1}\ns{0}v\ps{1}\ns{0}\ot u\ps{1}\ns{1}(u\ps{2}\rt v\ps{1}\ns{1})\ot u\ps{3}v\ps{2}.
\end{split}
\end{align}
Using \eqref{proof-label-1} and the fact that $\Fc$ is $U(\Fg)$-module algebra we prove our claim.
\begin{align}
\begin{split}
&\D(uv\rt f)= u\ps{1}\ns{0}\rt (v\rt f)\ps{1}\ot u\ps{1}\ns{1}(u\ps{2}\rt (v\rt f)\ps{2})=\\
&u\ps{1}\ns{0}\rt (v\ps{1}\ns{0}\rt f\ps{1})\ot u\ps{1}\ns{1}(u\ps{2}\rt (v\ps{1}\ns{1}(v\ps{2}\rt f\ps{2})=\\
&u\ps{1}\ns{0}\rt (v\ps{1}\ns{0}\rt f\ps{1})\ot u\ps{1}\ns{1}(u\ps{2}\rt (v\ps{1}\ns{1})(u\ps{3}v\ps{2}\rt f\ps{2})=\\
&(uv)\ps{1}\ns{0}\rt f\ps{1}\ot (uv)\ps{1}\ns{1}((uv)\ps{2}\rt f\ps{2}).
\end{split}
\end{align}
Finally we check that $U(\Fg)$ is  a $\Fc$-comodule coalgebra i.e, we verify \eqref{cc1} which is obvious for any $X\in \Fg$. Let assume that \eqref{cc1} is satisfied for $u,v\in U(\Fg)$, and prove it is so for $uv$. Indeed, by using \eqref{mp4} and the fact that $U(\Fg)$ is cocommutative, $\Fc$ is commutative, and $\Fc$ is $U(\Fg)$-module algebra, we observe that
\begin{align}
\begin{split}
&(uv)\ps{1}\ns{0}\ot (uv)\ps{2}\ns{0}\ot (uv)\ps{1}\ns{1}(uv)\ps{2}\ns{1}=\\
&(u\ps{1}v\ps{1})\ns{0}\ot (u\ps{2}v\ps{2})\ns{0}\ot (u\ps{1}v\ps{1})\ns{1}(u\ps{2}v\ps{2})\ns{1}=\\
&u\ps{1}\ns{0}v\ps{1}\ns{0}\ot u\ps{3}\ns{0}v\ps{2}\ns{0}\ot u\ps{1}\ns{1}(u\ps{2}\rt v\ps{1}\ns{1})u\ps{3}\ns{1}(u\ps{4}\rt v\ps{2}\ns{1})=\\
&u\ps{1}\ns{0}v\ps{1}\ns{0}\ot u\ps{2}\ns{0}v\ps{2}\ns{0}\ot u\ps{1}\ns{1} u\ps{2}\ns{1}(u\ps{3}\rt (v\ps{1}\ns{1}v\ps{2}\ns{1})=\\
&u\ps{1}\ns{0}\ps{1}v\ns{0}\ps{1}\ot u\ps{1}\ns{0}\ps{2}v\ns{0}\ps{2}\ot u\ps{1}\ns{1} (u\ps{2}\rt (v\ns{1})=\\
&(u\ps{1}\ns{0}v\ns{0})\ps{1}\ot (u\ps{1}\ns{0}v\ns{0})\ps{2}\ot u\ps{1}\ns{1}(u\ps{2}\rt v\ns{1})=\\
&(uv)\ns{0}\ps{1}\ot (uv)\ns{0}\ps{2}\ot (uv)\ns{1}.
\end{split}
\end{align}
Conversely, let $\Fc$ be a commutative Hopf algebra and $(U(\Fg), \Fc)$ be a matched pair of Hopf algebras. Let us denote the coaction of  $\Fc$ on $U(\Fg)$ by $\Db_{\Uc}$. First we prove that the restriction of $\Db_\Uc:U(\Fg)\ra U(\Fg)\ot \Fc$ on $\Fg$ lands in $\Fg\ot \Fc$. Indeed, since $U(\Fg)$ is $\Fc$-comodule coalgebra, we see that
\begin{align}
\begin{split}
&X\ns{0}\ot 1\ot X\ns{1}+ 1\ot X\ns{0}\ot X\ns{1}=\\
 &X\ps{1}\ns{0}\ot X\ps{2}\ns{0} \ot X\ps{1}\ns{1}X\ps{2}\ns{1}=\\
&X\ns{0}\ps{1}\ot X\ns{0}\ps{2}\ot X\ns{1}.
\end{split}
\end{align}
This shows that for any $X\in\Fg$, $\Db_\Uc(X)$ belongs to $P\ot \Fc$, where $P$ is the Lie algebra of primitive elements of $U(\Fg)$. But we know that $P=\Fg$ by \cite{Hoch}. So we get a coaction $\Db_\Fg:\Fg\ra \Fg\ot \Fc$ which is the restriction of $\Db_\Uc$. Since $\Fc$ is  a $U(\Fg)$-module algebra, $\Fg$ acts on $\Fc$ by derivations. An argument like the one in the proof of Lemma \ref{U-coaction} shows that the coaction $\Db_g:\Fg\ra \Fg\ot \Fc$ is a coalgebra map. Finally \eqref{mp2} implies that $\D$ is $\Fg$-equivariant. So we have proved that $\Fc$ is a $\Fg$-Hopf algebra.
\end{proof}

\subsection{Matched pair of  Hopf algebras associated to matched pair of  Lie algebras}
\label{SS-Lie algebra}
In this subsection we associate  a bi-crossed product Hopf algebra to any matched pair of Lie algebras $(\Fg_1,\Fg_2)$.
Let us recall the notion of matched pair of Lie algebras from \cite{Majid-book}.
 A pair of Lie algebras  $(\Fg_1, \Fg_2)$  is called a matched pair  if there
 are linear maps
\begin{equation}\label{g-1-g-2 action}
\alpha: \Fg_2\ot \Fg_1\ra \Fg_2, \quad \alpha_X(\z)=\z\lt X, \quad \beta:\Fg_2\ot \Fg_1\ra \Fg_1, \quad \beta_\z(X)=\z\rt X,
\end{equation}
satisfying the following conditions,
\begin{align}\label{mp-L-1}
&[\z,\x]\rt X=\z\rt(\x\rt X)-\x\rt(\z\rt X),\\\label{mp-L-2}
 & \z\lt[X, Y]=(\z\lt X)\lt Y-(\z\lt Y)\lt X, \\\label{mp-L-3}
 &\z\rt[X, Y]=[\z\rt X, Y]+[X,\z\rt Y] +
(\z\lt X)\rt Y-(\z\lt Y)\rt X\\\label{mp-L-4}
&[\z,\x]\lt X=[\z\lt X,\x]+[\z,\x\lt X]+ \z\lt(\x\rt X)-\x\lt(\z\rt X).
\end{align}

Given a matched pair of  Lie algebras $(\Fg_1,\Fg_2)$, one defines a double crossed sum Lie algebra
 $\Fg_1\bowtie \Fg_2$. Its underlying vector space is $\Fg_1\oplus\Fg_2$ and its Lie bracket
is defined  by:
\begin{equation}
[X\oplus\z,  Z\oplus\x]=([X, Z]+\z\rt Z-\x\rt X)\oplus ([\z,\x]+\z\lt Z-\x\lt X).
\end{equation}
 One checks that both $\Fg_1$ and $\Fg_2$ are Lie subalgebras of $\Fg_1\bowtie\Fg_2$ via obvious inclusions. Conversely,  if for a
Lie algebra $\Fg$ there are  two Lie subalgebras $\Fg_1$ and $\Fg_2$ so that $\Fg=\Fg_1\oplus\Fg_2$ as vector  spaces,  then
 $(\Fg_1, \Fg_2)$  forms a matched pair of Lie algebras and $\Fg\cong \Fg_1\bowtie \Fg_2$ as Lie algebras \cite{Majid-book}.
  In this case the actions of $\Fg_1$ on $\Fg_2$ and $\Fg_2$ on $\Fg_1$  for $\z\in \Fg_2$ and $X\in\Fg_1$ are uniquely determined  by
\begin{equation}\label{lie-actions}
[\z,X]=\z\rt X+\z\lt X.
\end{equation}

\begin{proposition}[\cite{Majid-book}] Let $\Fg=\Fg_1\bowtie\Fg_2$ be a double crossed sum of Lie algebras. Then the
enveloping algebras  $(U(\Fg_1),U(\Fg_2))$  becomes a mutual pair of Hopf algebras. Moreover, $U(\Fg)$ and $ U(\Fg_1)\dcp U(\Fg_2)$ are isomorphic as Hopf algebras.
\end{proposition}

In terms of the inclusions
\begin{align}
i_1:U(\Fg_1) \to U(\Fg_1 \dcp \Fg_2) \quad \mbox{ and } \quad i_2:U(\Fg_2) \to U(\Fg_1 \dcp \Fg_2)
\end{align}
the isomorphism mentioned in the above proposition is
\begin{align}
\mu\circ (i_1 \ot i_2):U(\Fg_1) \dcp U(\Fg_2) \to U(\Fg)
\end{align}
Here $\mu$ is the multiplication on $U(\Fg)$. One easily observes that there is a linear map
\begin{align}
\Psi:U(\Fg_2) \dcp U(\Fg_1) \to U(\Fg_1) \dcp U(\Fg_2),
\end{align}
satisfying
\begin{align} \label{calculationofPsi}
\mu \circ (i_2 \ot i_1) = \mu \circ (i_1 \ot i_2) \circ \Psi\,.
\end{align}
 The mutual actions of $U(\Fg_1)$ and $U(\Fg_2)$ are defined as follows
\begin{align} \label{actionstoeachother}
\rhd := (id_{U(\Fg_2)} \ot \ve) \circ \Psi \quad \mbox{ and } \quad \lhd := (\ve \ot id_{U(\Fg_1)}) \circ \Psi\,.
\end{align}
We now recall  the definition of   $R(\Fg)$, the  Hopf algebra of representative functions on $U(\Fg)$, for a Lie algebra $\Fg$.
\begin{equation}\notag
R(\Fg)=\{f\in \Hom(U(\Fg),\Cb)\mid \exists \text{   a   finite codimensional ideal}  \; I\subseteq \ker f   \}
\end{equation}
 The finite codimensionality condition  in the definition of $R(\Fg)$ guarantees that  for any $f\in R(\Fg)$  there exist  a  finite number of  functions $f_i',
 f_i''\in R(\Fg)$   such that  for any $u^1,u^2\in U(\Fg)$.
\begin{equation}
f(u^1u^2)=\sum_{i}f_i'(u^1)f_i''(u^2).
\end{equation}
The Hopf algebraic structure of $R(\Fg)$ is summarized by:
\begin{align}
&\mu: R(\Fg)\ot R(\Fg)\ra R(\Fg), && \mu(f\ot g)(u)=f(u\ps{1})g(u\ps{2}),\\ &\eta:\Cb\ra R(\Fg),&& \eta(1)=\ve,\\ &\D:R(\Fg)\ra R(\Fg)\ot R(\Fg),&&\\\notag & \D(f)=\sum_if_i'\ot
f_i'',&& \text{if}\;\; f(u^1u^2)= \sum f_i'(u^1)f_i''(u^2),\\
& S: R(\Fg)\ra R(\Fg), &&S(f)(u)=f(S(u)).
\end{align}

Let $<X_1,\dots, X_n>$  be a basis for $\Fg_1$ and  $<\t^1,\dots,\t^n>$ be its dual basis for $\Fg_1^\ast$, that is $\langle\t^i,X_j\rangle=\d^i_j$, where $\d^i_j$ are Kronecker's delta function. We define the following liner functionals  on $U(\Fg_2)$,
\begin{equation}
f_i^j(v)=<v\rt X_i,\t^j>,
\end{equation}
 equivalently
\begin{equation}
v\rt X_i=f_i^j(v) X_j.
\end{equation}

\begin{lemma}\label{f^i_j}
For any $1\le i,j\le n$, the functions  $f_i^j$ are representative.
\end{lemma}
\begin{proof}
 For $1\le i\le n$, we set $I_i=\{v\in U(\Fg_2)\mid v\rt X_i=0\}$. Since $\Fg_1$ is finite dimensional, $I_i$ is a finite codimensional left ideal of $U(\Fg_2)$ sitting in
 the null space of $f_i^j$ for any $j$.
\end{proof}
As a result, we have the following coaction admitting $f_i^j$'s as the first order matrix coefficients, c.f. \eqref{def-Db-Fg}.
\begin{align}\label{g-coaction}
\Db_{\rm Alg}:\Fg_1\ra \Fg_1\ot R(\Fg_2),&&
\Db_{\rm Alg}(X_i)=\sum_{k=1}^nX_j\ot f_i^j,
\end{align}

By the work of  Harish-Chandra  \cite{HC} we know that $R(\Fg_2)$ separates elements of $U(\Fg_2)$. This results  a nondegenerate Hopf pairing  between  $R(\Fg_2)$
and $U(\Fg_2)$.
\begin{equation}\label{F-V-pairing}
<f,v>:=f(v).
\end{equation}
We use the pairing \eqref{F-V-pairing} to define the natural action of $U(\Fg_1)$ on $R(\Fg_2)$,
\begin{equation}\label{actio-U-Uo}
U(\Fg_1)\ot R(\Fg_2) \ra R(\Fg_2), \quad <u\rt f, v>=<f, v\lt u>.
\end{equation}

         We define the new elements of $R(\Fg_2)$ by
 \begin{equation}
 f_{i_1,\dots,i_k}^j:= X_{i_k}\cdots X_{i_2}\rt f_{i_1}^j.
 \end{equation}

 \begin{lemma}\label{Lemma-structure-identity-g-1}
The coaction $\Db_{\rm Alg}$ satisfy the structure identity of $\Fg_1$.
\end{lemma}

\begin{proof}
We should prove that
\begin{equation}
f_{j,i}^k-f_{i,j}^k=\sum_{s,r} C_{s,r}^k f_{i}^rf_{j}^s+\sum_lC_{i,j}^lf_l^k,
\end{equation}
where $C^i_{j,k}$'s are the structure constants of $\Fg_1$.

We first use \eqref{mutual-1} to observe that
\begin{align}
\begin{split}
v\rt[X,Y]=[v\ps{1}\rt X,v\ps{2}\rt Y]+ (v\lt X)\rt Y-(v\lt Y)\rt X.
\end{split}
\end{align}

We apply two hand sides of \eqref{Bianchi} on an arbitrary element of $v\in U(\Fg_2)$ and use the above observation to finish the proof.
\begin{align}
\begin{split}
& C^l_{i,j}f^k_l(v)X_k= C^l_{i,j}v\rt X_l =v\rt[X_i, X_j]=\\
&[v\ps{1}\rt X_i, v\ps{2}\rt X_j]+ (v\lt X_i)\rt X_j-(v\lt X_j)\rt X_i=\\
& f^r_i(v\ps{1})f^s_j(v\ps{2})[X_r,X_s]+  f^k_j(v\lt X_i)X_k-f^k_i(v\lt X_j)X_k=\\
&C^k_{r,s} (f^r_if^s_j)(v)X_k+f^k_{j,i}(v)X_k- f^k_{i,j}(v)X_k.
\end{split}
\end{align}
\end{proof}

\begin{proposition}\label{Proposition-matched-Lie-Hopf-Lie}
For any matched pair of Lie algebras $(\Fg_1,\Fg_2)$, the Hopf algebra $R(\Fg_2)$ is a $\Fg_1$-Hopf algebra.
\end{proposition}
\begin{proof}
Lemma \ref{Lemma-structure-identity-g-1} proves that the coaction $\Db_{\rm Alg}$ satisfies the structure identity of $\Fg_1$. So, by using Theorem \ref{Theorem-Lie-Hopf-matched-pair} it suffices to prove that $\ve(X\rt f)=0$ and $\D(X\rt f)=X\bullet \D(f)$. We observe that $\ve(X\rt f)= (X\rt f)(1)= f(1\lt X)=0$.
Then it remains to show that $\D(X\rt f)=X\bullet\D(f)$. Indeed,
\begin{align}
\begin{split}
&\D(X\rt f)(v^1\ot v^2)= X\rt f(v^1v^2)= f(v^1v^2\lt X)= \\
&f(v^1\lt(v^2\ps{1}\rt X)v^2\ps{2})+ f(v^1(v^2\rt X))=\\
&f\ps{1}(v^1\lt(X\ns{1})(v^2\ps{1}) X\ns{0}) f\ps{2}(v^2\ps{2}) + f\ps{1}(v^1) f\ps{2}(v^2\lt X)=\\
&(X\ns{0}\rt f\ps{1})(v^1)X\ns{1}(v^2\ps{1}) f\ps{2}(v^2\ps{2}) + f\ps{1}(v^1) (X\rt f\ps{2})(v^2)=\\
&(X\bullet \D(f))(v^1\ot v^2).
\end{split}
\end{align}
\end{proof}
We summarize   the main result  of this section as follows.
\begin{theorem}\label{U-V-o-Theorem}
Let $(\Fg_1, \Fg_2)$ be a  matched pair Lie algebras. Then,   via  the canonical action and coaction defined in  \eqref{actio-U-Uo} and   \eqref{g-coaction} respectively,  $(U(\Fg_1), R(\Fg_2))$ is a matched pair of Hopf algebras.
\end{theorem}
\begin{proof}
By Proposition \ref{Proposition-matched-Lie-Hopf-Lie},  $R(\Fg_2)$ is a $\Fg_1$-Hopf algebra. One  then applies Theorem \ref{Theorem-Lie-Hopf-matched-pair}.
\end{proof}
As a result of Theorem \ref{U-V-o-Theorem}, to any  matched pair of Lie algebras $(\Fg_1,\Fg_2)$, we canonically  associate the  Hopf algebra
\begin{equation}
\Hc(\Fg_1,\Fg_2):=R(\Fg_2)\acl U(\Fg_1).
\end{equation}

\subsection{Matched pair of  Hopf algebras associated to matched pair of Lie groups}
\label{SS-Lie group}
In this subsection, our aim is to associate a bicrossed product Hopf algebra to any matched pair of Lie groups. We first recall the notion of matched pair of Lie groups.
Let $(G_1,G_2)$ be a pair of Lie groups with mutual smooth actions $\rhd: G_2 \times G_1 \to G_1$ and $\lhd: G_2 \times G_1 \to G_2$. Then $(G_1,G_2)$ is called a matched pair of Lie groups provided for any $\varphi, \varphi_1, \varphi_2 \in G_1$ and any $\psi, \psi_1, \psi_2 \in G_2$ the following compatibilities are satisfied
\begin{align}\label{matched-pair-general}
\begin{split}
&\psi \rhd \varphi_1\varphi_2 = (\psi \rhd \varphi_1)((\psi \lhd \varphi_1) \rhd \varphi_2), \qquad \psi \rhd e = e\\
&\psi_1\psi_2 \lhd \varphi = (\psi_1 \lhd (\psi_2 \rhd \varphi))(\psi_2 \lhd \varphi), \qquad e \lhd \varphi = e
\end{split}
\end{align}
Let  $\Fg_1$ and $\Fg_2$ be the Lie algebras of $G_1$ and $G_2$ respectively. One defines the action of $\Fg_1$ on $\Fg_2$  by taking derivative of the action of $G_1$ on $G_2$. Similarly, one defines the corresponding action of $\Fg_2$ on $\Fg_1$ and shows with a routine argument that if $(G_1,G_2)$ is a matched pair of Lie groups, then $(\Fg_1,\Fg_2)$ is a matched pair of Lie algebras.

Let $G$ be a Lie group and $\rho: G \to GL(V)$ a finite dimensional smooth representation. Then, the composition $\pi \circ \rho$ is called a representative function of $G$, where   $\pi \in End(V)^*$ is a linear functional.

We have the following characterization for the representative function due to \cite{Hochschild-Mostow-61}. For a smooth function  $f:G \to \mathbb{C}$ the following are equivalent:
1) $f$ is a representative function. 2) The right translates $\{f \cdot \psi | \psi \in G\}$ span a finite dimensional vector space over $\mathbb{C}$. 3) The left translates $\{\psi \cdot f | \psi \in G\}$ span a finite dimensional vector space over $\mathbb{C}$.
 Here the left and the right translation actions of $G_2$ on $R(G_2)$ are defined  by
\begin{equation}\label{auxiliary-action}
\psi \cdot f := f\ps{1}f\ps{2}(\psi), \quad f \cdot \psi := f\ps{1}(\psi)f\ps{2}.
\end{equation}

It is well-known that the representative functions form a commutative Hopf algebra via
\begin{align}
&(f^1f^2)(\psi)=f^1(\psi)f^2(\psi), \quad 1_{R(G)}(\psi)=1,\\
&\D(f)(\psi_1,\psi_2)=f(\psi_1\psi_2), \quad \ve(f)=f(e),\\
&S(f)(\psi)=f(\psi^{-1}).
\end{align}
 Our main objective now is  to prove that $R(G_2)$ is a $\Fg_1-$Hopf algebra.

  We first define a right coaction of $R(G_2)$ on $\Fg_1$. To do so, we differentiate the right action of $G_2$ on $G_1$ to get a right action of $G_2$ on $\Fg_1$. We then  introduce the functions $f_i^j:G_2 \to \mathbb{C}$.
\begin{align}\label{psi-act-X-j}
\psi \rhd X_i = X_jf_i^j(\psi), \quad X\in \Fg_1, \psi\in G_2.
\end{align}

\begin{lemma}\label{Lemma f-i-j-in-R(G)}
The functions $f_i^j:G_2 \to \mathbb{C}$ defined above are representative functions.
\end{lemma}
\begin{proof}
For $\psi_1,\psi_2 \in G_2$, we observe
\begin{align}
\begin{split}
& X_jf_i^j(\psi_1\psi_2) = \psi_1\psi_2 \rhd X_i = \psi_1 \rhd (\psi_2 \rhd X_i) = \\
 & \psi_1 \rhd X_lf_i^l(\psi_2) = X_jf_l^j(\psi_1)f_i^l(\psi_2),
 \end{split}
\end{align}
Therefore,
\begin{align}
\psi_2 \cdot f_i^j = f_i^l(\psi_2)f_l^j
\end{align}
In other words, $\psi_2 \cdot f_i^j \in span\{f_i^j\}$ for any $\psi_2 \in G_2$.
\end{proof}

As a direct consequence of Lemma \ref{Lemma f-i-j-in-R(G)}, the equation \eqref{psi-act-X-j} defines a coaction:
\begin{proposition}
The map $\Db_{\rm Gr}: \Fg_1 \to \Fg_1 \otimes R(G_2)$ defined by
\begin{align}\label{coactionforliegroup}
\Db_{\rm Gr}(X_i):= X_j \otimes f_i^j
\end{align}
is a right coaction of $R(G_2)$ on $\Fg_1$.
\end{proposition}

 Let us recall the natural left action of $G_1$ on $C^{\infty}(G_2)$ defined by $\vp\rt f(\psi):= f(\psi\lt \vp)$,  and define the derivative of this action by
\begin{align}\label{actionforliegroup}
X \rhd f := \dt  {\rm exp}(tX) \rhd f, \quad X\in \Fg_1, f\in R(G_2).
\end{align}
  In fact, considering $R(G_2) \subseteq C^{\infty}(G_2)$, this is nothing but $d_e \rho(X)|_{R(G_2)}$, derivative of the representation $\rho: G_1 \to GL(C^{\infty}(G_2))$ at identity.

\begin{lemma}
For any $X \in \Fg_1$ and any $f \in R(G_2)$, we have $X \rhd f \in R(G_2)$. Moreover we have
\begin{equation}\label{Db-equivariancy}
\D(X\rt f)= X\ns{0}\rt f\ps{1}\ot X\ns{1}f\ps{2}+ f\ps{1}\ot X\rt f\ps{2}.
\end{equation}
\end{lemma}

\begin{proof}
For any $\psi_1,\psi_2 \in G_2$, by using  the fact that ${(\rt )\circ \rm exp }={\rm exp}\circ {d_e\rt}$, we observe
\begin{align}
\begin{split}
& (X \rhd f)(\psi_1\psi_2) = \dt  f(\psi_1\psi_2 \lhd {\rm exp}(tX)) = \\
& \dt  f((\psi_1 \lhd (\psi_2 \rhd {\rm exp}(tX)))(\psi_2 \lhd {\rm exp}(tX))) = \\
 & \dt  f((\psi_1 \lhd (\psi_2 \rhd {\rm exp}(tX))\psi_2) + \dt f(\psi_1(\psi_2 \lhd {\rm exp}(tX))) = \\
 & \dt  (\psi_2 \cdot f)(\psi_1 \lhd (\psi_2 \rhd {\rm exp}(tX))) + \dt  (f \cdot \psi_1)(\psi_2 \lhd {\rm exp}(tX))=\\
 & f\ps{2}(\psi_2)((\psi_2 \rhd X) \rhd f\ps{1})(\psi_1) + f\ps{1}(\psi_1)(X \rhd f\ps{2})(\psi_2).
 \end{split}
\end{align}
Which shows
\begin{align}
\psi_2 \cdot (X \rhd f) = f\ps{2}(\psi_2)(\psi_2 \rhd X) \rhd f\ps{1} + (X \rhd f\ps{2})(\psi_2)f\ps{1}.
\end{align}
We conclude that $\psi_2 \cdot (X \rhd f) \in span\{X\ns{0} \rhd f\ps{1}, f\ps{1}\},$
that is, left translates of $X \rhd f$ span a finite dimensional vector space. $\Fg_1$-linearity  of $\D$ is shown by
\begin{align*}
& f\ps{2}(\psi_2)((\psi_2 \rhd X) \rhd f\ps{1})(\psi_1) + f\ps{1}(\psi_1)(X \rhd f\ps{2})(\psi_2)=\\
 &(X\ns{0}\rt f\ps{1})(\psi_1) X\ns{1}(\psi_2)f\ps{2}(\psi_2)+ f\ps{1}(\psi_1)(X \rhd f\ps{2})(\psi_2).
 \end{align*}
\end{proof}

\begin{proposition}
The map $\Fg_1 \otimes R(G_2) \to R(G_2)$ is a left action.
\end{proposition}

\begin{proof}
Using that fact that $\Ad\; \circ\; {\rm exp}= {\rm exp}\; \circ\; ad$, we prove the compatibility of the action and the bracket,
\begin{align}
\begin{split}
& [X,Y] \rhd f = \dt  {\rm exp}([tX,Y]) \rhd f = \\
 & \dt  Ad_{{\rm exp}(tX)}(Y) \rhd f = \\
 & \dt \ds  {\rm exp}(tX){\rm exp}(sY){\rm exp}(-tX) \rhd f = \\
 & \dt \ds {\rm exp}(tX){\rm exp}(sY) \rhd f - \dt \ds {\rm exp}(sY){\rm exp}(tX) \rhd f = \\
 & X \rhd (Y \rhd f) - Y \rhd (X \rhd f).
 \end{split}
\end{align}
\end{proof}

\begin{lemma}\label{Lemma-Lie-group-Bianchi}
The coaction $\Db_{Gr}:\Fg_1 \to \Fg_1 \ot R(G_2)$ satisfies the structure identity of $\Fg_1$.
\end{lemma}

\begin{proof}
We will prove that
\begin{align}
f_{j,i}^k - f_{i,j}^k = \sum_{r,s}C_{s,r}^kf_i^rf_j^s + \sum_lC_{i,j}^lf_l^k\,.
\end{align}
Realizing the elements of $\Fg_1$ as local derivations  on $C^{\infty}(G_1)$,
\begin{align}
\begin{split}
& (\psi \rhd [X_i,X_j])(\hat{f}) = \dt  \hat{f}( {\rm exp}(t(\psi \rhd [X_i,X_j]))) = \\
 & \dt  \hat{f}(\psi \rhd {\rm exp}(ad(tX_i)(X_j)) =\\
 &  [\psi \rhd X_i, \psi \rhd X_j](\hat{f}) + \dt  ((\psi \lhd {\rm exp}(tX_i)) \rhd X_j)(\hat{f}) - \\
 &  \ds  ((\psi \lhd {\rm exp}(sX_j)) \rhd X_i)(\hat{f}),
\end{split}
\end{align}
for any $\psi \in G_2$, $\hat{f} \in C^{\infty}(G_1)$ and $X_i,X_j \in \Fg_1$.
  Hence we conclude that
\begin{align}
\begin{split}
& \psi \rhd [X_i,X_j] = [\psi \rhd X_i, \psi \rhd X_j] + \dt  ((\psi \lhd {\rm exp}(tX_i)) \rhd X_j) - \\
& \ds  ((\psi \lhd {\rm exp}(sX_j)) \rhd X_i).
\end{split}
\end{align}
Now, writing $[X_i,X_j] = C_{i,j}^lX_l$ and recalling that by the definition of the coaction we have $\psi \rhd [X_i,X_j] = X_kC_{i,j}^lf_l^k(\psi).$
Similarly,
\begin{equation}
 [\psi \rhd X_i, \psi \rhd X_j] = [X_rf_i^r(\psi), X_sf_j^s(\psi)] =  C_{r,s}^kX_kf_i^r(\psi)f_j^s(\psi) = X_kC_{r,s}^k(f_j^sf_i^r)(\psi).
\end{equation}
Finally
\begin{align}
\begin{split}
& \dt  ((\psi \lhd {\rm exp}(tX_i)) \rhd X_j) = \dt  X_kf_j^k(\psi \lhd {\rm exp}(tX_i)) = \\
 & X_k (X_i \rhd f_j^k)(\psi) = X_kf_{j,i}^k(\psi),
 \end{split}
\end{align}
and in the same way
\begin{align}
\ds ((\psi \lhd {\rm exp}(sX_j)) \rhd X_i) = X_kf_{i,j}^k(\psi).
\end{align}
\end{proof}

\begin{theorem}
Let $(G_1, G_2)$ be a matched pair of Lie groups. Then by the action (\ref{actionforliegroup}) and the coaction (\ref{coactionforliegroup}), the pair  $(U (\Fg_1), R(G_2))$ is a matched pair of Hopf algebras.
\end{theorem}
\begin{proof}

According to Theorem \ref{Theorem-Lie-Hopf-matched-pair} we need to prove that  the Hopf algebra $R(G_2)$ is a $\Fg_1-$Hopf algebra.
Considering the Hopf algebra structure of $R(G_2)$, we see that,
\begin{equation}
 \varepsilon(X \rhd f) = (X \rhd f)(e) = \dt  f(e \lhd {\rm exp}(tX)) =  \dt  f(e) = 0.
\end{equation}
By the  Lemma \ref{Lemma-Lie-group-Bianchi}, we know that $\Db_{\rm Gr}$  satisfies the structure identity of $\Fg_1$. The equation \eqref{Db-equivariancy} proves that  $\Delta(X \rhd f) = X \bullet \Delta(f)$.
\end{proof}
 As a result, to any matched pair of Lie groups $(G_1, G_2)$, we associate the Hopf algebra
\begin{align}
\mathcal{H} (G_1, G_2) := R(G_2) \acl U (\Fg_1).
\end{align}

We proceed by providing the relation between the Hopf algebras $\mathcal{H} (G_1, G_2)$ and $\mathcal{H} (\Fg_1, \Fg_2)$. To this end, we first introduce a map
\begin{align}\label{map-R(G)->R(g)}
\theta: R(G_2) \to R(\Fg_2), \qquad \pi \circ \rho \to \pi \circ d_e\rho
\end{align}
for any finite dimensional representation $\rho: G_2 \to GL(V)$, and linear functional $\pi: End(V) \to \mathbb{C}$. Here we identify the representation $d_e\rho: \Fg_2 \to \Fg\Fl(V)$ of $\Fg_2$ and the unique algebra map $ d_e\rho:U(\Fg_2) \to \Fg\Fl(V)$ making the following diagram commutative
$$
\xymatrix {
\ar[d]_i \Fg_2 \ar[r]^{d_e\rho} & \Fg\Fl(V) \\
U(\Fg_2)\ar[ur]^{ d_e\rho}.
}
$$
Let $G$ and $H$ be two Lie groups, where $G$ is  simply connected. Let also $\Fg$ and $\Fh$ be the corresponding Lie algebras respectively. Then, a linear map $\sigma: \Fg \to \Fh$ is the differential of a map $\rho:G \to H$ of Lie groups if and only if it is a map of Lie algebras \cite{FultHarr}. Therefore, in case of $G_2$ to be  simply connected, the map $\theta: R(G_2) \to R(\Fg_2)$ is bijective.

We can express $\theta: R(G_2) \to R(\Fg_2)$ explicitly. The map $d_e\rho:U(\Fg_2) \to \Fg\Fl(V)$ sends $1 \in U(\Fg_2)$ to $Id_V \in \Fg\Fl(V)$, hence for $f \in R(G_2)$
\begin{align}
\theta(f)(1) = f(e)
\end{align}
and since it is multiplicative, for any $\xi_1,...,\xi_n \in \Fg_2$
\begin{align}
\theta(f)(\xi^1...\xi^n) = \frac{d}{dt_1}\big|_{t_1=0}...\frac{d}{dt_n}\big|_{t_n=0}f({\rm exp}(t_1\xi_1)...{\rm exp}(t_n\xi_n))
\end{align}

\begin{proposition}
The following is a map of Hopf algebras
\begin{align}
\Theta: \mathcal{H}(G_1, G_2) \to \mathcal{H}(\Fg_1, \Fg_2), \qquad \Theta(f \acl u)=\theta(f) \acl u.
\end{align}
Moreover, $\mathcal{H}(G_1, G_2) \cong \mathcal{H}(\Fg_1, \Fg_2)$ provided $G_2$ is simply connected.
\end{proposition}
\begin{proof}

First we show that $\Theta$ is an algebra map. To this end, we need to prove that $\theta$ is a map of $U(\Fg_1)$-module algebras. It is easy to observe that $\theta: R(G_2) \to R(\Fg_2)$ is a map of Hopf algebras. However, we prove that   $\theta$ is a $U(\Fg_1)-$module map.
Indeed for any $X \in \Fg_1$ and any $\xi \in \Fg_2$,
\begin{align}\label{theta-is-g-linear}
\begin{split}
& \theta(X \rhd f)(\xi) = \frac{d}{dt}\big|_{t=0}\frac{d}{ds}\big|_{s=0}f({\rm exp}(s\xi) \lhd {\rm exp}(tX)) = \\
 & \frac{d}{dt}\big|_{t=0}f({\rm exp}(t(\xi \lhd X))) = X \rhd \theta(f)(\xi).
 \end{split}
\end{align}
Next we prove  that the following diagram is commutative
\begin{equation}\label{diagram-R(G)-R(g)}
\xymatrix{  \ar[rrd]_{\Db_{\rm alg}}\Fg_1\ar[rr]^{\Db_{\rm Gr}}&& \Fg_1\ot R(G_2)\ar[d]^{\theta}\\
&&    \Fg_1\ot R(\Fg_2)  }
\end{equation}
Indeed, by valuating on $\xi \in \Fg_2$ we have
\begin{equation}
 X^j\theta(g^i_j)(\xi) = \frac{d}{dt}\big|_{t=0}X^jg^i_j({\rm exp}(t\xi)) = \frac{d}{dt}\big|_{t=0}({\rm exp}(t\xi) \rhd X^i) =  \xi \rhd X^i =  X^jf^i_j(\xi).
\end{equation}
This shows that $\Theta$ is a map of coalgebras and hence the proof is complete.
\end{proof}

\subsection{Matched pair of  Hopf algebras associated to matched pair of affine algebraic groups}
\label{SS-algebraic group}
In this subsection we aim to associate a bicrossed product Hopf algebra to any matched pair of affine algebraic groups. Let $G_1$ and $ G_2$ be two affine algebraic groups. We assume that there  are  maps of affine algebraic sets
\begin{align}
\rhd: G_2 \times G_1 \to G_1\qquad \text{and}\qquad\lhd: G_2 \times G_1 \to G_2,
\end{align}
 which means the existence of  the following maps \cite[Chap 22]{TauvYu}
\begin{align}
\begin{split}
& \Pc(\rhd):\Pc(G_1) \to \Pc(G_2 \times G_1) = \Pc(G_2) \ot \Pc(G_1) \\
& f \mapsto f^{\ns{-1}} \ot f^{\ns{0}} \quad \mbox{ such that } \quad f^{\ns{-1}}(\psi) f^{\ns{0}}(\varphi) = f(\psi \rhd \varphi)
\end{split}
\end{align}
and
\begin{align}
\begin{split}
& \Pc(\lhd):\Pc(G_2) \to \Pc(G_2 \times G_1) = \Pc(G_2) \ot \Pc(G_1) \\
& f \mapsto f^{\ns{0}} \ot f^{\ns{1}} \quad \mbox{ such that } \quad f^{\ns{0}}(\psi) f^{\ns{1}}(\varphi) = f(\psi \lhd \varphi)
\end{split}
\end{align}

We say $(G_1,G_2)$ is a matched pair if they satisfy (\ref{matched-pair-general}).

Then we  define mutual actions
\begin{align}
f \lhd \psi := f^{\ns{-1}}(\psi)f^{\ns{0}}, \qquad\text{and }\qquad \varphi \rhd f := f^{\ns{0}} f^{\ns{1}}(\varphi)
\end{align}
to get  representations of $G_2$ on $\Pc(G_1)$ and  $G_1$ on $\Pc(G_2)$. We denote the action of $G_1$ on $\Pc(G_2)$ by $\rho$.

 In a similar fashion to the Lie group case, we define the action $\Fg_1$ as the derivative of the action of $G_1$. Here $\rho^{\circ}$ being the derivative of $\rho$. By \cite{Hoch3}
\begin{align}\label{polaction}
X \rhd f := \rho^{\circ}(X)(f) = f^{\ns{0}}X(f^{\ns{1}}).
\end{align}

\begin{remark}
In the case of Lie groups, assuming that  the action
\begin{align}
\lhd: G_2 \times G_1 \to G_2\;,
\end{align}
induces a map
\begin{align}
R(\lhd): R(G_2) \to R(G_2 \times G_1) = R(G_2) \ot R(G_1),
\end{align}
we arrive
\begin{align}
\begin{split}
& (X \rhd f) (\psi) = f^{\ns{0}}(\psi)X(f^{\ns{1}}) = \dt f^{\ns{0}}(\psi)f^{\ns{1}}(exptX) = \\
& \dt f(\psi \lhd exptX).
\end{split}
\end{align}
That is, we derive exactly the same action as we got in the \eqref{actionforliegroup}.
\end{remark}

The action of $G_2$ on $\Fg_1$ is also defined as before
\begin{align}
\psi \rhd X := (L_{\psi})^{\circ}(X)
\end{align}
where
\begin{align}
L_{\psi}: G_1 \to G_1, \qquad \varphi \mapsto \psi \rhd \varphi.
\end{align}

Now we prove that $\Pc(G_2)$ coacts on $\Fg_1$. In view of Lemma 1.1 of \cite{Hoch2}, there exists a basis $\{X_1, \cdots, X_n\}$ of $\Fg_1$ and a corresponding subset $S = \{f^1, \cdots, f^n\} \subseteq \Pc(G_1)$ such that $X_j(f^i) = \delta_j^i$. Using the left action of $G_2$ on $\Fg_1$ and dual basis $\{\theta^1, \cdots , \theta^n\}$ for $\Fg_1^\ast$, we introduce the functions $f_i^j:G_2 \to \mathbb{C}$ exactly as before
\begin{align}
f_i^j(\psi) := <\psi \rhd X_i, \theta^j>,
\end{align}
that is
\begin{align}
\psi \rhd X_i = f_i^j(\psi)X_j.
\end{align}

\begin{lemma}
The functions $f_i^j:G_2 \to \mathbb{C}$ defined above are polynomial functions.
\end{lemma}

\begin{proof}
On one hand,
\begin{align}
(\psi \rhd X_i)(f) = X_i(f \lhd \psi) = X_i(f^{\ns{-1}}(\psi)f^{\ns{0}}) = f^{\ns{-1}}(\psi)X_i(f^{\ns{0}}),
\end{align}
while on the other
\begin{align}
(\psi \rhd X_i)(f) = f_i^j(\psi)X_j(f).
\end{align}
Hence, for $f = f^k \in S$ we have $f_i^k(\psi) = (f^k)^{\ns{-1}}(\psi)X_i((f^k)^{\ns{0}}),$
that is,  $f_i^k = X_i((f^k)^{\ns{0}})(f^k)^{\ns{-1}} \in \Pc(G_2).$
\end{proof}
As before, by the very definition of $f_i^j$'s, we have the following coaction.
\begin{align}\label{polcoaction}
\begin{split}
& \Db_{pol}:\Fg_1 \to \Fg_1 \ot \Pc(G_2),\qquad
 X_i \mapsto X_j \ot f_i^j,
\end{split}
\end{align}
and the second order matrix coefficients,
\begin{align}
X_k \rhd f_i^j = f_{i, k}^j.
\end{align}
\begin{proposition}\label{polstructureidentity}
The coaction $\Db_{pol}:\Fg_1 \to \Fg_1 \ot \Pc(G_2)$ satisfies the structure identity of $\Fg_1$.
\end{proposition}

\begin{proof}
We have to show that
\begin{align}
f_{j,i}^k - f_{i,j}^k = \sum_{r,s}C_{s,r}^kf_i^rf_j^s + \sum_lC_{i,j}^lf_l^k.
\end{align}

We first observe
\begin{align}
\begin{split}
& (f \lhd \psi)\ps{1}(\varphi)(f \lhd \psi)\ps{2}(\varphi') = (f \lhd \psi)(\varphi \varphi') = f(\psi \rhd \varphi \varphi') = \\
& f\ps{1}(\psi \rhd \varphi)f\ps{2}((\psi \lhd \varphi) \rhd \varphi') = \\
& (f\ps{2})^{\ns{-1}\ns{0}}(\psi)((f\ps{1} \lhd \psi) \cdot (f\ps{2})^{\ns{-1}\ns{1}})(\varphi)(f\ps{2})^{\ns{0}}(\varphi'),
\end{split}
\end{align}
which implies that
\begin{align}\label{proof-aux-101}
(f \lhd \psi)\ps{1} \ot (f \lhd \psi)\ps{2} = (f\ps{2})^{\ns{-1}\ns{0}}(\psi)(f\ps{1} \lhd \psi) \cdot (f\ps{2})^{\ns{-1}\ns{1}} \ot (f\ps{2})^{\ns{0}}.
\end{align}

Next, by using \eqref{proof-aux-101} we have
\begin{align}
\begin{split}
& (\psi \rhd [X_i, X_j])(f) = [X_i, X_j](f \lhd \psi) = (X_i \cdot X_j - X_j \cdot X_i)(f \lhd \psi) = \\
& X_i((f \lhd \psi)\ps{1})X_j((f \lhd \psi)\ps{2}) - X_j((f \lhd \psi)\ps{1})X_i((f \lhd \psi)\ps{2})=\\
  &   (f\ps{2})^{\ns{-1}\ns{0}}(\psi)[X_i(f\ps{1} \lhd \psi)(f\ps{2})^{\ns{-1}\ns{1}}(e_1) + \\
 & (f\ps{1} \lhd \psi)(e_1)X_i((f\ps{2})^{\ns{-1}\ns{1}})]X_j((f\ps{2})^{\ns{0}}) - \\
 & (f\ps{2})^{\ns{-1}\ns{0}}(\psi)[X_j(f\ps{1} \lhd \psi)(f\ps{2})^{\ns{-1}\ns{0}}(e_1) + \\
  & (f\ps{1} \lhd \psi)(e_1)X_j((f\ps{2})^{\ns{-1}\ns{1}})]X_i((f\ps{2})^{\ns{0}}) = \\
 & [\psi \rhd X_i, \psi \rhd X_j](f) + f^{\ns{-1}\ns{0}}(\psi)X_i(f^{\ns{-1}\ns{1}})X_j(f^{\ns{0}}) - \\
 & f^{\ns{-1}\ns{0}}(\psi)X_j(f^{\ns{-1}\ns{1}})X_i(f^{\ns{0}}).
\end{split}
\end{align}
We finally notice that
\begin{align}
\begin{split}
& (f^{\ns{-1}}X_j(f^{\ns{0}}))(\psi) = f^{\ns{-1}}(\psi)X_j(f^{\ns{0}}) = X_j(f \lhd \psi) = (\psi \rhd X_j)(f) = \\
& (X_kf_j^k(\psi))(f) = (X_k(f)f_j^k)(\psi).
\end{split}
\end{align}
Hence,
\begin{align}
\begin{split}
& f^{\ns{-1}\ns{0}}(\psi)X_i(f^{\ns{-1}\ns{1}})X_j(f^{\ns{0}}) = (X_i \rhd f^{\ns{-1}})(\psi)X_j(f^{\ns{0}}) = \\
& (X_i \rhd f^{\ns{-1}}X_j(f^{\ns{0}}))(\psi) = (X_k(X_i \rhd f_j^k)(\psi))(f)
\end{split}
\end{align}
Similarly
\begin{align}
f^{\ns{-1}\ns{0}}(\psi)X_j(f^{\ns{-1}\ns{1}})X_i(f^{\ns{0}}) = (X_k(X_j \rhd f_i^k)(\psi))(f).
\end{align}
So we have observed the following
\begin{align}
\psi \rhd [X_i, X_j] = [\psi \rhd X_i, \psi \rhd X_j] + X_k(X_i \rhd f_j^k)(\psi) - X_k(X_j \rhd f_i^k)(\psi),
\end{align}
which implies the structure equality immediately.
\end{proof}

We now express the main result of this subsection.

\begin{theorem} \label{Delta(X rhd f)}
Let $(G_1, G_2)$ be a matched pair of affine algebraic groups. Then by the action (\ref{polaction}) and the coaction (\ref{polcoaction}) defined above, $(U(\Fg_1), \Pc(G_2))$ is a matched pair of Hopf algebras.
\end{theorem}

\begin{proof}
In view of Theorem \ref{Theorem-Lie-Hopf-matched-pair} it is enough to show that $\Pc(G_2)$ is a $\Fg_1-$Hopf algebra. In Proposition \ref{polstructureidentity} we prove  that the structure identity is satisfied. Therefore, here we need to prove
\begin{align}
\Delta(X \rhd f) = X \bullet \Delta(f) \quad \mbox{ and } \quad \ve(X \rhd f) = 0.
\end{align}
First we observe that
\begin{align}
\begin{split}
& f^{\ns{0}}\ps{1}(\psi_1)f^{\ns{0}}\ps{2}(\psi_2)f^{\ns{1}}(\varphi) = f^{\ns{0}}(\psi_1\psi_2)f^{\ns{1}}(\varphi) = \\
& f(\psi_1\psi_2 \lhd \varphi) = f\ps{1}(\psi_1 \lhd (\psi_2 \rhd \varphi))f\ps{2}(\psi_2 \lhd \varphi) = \\
& (f\ps{1})^{\ns{0}}(\psi_1)((f\ps{1})^{\ns{1}\ns{-1}} \cdot (f\ps{2})^{\ns{0}})(\psi_2)((f\ps{1})^{\ns{1}\ns{0}} \cdot (f\ps{2})^{\ns{1}})(\varphi),
\end{split}
\end{align}
which  shows that
\begin{align}
\begin{split}
&f^{\ns{0}}\ps{1} \ot f^{\ns{0}}\ps{2} \ot f^{\ns{1}} = \\
&(f\ps{1})^{\ns{0}} \ot (f\ps{1})^{\ns{1}\ns{-1}} \cdot (f\ps{2})^{\ns{0}} \ot (f\ps{1})^{\ns{1}\ns{0}} \cdot (f\ps{2})^{\ns{1}}
\end{split}
\end{align}
Therefore,
\begin{align}
\begin{split}
& \D(X\rt f)(\psi_1,\psi_2)=(X \rhd f)(\psi_1\psi_2) = f^{\ns{0}}\ps{1}(\psi_1)f^{\ns{0}}\ps{2}(\psi_2)X(f^{\ns{1}}) = \\
& (f\ps{1})^{\ns{0}}(\psi_1) (\psi_2 \rhd X)((f\ps{1})^{\ns{1}})f\ps{2}(\psi_2) + f\ps{1}(\psi_1 \lhd e)(X \rhd f\ps{2})(\psi_2) = \\
& ((\psi_2 \rhd X) \rhd f\ps{1})(\psi_1)f\ps{2}(\psi_2) + f\ps{1}(\psi_1)(X \rhd f\ps{2})(\psi_2) = \\
& (X\ns{0} \rhd f\ps{1})(\psi_1)(X\ns{1} \cdot f\ps{2})(\psi_2) + f\ps{1}(\psi_1)(X \rhd f\ps{2})(\psi_2)=\\
&(X\bullet \D(f))(\psi_1,\psi_2).
\end{split}
\end{align}
 Next, we want to prove that $\ve(X \rhd f) = 0$. To this end, we  notice that
\begin{align}
\ve(X \rhd f) = (X \rhd f)(e_2) = f^{\ns{0}}(e_2)X(f^{\ns{1}}) = X(f^{\ns{0}}(e_2)f^{\ns{1}})=0,
\end{align}
 because $f^{\ns{0}}(e_2)f^{\ns{1}} \in \Pc(G_1)$ is constant. The proof  is  done.
\end{proof}

Now we construct the following pairing
\begin{align}\label{pairing-polynomial}
\begin{split}
& \langle ,\rangle :\Pc(G_2) \times U(\Fg_2) \to \mathbb{C} \\
& (f, v) \mapsto \langle f,v\rangle  := f^*(v) \quad \mbox{ where } \quad f^*(v) := (v \cdot f)(e_2)
\end{split}
\end{align}
Here, the left action $U(\Fg_2)$ comes from the derivative of the left translation by $G_2$.

\begin{proposition}\label{starmapcomultiplicationcompatibility}
The pairing \eqref{pairing-polynomial} defines a Hopf duality between  $\Pc(G_2)$ and $U(\Fg_2)$. In other words
\begin{align}
&\langle f, v^1v^2\rangle =  \langle f\ps{1}, v^1\rangle\langle f\ps{2}, v^2\rangle,\quad \langle f, 1\rangle=\ve(f),\\
&\langle fg, v\rangle=\langle f, v\ps{1}\rangle\langle g, v\ps{2}\rangle,\quad \langle 1, v\rangle=\ve(v).
\end{align}
\end{proposition}

\begin{proof}
Let $v^1 = \xi_1 \cdots \xi_n $ and $v^2 = \xi'_1 \cdots \xi'_m$ for $\xi_1, \cdots , \xi_n, \xi'_1, \cdots , \xi'_m \in \Fg_2$. Since  left translation is given by $\psi \rhd f = f\ps{1}f\ps{2}(\psi) $, we observe   $ \xi \cdot f = f\ps{1}\xi(f\ps{2})$,
which evidently implies that $\langle f,\xi_1 \cdots \xi_n\rangle  = \xi_1(f\ps{1}) \cdots \xi_n(f\ps{n})$. Therefore
\begin{align}
\begin{split}
& \langle f,v^1v^2\rangle  = f^*(v^1v^2) = \xi_1(f\ps{1}) \cdots \xi_n(f\ps{n})\xi'_1(f\ps{n+1}) \cdots \xi'_m(f\ps{n+m}) = \\
& (f\ps{1})^*(\xi_1 \cdots \xi_n)(f\ps{2})^*(\xi'_1 \cdots \xi'_m) = (f\ps{1})^*(v^1)(f\ps{2})^*(v^2) = \\
& \langle f\ps{1},v^1\rangle \langle f\ps{2},v^2\rangle .
\end{split}
\end{align}
We easily see that
\begin{align}
\langle   f,1\rangle  = (1 \cdot f)(e_2) =  f(e_2) = \ve(f).
\end{align}

To show the other compatibility, we first observe that
\begin{align}
\begin{split}
& \langle f  g,\xi\rangle  = (f g)^*(\xi) = \xi(f \cdot g) = \\
& \xi(f)g(e_2) + f(e_1)\xi(g) = f^*(\xi)g^*(1) + f^*(1)g^*(\xi) =  \langle f,\xi\ps{1}\rangle \langle g,\xi\ps{2}\rangle
\end{split}
\end{align}
and by induction
\begin{align}
\begin{split}
& \langle f g,v\xi\rangle  = (f g)^*(v\xi) = \\
& (f\ps{1} g\ps{1})^*(v)(f\ps{2}  g\ps{2})^*(\xi) = \\
& (f\ps{1})^*(v\ps{1})(g\ps{1})^*(v\ps{2})[(f\ps{2})^*(\xi)(g\ps{2})^*(1) + (f\ps{2})^*(1)(g\ps{g})^*(\xi)] = \\
& f^*(v\ps{1}\xi)g^*(v\ps{2}) + f^*(v\ps{1})g^*(v\ps{2}\xi) = f^*((v\xi)\ps{1})g^*((v\xi)\ps{2}) = \\
& \langle f,(v\xi)\ps{1}\rangle \langle g,(v\xi)\ps{2}\rangle.
\end{split}
\end{align}
\end{proof}
Finally we prove the following proposition.
\begin{proposition}\label{starmapisg1linear}
The pairing $\langle ,\rangle :\Pc(G_2) \times U(\Fg_2) \to \mathbb{C}$ is $U(\Fg_1)-$balanced.
In other words
\begin{equation}
\langle f\,,\, v\lt X\rangle =\langle X\rt f\,,\, v\rangle.
\end{equation}
\end{proposition}

\begin{proof}
The right action of $G_1$ on $\Fg_2$  leads to the coaction
\begin{align}
\begin{split}
& \Db:\Fg_2 \to \Pc(G_1) \ot \Fg_2, \\
& \xi \mapsto \xi\ns{-1} \ot \xi\ns{0} \quad \mbox{ such that } \xi\ns{-1}(\vp)\xi\ns{0} = \xi \lhd \vp.
\end{split}
\end{align}
We first observe that
\begin{align}
\begin{split}
& \xi\ns{-1}(\vp)\xi\ns{0}(f) = (\xi \lhd \vp)(f) = \xi(\vp \rhd f) = \\
& \xi(f^{\ns{0}}f^{\ns{1}}(\vp)) = \xi(f^{\ns{0}})f^{\ns{1}}(\vp)),
\end{split}
\end{align}
that is,
\begin{align}
\xi\ns{-1}\xi\ns{0}(f) = \xi(f^{\ns{0}})f^{\ns{1}}.
\end{align}
Therefore,
\begin{equation*} \langle f,\xi \lhd X\rangle  = \xi(f^{\ns{0}})X(f^{\ns{1}}) = \xi(f^{\ns{0}}X(f^{\ns{1}})) =
 \xi(X \rhd f) = \langle X \rhd f,\xi\rangle.
\end{equation*}
We finish the proof by induction as follows
\begin{align}
\begin{split}
& \langle X \rhd f, v\xi\rangle  = (X \rhd f)^*(v\xi) = ((X \rhd f)\ps{1})^*(v)((X \rhd f)\ps{2})^*(\xi) = \\
& (X\ns{0} \rhd f\ps{1})^*(v)(X\ns{1} \cdot f\ps{2})^*(\xi) + (f\ps{1})^*(v)(X \rhd f\ps{2})^*(\xi) = \\
& (X\ns{0} \rhd f\ps{1})^*(v)\xi(X\ns{1} \cdot f\ps{2}) + (f\ps{1})^*(v)\xi(X \rhd f\ps{2}) = \\
& (X\ns{0} \rhd f\ps{1})^*(v)[\xi(X\ns{1})f\ps{2}(e_2) + X\ns{1}(e_2)\xi(f\ps{2})] + \\
& (f\ps{1})^*(v)(\xi \lhd X)(f\ps{2}) \\
&  f^*(v \lhd (\xi \rhd X)) + f^*((v \lhd X)\xi) + f^*(v(\xi \lhd X)) = f^*(v\xi \lhd X) = \\
& \langle f,v\xi \lhd X\rangle.
\end{split}
\end{align}
\end{proof}

\section{Hopf cyclic coefficients}
\label{S2}
In this section we first recall the definition of  modular pair in involution (MPI)  and stable-anti-Yetter-Drinfeld (SAYD)  module  from \cite{Connes-Moscovici-00} and \cite{Hajac-Khalkhali-Rangipour-Sommerhauser-04-1} respectively. In  Subsection \ref{Canonical MPI associated to Lie-Hopf algebras},
 we canonically associate a modular pair in involution over $\Fc\acl U(\Fg)$ to  any $\Fg$-Hopf algebra $\Fc$. This modular pair in involution plays an important r\^ole in the last section. In Subsection \ref{Induced Hopf cyclic coefficients}, we characterize a subcategory of the category of SAYD modules over a Lie-Hopf algebra and call them induced SAYD modules.  In Subsection \ref{Induced Hopf cyclic coefficients in geometric cases}, we completely determine the induced modules over the  geometric Hopf algebras that  we constructed in Subsections \ref{SS-Lie algebra}, \ref{SS-Lie group}, and \ref{SS-algebraic group}, based on the modules over the geometric object on which  the Hopf algebra is constructed.

\medskip

Let $\Hc$ be a Hopf algebra. By definition, a character $\d: \Hc\ra \Cb$ is an algebra map.
A group-like  $\s\in \Hc$ is the dual object of the character, i.e, $\D(\s)=\s\ot \s$. The pair $(\d,\s)$ are called  modular pair in involution \cite{Connes-Moscovici-00} if
 \begin{equation}\label{def-MPI}
 \d(\s)=1, \quad \text{and}\quad  S_\d^2=Ad_\s,
 \end{equation}

 where  $Ad_\s(h)= \s h\s^{-1}$ and  $S_\d$ is defined
by
\begin{equation}
S_\d(h)=\d(h\ps{1})S(h\ps{2}).
\end{equation}
We recall from \cite{Hajac-Khalkhali-Rangipour-Sommerhauser-04-1} the definition of a  right-left  stable-anti-Yetter-Drinfeld module over a Hopf algebra $\Hc$. Let $\Mc$ be a right module and left comodule over a Hopf algebra $\Hc$. We say it is stable-anti-Yetter-Drinfeld (SAYD) if
\begin{equation}\label{def-SAYD}
\Db(m\cdot h)= S(h\ps{3})m\ns{-1}h\ps{1}\ot m\ns{0}\cdot h\ps{2},\qquad  m\ns{0}m\ns{-1}=m,
\end{equation}
for any $m\in \Mc$ and $h\in \Hc$.
 It is shown in \cite{Hajac-Khalkhali-Rangipour-Sommerhauser-04-1} that any MPI defines a one dimensional SAYD module and all one dimensional SAYD modules come this way.

\subsection{ Canonical MPI associated to Lie-Hopf algebras}
\label{Canonical MPI associated to Lie-Hopf algebras}
In this subsection, we associate a  canonical modular pair in involution to any  Hopf algebra we constructed in the previous section. More generally let $\Fc$ be a $\Fg$-Hopf algebra. We associate a canonical modular pair in involution to the Hopf algebra  $\Fc\acl U(\Fg)$.

First we define
\begin{equation}\label{trace-ad}
\d_\Fg:\Fg\ra\Cb, \quad \d_\Fg(X)=Tr(Ad_X).
 \end{equation}

 It is known that $\d_\Fg$ is a derivation of Lie algebras. So we  extend $\d_\Fg$ to an algebra map on $U(\Fg)$ which is
again denoted by  $\d_\Fg$.   We extend $\d_\Fg$ to an algebra map  $\d:=\ve\acl \d_\Fg:\Fc\acl U(\Fg)\ra \Cb$ by
\begin{equation}\label{delta}
\d(f\acl u)=\ve(f)\d_\Fg(u).
\end{equation}
Let us introduce a canonical  element  $\s_\Fc\in\Fc$ which plays an important r\^ole   in our study.   For a $\Fg$-Hopf algebra $\Fc$, we recall  $f^i_j\in \Fc$  defined in \eqref{def-Db-Fg} for a fixed basis $X_1, \ldots, X_m$ of $\Fg$, i.e,

\begin{equation}
\Db(X_j)=\sum_{i=1}^{m=\dim\Fg} X_i\ot f^i_j.
\end{equation}
As it is shown, $f^i_j$ are independent of $\{X_1, \ldots, X_m\}$ as a basis for $\Fg$. Now we define,

\begin{equation}\label{sigma}
\s_\Fc:= \det M_f=\sum_{\pi\in S_m}(-1)^\pi f_1^{\pi(1)}\dots f_m^{\pi(m)}.
\end{equation}
where  $M_f\in M_m(\Fc)$ is the matrix  $[f^i_j]\in M_m(\Fc)$.
\begin{lemma}\label{group-like-lemma}
The element $\s_\Fc\in \Fc$ is group-like.
\end{lemma}
\begin{proof}
We know that  $\D(f_i^j)=\sum_k f_k^j\ot f_i^k$.  So,
\begin{align*}
&\D(\s_\Fc)= \sum_{\pi\in S_m}\sum_{i_1,\dots, i_m}(-1)^\pi f_{i_1}^{\pi(1)}\cdots f_{i_m}^{\pi(m)} \ot f_1^{i_1}\cdots f_m^{i_m}=\\ &\sum_{\pi\in
S_m}\sum_{\underset{\text{  \footnotesize are  distinct }}{i_1,\dots, i_m}}(-1)^\pi f_{i_1}^{\pi(1)}\cdots f_{i_m}^{\pi(m)} \ot f_1^{i_1}\cdots f_m^{i_m}+\\
&\sum_{\pi\in S_m}\sum_{\underset{\text{  \footnotesize are not  distinct }}{i_1,\dots, i_m}}(-1)^\pi f_{i_1}^{\pi(1)}\cdots f_{i_m}^{\pi(m)} \ot f_1^{i_1}\cdots
f_m^{i_m}.
\end{align*}
One uses the fact that $\Fc$ is commutative to prove that the second sum is zero. We now deal with the the first sum. We associate a unique permutation
$\mu\in S_m$ to each distinct  $m$-tuple $(i_1,\dots,i_m)$ by the rule $\mu(j)=i_j$.  So we have
\begin{align*}
&\D(\s_\Fc)= \sum_{\pi\in S_m} \sum_{\mu\in S_m}(-1)^\pi f_{\mu(1)}^{\pi(1)}\cdots f_{\mu(m)}^{\pi(m)} \ot f_1^{\mu(1)}\cdots f_m^{\mu(m)}=\\ & \sum_{\pi\in S_m}
\sum_{\mu\in S_m}(-1)^\pi(-1)^{\mu}(-1)^{\mu^{-1}} f_{1}^{\pi(\mu^{-1}(1))}\cdots f_{m}^{\pi(\mu^{-1}(m))} \ot f_1^{\mu(1)}\cdots f_m^{\mu(m)}=\\ &
\sum_{\eta,\mu\in S_m} (-1)^\eta(-1)^{\mu} f_{1}^{\eta(1)}\cdots f_{m}^{\eta(m)} \ot f_1^{\mu(1)}\cdots f_m^{\mu(m)}=\s_\Fc\ot\s_\Fc.
\end{align*}
\end{proof}

One easily sees that $\s:=\s_\Fc\acl 1$ is a group-like element in the Hopf algebra  $\Fc\acl U(\Fg)$.

\begin{theorem}\label{MPI-theorem}
For any $\Fg$-Hopf algebra  $\Fc$, the pair $(\delta, \sigma)$ is a modular  pair in involution  for the Hopf algebra $\Fc\acl U(\Fg)$.
\end{theorem}
\begin{proof}
Let us first do some preliminary computations. For an element
 $1 \acl X_i \in \Hc:=\Fc\acl U(\Fg)$,  the action of iterated comultiplication $\D^{(2)}$ is calculated by
\begin{align*}
&\D^{(2)}(1 \acl X_i) =  (1 \acl X_i)\ps{1} \ot (1 \acl X_i)\ps{2} \ot (1 \acl X_i)\ps{3} \\
 & =  1 \acl X_i\ns{0} \ot X_i\ns{1} \acl 1 \ot X_i\ns{2} \acl 1 + \\
 &1 \acl 1 \ot 1 \acl X_i\ns{0} \ot X_i\ns{1} \acl 1   + 1 \acl 1 \ot 1 \acl1 \ot 1 \acl X_i.
\end{align*}
By definition of the antipode \eqref{anti}, we observe
\begin{align*}
&S(1 \acl X_i)  =  (1 \acl S(X_i\ns{0})) (S(X_i\ns{1}) \acl 1) =  - (1 \acl X_j)  (S(f_i^j) \acl 1) \\
 & =  - X_j\ps{1} \rhd S(f_i^j) \acl X_j\ps{2}  =  - X_j \rhd S(f_i^j) \acl 1 - S(f_i^j) \acl  X_j,
\end{align*}
and hence
\begin{align*}
&S^2(1 \acl  X_i)  = - S(X_j \rhd S(f_i^j)) \acl  1 - (1 \acl  S(X_j\ns{0})) \cdot (S(S(f_i^j)X_j\ns{1}) \acl  1) \\
 & =  - S(X_j\ns{1})(X_j\ns{0} \rhd f_i^j) \acl  1 + (1 \acl  X_k) \cdot (f_i^jS(f_j^k) \acl  1) \\
  &=  - S(f_j^k)(X_k \rhd f_i^j) \acl  1 + 1 \acl  X_i   =  - S(f_j^k)f_{i,k}^j \acl  1 + 1 \acl  X_i.
\end{align*}
Finally, for the twisted antipode $S_{\delta}:\cal H \to \cal H$, we  simplify   its square action by
\begin{equation}
  S_{\delta}^2(h) =  \delta(h\ps{1})\delta(S(h\ps{3}))S^2(h\ps{2}), \quad h\in \Hc.
\end{equation}
We aim to prove that
\begin{equation}\label{MPI-equation}
S_{ \delta}^2 = \Ad_{\sigma}.
\end{equation}

Since the twisted antipode is anti-algebra map, it is enough to  prove \eqref{MPI-equation}  is held for the elements of the form $1 \acl X^i$ and $f\acl 1$.
 For the latter elements, it is seen that  $S_\d(f\acl 1)=S(f)\acl 1$. Hence there is nothing to prove, since $S^2(f)=f=\s f\s^{-1}$.
  According to the above preliminary calculations, we observe that
\begin{equation}\label{aux6}
S_{ \delta}^2(1 \acl X_i)  =   \delta_\Fg(X_j)(f_i^j \acl 1) + (1 \acl X_i) - S(f_j^k)f_{i,k}^j \acl 1 - \delta_\Fg(X_i)1 \acl 1).
\end{equation}
Multiplying  $S(f_j^k)$ to both hand sides of  the structure  identity \eqref{Bianchi}  which is recalled here
\begin{align}
f_{i,k}^j - f_{k,i}^j = \sum_{r,s}C_{s,r}^jf_k^rf_i^s - \sum_lC_{i,k}^lf_l^j,
\end{align}
we obtain the following expression
\begin{align}\notag
&- S(f_j^k)f_{k,i}^j =  - S(f_j^k)f_{i,k}^j + \sum_{r,s}C_{s,r}^jS(f_j^k)f_k^rf_i^s - \sum_lC_{i,k}^lS(f_j^k)f_l^j \\\label{aux5}
 & =  - S(f_j^k)f_{i,k}^j + \sum_{s,j}C_{s,j}^jf_i^s - \sum_lC_{i,l}^l  1_\Fc  =  - S(f_j^k)f_{i,k}^j + \delta_\Fg(X_s)f_i^s - \delta_\Fg(X_i) 1_\Fc.
\end{align}
Combining \eqref{aux5} and  \eqref{aux6} we get
\begin{equation}\label{aux7}
S_{\delta}^2(1 \acl X_i) = - S(f_j^k)f_{k,i}^j \acl 1 + 1 \acl X_i.
\end{equation}
On the other hand, since $\Fg$ acts on $\Fc$ by derivation, we see that
\begin{align}\notag
&0  =  X_i \rhd (S(f_j^k)f_k^j)  =  f_k^j (X_i \rhd S(f_j^k)) + S(f_j^k)(X_i \rhd f_k^j) \\\label{aux8}
& =  f_k^j(X_i \rhd S(f_j^k)) + S(f_j^k)f_{k,i}^j.
\end{align}

From  \eqref{aux8}
 and \eqref{aux7}  we deduce  that
\begin{equation}\label{aux9}
S_{\delta}^2(1 \acl X_i) = f_k^j(X_i \rhd S(f_j^k)) \acl 1 + 1 \acl X_i
\end{equation}

Now we  consider the element
\begin{equation}
  \sigma^{-1} = det[S(f_j^k)] = \sum_{\pi \in S_m}(-1)^{\pi}S(f_1^{\pi(1)})S(f_2^{\pi(2)}) \dots S(f_m^{\pi(m)}),
\end{equation}
 and by using the fact that   $\Fg$ acts on $\Fc$   by  derivation we observe that
\begin{align}\notag
&X_i \rhd \sigma^{-1} = X\rt\det[S(f_j^k)] =\\\label{aux10}
  &\sum_{1\le j\le m,\;\pi \in S_m}(-1)^{\pi}S(f_1^{\pi(1)})S(f_2^{\pi(2)}) \cdots\; X_i \rhd S(f_j^{\pi(j)})\;\cdots S(f_m^{\pi(m)}).
\end{align}
Since   $\s=\det[f_i^j]$ and $\Fc$ is commutative, we observe  $\s=\det [f_i^j]^T$. Here  $[f_i^j]^T$ is meant the  transpose of the matrix $[f_i^j]$. We can then conclude that
\begin{align}
\sigma(\sum_{\pi \in S_m}(-1)^{\pi}[X_i \rhd S(f_1^{\pi(1)})]S(f_2^{\pi(2)}) \dots S(f_m^{\pi(m)})) = f_k^1(X_i \rhd S(f_1^k)),
\end{align}
which  implies
\begin{align}\label{aux11}
\sigma(X_i \rhd \sigma^{-1}) = f_k^j(X_i \rhd S(f_j^k)).
\end{align}
Finally, by
\begin{align*}
&Ad_{\sigma}(1 \acl X_i) =  (\sigma \acl 1)(1 \acl X_i)(\sigma^{-1} \acl 1)= \sigma X_i \rhd \sigma^{-1} \acl 1 + 1 \acl X_i,
\end{align*}
followed we substitute   \eqref{aux11} in \eqref{aux9} finishes the proof verifying
\begin{align}\label{MPI-equation-proof}
S_{\delta}^2(1 \acl X_i) = Ad_{\sigma}(1 \acl X_i).
\end{align}
\end{proof}

\subsection{Induced Hopf cyclic coefficients}
\label{Induced Hopf cyclic coefficients}
We investigate the category of SAYD  modules over the Hopf algebra associated to a  $\Fg$-Hopf algebra $\Fc$, i.e, ${\nobreak \Hc:=\Fc\acl U(\Fg)}$. We determine a subcategory of it whose objects are  called  by us induced SAYD modules over $\Hc$. Our strategy  is to find a YD module over $\Hc$, simple enough to work with and rich enough to give us a representation of the geometric  ambient object i.e,  the Lie algebra,  Lie group, and  algebraic group that we started with. Next, we tensor the induced  YD module with the canonical modular pair in involution and use \cite[Lemma 2.3]{Hajac-Khalkhali-Rangipour-Sommerhauser-04-1} to get our desired SAYD module.

\medskip

Let $M$ be a left   $\Fg$-module and a  right $\Fc$-comodule via $\Db_M:M\ra M\ot \Fc$. We say that $M$ is an induced $(\Fg,\Fc)$-module if
\begin{equation}\label{def-induced}
  \Db_M(X\cdot m)= X\bullet \Db_M(m).
\end{equation}
Here, as before, $X\bullet ( m\ot f)= X\ns{0}m\ot X\ns{1}f+ m\ot X\rt f$.
One extend the action of $\Fg$ on $M$ to an action of $U(\Fg)$ on $M$ in the natural way.
We observe that
\begin{lemma}The coaction $\Db:M\ra F\ot M$ is $U(\Fg)$-linear. In other words,
\begin{equation}\label{YD-proof}
\Db_M(u\cdot m)= u\bullet \Db_M(m):= u\ps{1}\ns{0}\cdot m\ns{0}\ot u\ps{1}\ns{1}(u\ps{2}\rt m\ns{1}).
\end{equation}
\end{lemma}
\begin{proof}
For any $X\in \Fg$, the condition \eqref{YD-proof} is obviously satisfied. Let assume that it is satisfied for $u^1, u^2\in U(\Fg)$ and any $m\in M$ we show it is also held for $u^1u^2$ and any $m\in M$. Using \eqref{mp4} we observe
\begin{align}
\begin{split}
&\Db_M(u^1u^2\cdot m)= u^1\ps{1}(u^2\cdot m)\ns{0}\ot u^1\ps{1}\ns{1}(u^1\ps{2}\rt (u^2\cdot m)\ns{1})=\\
& u^1\ps{1}\ns{0} u^2\ps{1}\ns{0}\cdot m\ns{0}\ot u^1\ps{1}\ns{1} (u^1\ps{2}\rt (u^2\ps{1}\ns{1}(u^2\ps{1}\rt m\ns{1})))=\\
& u^1\ps{1}\ns{0} u^2\ps{1}\ns{0} \cdot m\ns{0} \ot u^1\ps{1}\ns{1} u^1\ps{2}\rt u^2\ps{1}\ns{1}( u^1\ps{3}u^2\ps{1}\rt m\ns{1})=\\
& (u^1u^2)\ps{1}\ns{0}\cdot m\ns{0} \ot (u^1u^2)\ps{1}\ns{1} (u^1u^2)\ps{2}\rt m\ns{1}.
\end{split}
\end{align}
\end{proof}

Now we let $\Hc$ act on $M$ from left via
\begin{equation}\label{action-M}
 \Hc\ot M \ra M, \quad (f\acl u)m=\ve(f)um.
\end{equation}
Since $\Fc$ is a Hopf subalgebra of $\Hc$ one easily extend the coaction of $\Fc$ on $M$ to a right coaction of $\Hc$ on $M$ via,
\begin{equation}\label{coaction-M}
\Db_M: M\ra  M\ot \Hc, \quad \Db_M(m)= m\ns{0}\ot m\ns{1}\acl 1.
\end{equation}

For a Hopf algebra $\Hc$, a left module and right comodule $M$, via $\Db_M(m)=m\ns{-1}\ot m\ns{0}$ is called Yetter-Drinfeld module (YD for short)  if the following condition satisfied
\begin{equation}\label{YD-condition-1}
(h\ps{2}\cdot m)\ns{0}\ot (h\ps{2}\cdot m)\ns{1}h\ps{1}= h\ps{1}\cdot  m\ns{0}\ot h\ps{2}m\ns{1}.
\end{equation}
 The above condition equivalent to the following one in case  $\Hc$ has invertible antipode.
\begin{equation}\label{YD-condition-2}
\Db_M(h\cdot m)= h\ps{2}\cdot m\ns{0}\ot  h\ps{3}m\ns{1} S^{-1}(h\ps{1}).
\end{equation}
\begin{proposition}
Let $\Fc$ be a $\Fg$-Hopf algebra and  $M$ an induced $(\Fg, \Fc)$-module. Then, via the action and coaction  defined in \eqref{action-M} and \eqref{coaction-M},  $M$ is a YD-module over $\Fc\acl U(\Fg)$.
\end{proposition}
\begin{proof}
Since condition \eqref{YD-condition-1} is multiplicative, it suffices to check it for
$f\acl 1$ and $1\acl u$. We check it for the latter element as for  the former it is obviously satisfied.
We see that $\D(1\acl u)= 1\acl u\ps{1}\ns{0}\ot u\ps{1}\ns{1}\acl u\ps{2}$. We use \eqref{YD-proof} and the fact that $U(\Fg)$ is cocommutative to observe that for $h= 1\acl u$ we have
\begin{align}
\begin{split}
&(h\ps{2}\cdot m)\ns{0}\ot (h\ps{2}\cdot m)\ns{1}h\ps{1}=\\
&((u\ps{1}\ns{1}\acl u\ps{2})\cdot m)\ns{0}\ot ((u\ps{1}\ns{1}\acl u\ps{2})\cdot m)\ns{1}1\acl u\ps{1}\ns{0}=\\
&u\ps{2}\ns{0}\cdot m\ns{0}\ot u\ps{2}\ns{1}(u\ps{3}\rt m\ns{1})\acl u\ps{1}=\\
&h\ps{1}\cdot m\ns{0}\ot h\ps{2}m\ns{1}.
\end{split}
\end{align}
\end{proof}

Now, one knows that by composing $S$ with the action and  $S^{-1}$with the coaction, we change  a left-right YD module over $\Hc$ into a right-left YD module over $\Hc$. However, since the coaction always lands in $\Fc$ and $S^{-1}(f\acl 1)= S^{-1}(f)\acl= S(f)\acl$,  we  conclude that the following defines a right-left YD module over $\Hc$.
\begin{equation}\label{right-action-M}
  M \ot \Hc\ra M, \quad m(f\acl u)=\ve(f)S(u)\cdot m.
\end{equation}
\begin{equation}
\Db_M: M\ra \Hc\ot  M, \quad \Db_M(m)= S(m\ns{1})\acl 1\ot m\ns{0}.
\end{equation}
It means that the above right $\Hc$-module and left $\Hc$-comodule  $M$ satisfies
\begin{equation}\label{right-YD-condition-1}
h\ps{2}(m\cdot h\ps{1})\ns{-1}\ot (m\cdot h\ps{1})\ns{0}= m\ns{-1}h\ps{1}\ot m\ns{0}\cdot h\ps{2},
\end{equation}
 or equivalently,
\begin{equation}\label{right-YD-condition-2}
\Db_M(m\cdot h)= S^{-1}(h\ps{3})m\ns{-1} h\ps{1}\ot m\ns{0}\cdot h\ps{2}.
\end{equation}

The main aim of this section is to provide a class of SAYD module over the Hopf algebra $\Fc\acl U(\Fg)$ for a $\Fg$-Hopf algebra $\Fc$.

It is known that tensor product of an AYD module and a  YD module is an AYD \cite{Hajac-Khalkhali-Rangipour-Sommerhauser-04-1}.
Let us give a proof for this fact here.
Let $N$ be a right-left AYD module  over a Hopf algebra $\Hc$ and $M$ be a right-left   YD module over $\Hc$. We endow $N\ot M$ with the the action $(n\ot m)h:= nh\ps{2}\ot mh\ps{1}$ we also endow $N\ot M$ with the coaction $\Db_{N\ot M}(n\ot m)= n\ns{-1}m\ns{1}\ot n\ns{0}\ot m\ns{0}$.
We observe that
\begin{align}
\begin{split}
&\Db_{N\ot M}((n\ot m)\cdot h)=\Db_{N\ot M}(n\cdot h \ps{2}\ot m\cdot h\ps{1})=\\
& S(h\ps{6})n\ns{-1}h\ps{4} S^{-1}(h\ps{3})m\ns{-1} h\ps{1}\ot n\ns{0}\cdot h\ps{5}\ot m\ns{0}\cdot h\ps{2}=\\
& S(h\ps{3})n\ns{-1}m\ns{-1}h\ps{1}\ot (n\ns{0}\ot m\ns{0})\cdot h\ps{2}.
\end{split}
\end{align}

In general, even if $N$ and $M$ are stable, there is no grantee that $N\ot M$ becomes stable. However, our case is special and one easily sees that $N\ot M$ is stable. Since any MPI is a one dimensional AYD module,  $^\s\Cb_\d\ot M$ is an AYD module over $\Fc\acl U(\Fg)$.  We denote $^\s\Cb_\d\ot M$ by $^\s{M}_\d$. We simplify the action of $\Fc\acl U(\Fg)$ as follows.

\begin{align}
\begin{split}\label{action-SAYD}
&\; ^\s{M}_\d\ot \Fc\acl U(\Fg)\ra\; ^\s{M}_\d,\quad m\lt_{(\d,\s)}(f\acl u):= \ve(f) \d(u\ps{2}) S(u\ps{2})m\,
 \end{split}
\end{align}
Its coaction is given by,
\begin{align}\label{coaction-SAYD}
\begin{split}
&^\s\Db_\d: \; ^\s{M}_\d\ra  \Fc\acl U(\Fg)\ot\; ^\s{M_\d},\\
&\Db_{(\d,\s)}(m):= \s S(m\ns{1})\acl 1\ot m\ns{0}.
 \end{split}
\end{align}

We summarize this subsection in the following.
\begin{theorem}\label{Theorem-SAYD}
  For any $\Fg$-Hopf algebra $\Fc$ and any induced $(\Fg, \Fc)$-module $M$, there is a  SAYD module structure on $^\s{M}_\d:= \; ^\s\Cb_\d\ot M$ over $\Fc\acl U(\Fg)$ defined in \eqref{action-SAYD} and \eqref{coaction-SAYD}. Here,  $(\d,\s)$ is the canonical modular pair in involution associated to $(\Fg,\Fc)$.
\end{theorem}

\subsection{Induced Hopf cyclic coefficients in geometric cases}
\label{Induced Hopf cyclic coefficients in geometric cases}
In Section \ref{S1}, we associated a  Hopf algebra  to any  matched  pair of Lie algebras,  Lie groups, and  affine algebraic groups. These geometric objects admit their own representations. In this section, to such a representation we associate an induced module over the corresponding matched pair and  hence a SAYD module over the Hopf algebra in question.

 \medskip

 We start with the $\Fg_1$-Hopf algebra $R(\Fg_2)$, where $(\Fg_1,\Fg_2)$ is a matched pair of Lie algebras. A left module $M$ over the double crossed sum Lie algebra $\Fg_1\bi\Fg_2$ is naturally a left  $\Fg_1$-module as well as a left $\Fg_2$-module. In addition, it satisfy the following compatibility condition.
 \begin{equation}\label{Lie-module-compatibility}
 \z\cdot(X\cdot m)- X\cdot (\z\cdot m)= (\x\rt X)\cdot m+ (\z\lt X)\cdot m.
 \end{equation}
 Conversely, if $M$ is left module over $\Fg_1$ and $\Fg_2$ satisfying \eqref{Lie-module-compatibility}, then $M$ is a $\Fg_1\bi\Fg_2$-module in its natural way, i.e. $(X\oplus \z)\cdot m:= X\cdot m+\z\cdot m$. This is generalized in the following lemma whose proof is elementary and is omitted.

 \begin{lemma}\label{lemma-U-V-module}
Let $\Uc$ and $\Vc$ be a mutual pair of Hopf algebras and  $M$  a left module over both Hopf algebras. Then $M$ is a left $\Uc \bi \Vc$-module via the action
 \begin{equation}
 (u \bi v)\cdot m:=u \cdot (v \cdot m),
  \end{equation}
  if and only if
  \begin{equation}\label{Hopf-module-compatibility}
  (v\ps{1}\rt u\ps{1})\cdot ((v\ps{2}\lt u\ps{2}) \cdot m)=v\cdot(u \cdot m),
   \end{equation}
    for any $u \in \Uc$, $v\in \Vc$ and $m \in M$. Conversely, every left module over $\Uc\bi\Vc$ comes this way.
\end{lemma}

Now let $M$ be a left $\Fg:=\Fg_1\bi\Fg_2$-module such that the restriction of action $M$  is  locally finite over $\Fg_2$.  We let $U(\Fg_1)$ and $U(\Fg_2)$ act from left on $M$ in the natural way. Hence the resulting actions satisfy the condition \eqref{Hopf-module-compatibility}. We know that the category of locally finite $\Fg_2$ modules is equivalent with the category of $R(\Fg)$-comodules \cite{Hochschild-Mostow-57}. The functor between these  two categories is as follows. Let $M$ be a locally finite left $U(\Fg_2)$-module and define the coaction $\Db_M:M\ra M\ot R(\Fg_2)$ by
\begin{equation}\label{U(g)-dual-coaction}
\Db(m)=m\ns{0}\ot m\ns{1}, \quad \text{if and only if}, \quad  v\cdot m= m\ns{1}(v)m\ns{0}.
\end{equation}

Conversely, let $M$ be a right coaction via $\Db_M:M\ra M\ot R(\Fg_2)$. Then one defines the left action of $U(\Fg_2)$ on $M$ by
\begin{equation}\label{action-coaction}
v\cdot m:= m\ns{1}(v)m\ns{0}.
\end{equation}

\begin{proposition}\label{proposition-module-induced-module-Lie algebraas}
Let $M$ be a left $\Fg:=\Fg_1\bi\Fg_2$-module such that the restriction of the action results a locally finite $\Fg_2$-module. Then, via the $\Fg_1$ action on $M$, by restriction, and the coaction defined in  \eqref{U(g)-dual-coaction}, $M$ becomes an induced $(\Fg_1,R(\Fg_2))$-module.  Conversely, every induced $(\Fg_1, R(\Fg_2))$-module comes this way.
\end{proposition}

\begin{proof}
Let $M$ satisfies the criteria of the proposition. We prove that it is an induced $(\Fg_1, R(\Fg_2))$-module, i.e. $\Db_M(X\rt m)=X\bullet \Db(m)$. Using the compatibility condition \eqref{Hopf-module-compatibility} for $v\in U(\Fg_2)$ and $X \in \Fg_1$, we observe that
\begin{equation}\label{proof-compatibility}
(v\ps{1}\rt X)\cdot (v\ps{2} \cdot m)+ (v\lt X)\cdot m= v\cdot (X\cdot m).
\end{equation}
Translating \eqref{proof-compatibility} via \eqref{U(g)-dual-coaction}, we observe
\begin{align}\label{proof-label-2}
\begin{split}
& (X\bullet \Db_M(m))(v)= (X\ns{0}\cdot m\ns{0})\ot (X\ns{1}m\ns{1})(v)+ m\ns{0}\ot (X\rt m\ns{1})(v)=\\
&  X\ns{1}(v\ps{1})m\ns{1}(v\ps{2}) X\ns{0}\cdot m\ns{0}+ (X\rt m\ns{1})(v) m\ns{0}=\\
&  m\ns{1}(v\ps{2}) (v\ps{1}\rt X)\cdot m\ns{0}+ (X\rt m\ns{1})(v) m\ns{0}=\\
& (v\ps{1}\rt X)\cdot(v\ps{2}\cdot m)  +m\ns{0}m\ns{1}(v\lt X)=\\
&  (v\ps{1}\rt X)\cdot (v\ps{2}\cdot m) + (v\lt X)\cdot m = v\rt (X\cdot m)=\\
&(X \cdot m)\ns{0} (X\cdot m)\ns{1}(v)= \Db_M(X\cdot m)(v).
\end{split}
\end{align}
Conversely, from  a comodule over $R(\Fg_2)$ one obtains  a locally finite module over $\Fg_2$ by  \eqref{action-coaction}. One shows that the compatibility \eqref{Lie-module-compatibility} follows from $\Db_M(X\cdot m)=X\bullet \Db_M(m)$ via   \eqref{proof-label-2}.
\end{proof}

\bigskip

Now we investigate the same correspondence  for matched pair of Lie groups.
Let $(G_1, G_2)$ be a matched pair of Lie groups. Then $(\mathbb{C}G_1, \mathbb{C}G_2)$ is a mutual pair of Hopf algebras and for $G=G_1 \bi G_2$ we have $\mathbb{C}G = \mathbb{C}G_1 \dcp \mathbb{C}G_2$. Therefore, as it is indicated by Lemma \ref{lemma-U-V-module}, a  module $M$ over  $G_1$ and $G_2$  is a module over $G$ with
\begin{align}\label{G1G2module}
(\varphi \cdot \psi) \cdot m = \varphi \cdot (\psi \cdot m), \qquad \varphi \in G_1, \psi \in G_2
\end{align}
if and only if
\begin{align}\label{compatibilityforgroups}
(\psi \rhd \varphi) \cdot ((\psi \lhd \varphi) \cdot m) = \psi \cdot (\varphi \cdot m).
\end{align}
Now let $\Fg_1$ and $\Fg_2$ be the corresponding Lie algebras of the Lie groups $G_1$ and $G_2$ respectively. We define the $\Fg_1-$module structure of $M$ as follows
\begin{align}\label{g1action}
X \cdot m = \dt exp(tX) \cdot m, \qquad \mbox{ for any } X \in \Fg_1, m \in M.
\end{align}

Assuming $M$ to be locally finite as a left $G_2-$module, we define the right coaction $\nb:M \to M \ot R(G_2)$ of $R(G_2)$ on $M$ as usual.
\begin{align}\label{RG2coaction}
\nb(m) = m\ns{0} \ot m\ns{1}, \qquad \mbox{ if and only if }, \quad \psi \cdot m = m\ns{1}(\psi)m\ns{0}.
\end{align}
In this case, we can express an infinitesimal version of the compatibility condition (\ref{compatibilityforgroups}) as follows
\begin{align}\label{infinitesimalcompatibility}
(\psi \rhd X) \cdot (\psi \cdot m) + \psi \cdot m\ns{0}(X \rhd m\ns{1})(\psi) = \psi \cdot (X \cdot m),
\end{align}
for any $\psi \in G_2$, $X \in \Fg_1$ and $m \in M$.
\\\\
 Now we state the following  proposition whose proof is similar to that of Proposition \ref{proposition-module-induced-module-Lie algebraas}.
\begin{proposition}\label{proposition-module-induced-module-Lie groups}
For a matched pair of Lie groups $(G_1,G_2)$, let $M$ be a left $G = G_1 \bi G_2-$module by (\ref{G1G2module}) such that $G_2$ acts locally finitely. Then by the $\Fg_1$ action (\ref{g1action}) and the $R(G_2)$ coaction (\ref{RG2coaction}) defined, $M$ becomes an induced $(\Fg_1, R(G_2))$-module. Conversely, every induced $(\Fg_1, R(G_2))$-module comes this way.
\end{proposition}

We conclude by discussing the case of affine algebraic groups. Let $(G_1, G_2)$ be a matched pair of affine algebraic groups and $M$ a locally finite polynomial representation of $G_1$ and $G_2$. Let us write the dual comodule structures as
\begin{align}
\bD_1:M \to M \ot \Pc(G_1), \qquad m \mapsto m\ps{0} \ot m\ps{1},
\end{align}
and
\begin{align}\label{affinePG2coaction}
\bD_2:M \to M \ot \Pc(G_2), \qquad m \mapsto m\ns{0} \ot m\ns{1}.
\end{align}
On the other hand, if we call the $G_1-$module structure as
\begin{align}
\rho:G_1 \to GL(M),
\end{align}
then we have the $\Fg_1-$module structure
\begin{align}\label{affineg1action}
\rho^{\circ}:\Fg_1 \to \Fg\Fl(M).
\end{align}
The compatibility condition is the same as the Lie group case and the main result is as follows.

\begin{theorem}
Let $M$ be a $G = G_1 \cdot G_2-$module as well as a locally finite polynomial representation of $G_1$ and $G_2$. Then via the action (\ref{affineg1action}) and the coaction (\ref{affinePG2coaction}), $M$ becomes an induced $(\Fg_1, \Pc(G_2))-$module. Conversely, any induced $(\Fg_1, \Pc(G_2))$-module comes this way.
\end{theorem}

\begin{proof}
For arbitrary $\psi \in G_2$, $X \in \Fg_1$ and any $m \in M$
\begin{align}
\begin{split}
\bD_2(X \cdot m)(\psi) = \bD_2(m\ps{0})(\psi)X(m\ps{1}) = m\ps{0}\ns{0}m\ps{0}\ns{1}(\psi)X(m\ps{1}).
\end{split}
\end{align}
On the other hand, for any $\vp \in G_1$,
\begin{align}
\begin{split}
& m\ps{0}\ns{0}m\ps{0}\ns{1}(\psi)m\ps{1}(\vp) = (\vp \cdot m)\ns{0}(\vp \cdot m)\ns{1}(\psi) = \psi \cdot (\vp \cdot m) = \\
& (\psi \rhd \vp) \cdot ((\psi \lhd \vp) \cdot m) = ((\psi \rhd \vp) \cdot m\ns{0}) m\ns{1}(\psi \lhd \vp) = \\
& m\ns{0}\ps{0}(m\ns{0}\ps{1} \lhd \psi)(\vp)m\ns{1}^{\ns{0}}(\psi)m\ns{1}^{\ns{1}}(\vp) = \\
& m\ns{0}\ps{0}m\ns{1}^{\ns{0}}(\psi)((m\ns{0}\ps{1} \lhd \psi) \cdot m\ns{1}^{\ns{1}})(\vp)
\end{split}
\end{align}
Therefore,
\begin{align}
m\ps{0}\ns{0} \ot m\ps{0}\ns{1} \ot m\ps{1} = m\ns{0}\ps{0} \ot m\ns{1}^{\ns{0}} \ot (m\ns{0}\ps{1} \lhd \psi) \cdot m\ns{1}^{\ns{1}}.
\end{align}
Plugging this result in
\begin{align}
\begin{split}
& \bD_2(X \cdot m)(\psi) = m\ns{0}\ps{0}m\ns{1}^{\ns{0}}(\psi)X((m\ns{0}\ps{1} \lhd \psi) \cdot m\ns{1}^{\ns{1}}) = \\
& m\ns{0}\ps{0}m\ns{1}^{\ns{0}}(\psi)X((m\ns{0}\ps{1} \lhd \psi))m\ns{1}^{\ns{1}}(e_1) + \\
& m\ns{0}\ps{0}m\ns{1}^{\ns{0}}(\psi)(m\ns{0}\ps{1} \lhd \psi)(e_1)X(m\ns{1}^{\ns{1}}) = \\
& m\ns{0}\ps{0}m\ns{1}(\psi)(\psi \rhd X)(m\ns{0}\ps{1}) + m\ns{0}(X \rhd m\ns{1})(\psi) = \\
& m\ns{1}(\psi)(\psi \rhd X) \cdot m\ns{0} + m\ns{0}(X \rhd m\ns{1})(\psi) = \\
& (\psi \rhd X) \cdot (\psi \cdot m) + m\ns{0}(X \rhd m\ns{1})(\psi).
\end{split}
\end{align}
Also as before, we evidently have
\begin{align}
(X \bullet \bD_2(m))(\psi) = (\psi \rhd X) \cdot (\psi \cdot m) + m\ns{0}(X \rhd m\ns{1})(\psi).
\end{align}
So,  the equality
\begin{align}
\bD_2(X \cdot m) = X \bullet \bD_2(m),\qquad m\in M,\; X\in\Fg_1,
\end{align}
 is proved.
\end{proof}

\section{ Hopf cyclic cohomology of commutative geometric Hopf algebras }
\label{S3}
In this subsection, we compute the Hopf cyclic cohomology of the commutative Hopf algebras $R(G)$, $\Pc(G)$, and $R(\Fg)$ with coefficients in a suitable comodule.

\subsection{Preliminaries about Hopf cyclic cohomology}
\label{SS-Preliminaries about Hopf cyclic cohomology}
Hopf cyclic cohomology was first defined by Connes and Moscovici as a computational tool to compute the index cocycle defined by the local index formula \cite{Connes-Moscovici-98}. The original definition of Hopf cyclic cohomology  involved only a Hopf algebra and a character with involutive twisted antipode. However, they completed  the picture in \cite{Connes-Moscovici-00} by bringing the notion of modular pair in involution which is recalled in \eqref{def-MPI}. The theory has had  rather rapidly developments in different directions in the last decade. One of the development was to enlarge the space of coefficients to arbitrary dimension and also replace the Hopf algebra with a (co)algebra upon which the Hopf algebra (co)acts \cite{Hajac-Khalkhali-Rangipour-Sommerhauser-04-2}. The suitable  coefficients for this development are defined  in \cite{Hajac-Khalkhali-Rangipour-Sommerhauser-04-1} and called SAYD modules which are also recalled in \eqref{def-SAYD}. Here, we do not recall the Hopf cyclic cohomology in its full scope, it suffices to mention the case of Hopf algebras and SAYD modules will be. Let $M$ be a right-left SAYD module over a Hopf algebra $\Hc$. Let
 \begin{equation}
 C^q(\Hc,M):= M\ot \Hc^{\ot q}, \quad q\ge 0.
 \end{equation} We recall the following operators on $C^\ast(\Hc,M)$
\begin{align*}
&\text{face operators} \quad\p_i: C^q(\Hc,M)\ra C^{q+1}(\Hc,M), && 0\le i\le q+1\\
&\text{degeneracy operators } \quad\s_j: C^q(\Hc,M)\ra C^{q-1}(\Hc,M),&& \quad 0\le j\le q-1\\
&\text{cyclic operators} \quad\tau: C^q(\Hc,M)\ra C^{q}(\Hc,M),&&
\end{align*}
by
\begin{align}
\begin{split}
&\p_0(m\ot h^1\odots h^q)=m\ot 1\ot h^1\odots h^q,\\
&\p_i(m\ot h^1\odots h^q)=m\ot h^1\odots h^i\ps{1}\ot h^i\ps{2}\odots h^q,\\
&\p_{q+1}(m\ot h^1\odots h^q)=m\sns{0}\ot h^1\odots h^q\ot m\sns{-1},\\
&\s_j (m\ot h^1\odots h^q)= m\ot h^1\odots \ve(h^{j+1})\odots h^q,\\
&\tau(m\ot h^1\odots h^q)=m\sns{0}h^1\ps{1}\ot S(h^1\ps{2})\cdot(h^2\odots h^q\ot m\sns{-1}),
\end{split}
\end{align}
where $\Hc$ acts on $\Hc^{\ot q}$ diagonally.

The graded module $C^{\ast}(\Hc,M)$  endowed with the above operators is then a cocyclic module \cite{Hajac-Khalkhali-Rangipour-Sommerhauser-04-2}, which means that $\p_i,$ $\s_j$ and $\tau$ satisfy the following identities
\begin{eqnarray}
\begin{split}
& \p_{j}  \p_{i} = \p_{i} \p_{j-1},   \hspace{35 pt}  \text{ if} \quad\quad i <j,\\
& \sigma_{j}  \sigma_{i} = \sigma_{i} \sigma_{j+1},    \hspace{30 pt}  \text{    if} \quad\quad i \leq j,\\
&\sigma_{j} \p_{i} =   \label{rel1}
 \begin{cases}
\p_{i} \sigma_{j-1},   \quad
 &\text{if} \hspace{18 pt}\quad\text{$i<j$}\\
\text{Id}   \quad\quad\quad
 &\text{if}\hspace{17 pt} \quad   \text{$i=j$ or $i=j+1$}\\
\p_{i-1} \sigma_{j}  \quad
 &\text{    if} \hspace{16 pt}\quad \text{$i>j+1$},\\
\end{cases}\\
&\tau\p_{i}=\p_{i-1} \tau, \hspace{43 pt} 1\le i\le  q\\
&\tau \p_{0} = \p_{q+1}, \hspace{43 pt} \tau \sigma_{i} = \sigma _{i-1} \tau,  \hspace{33 pt} 1\le i\le q\\ \label{rel2}
&\tau \sigma_{0} = \sigma_{n} \tau^2, \hspace{43 pt} \tau^{q+1} = \Id.
\end{split}
\end{eqnarray}

One uses the face operators to define the Hochschild coboundary
\begin{align}
\begin{split}
&b: C^{q}(\Hc,M)\ra C^{q+1}(\Hc,M), \qquad \text{by}\qquad b:=\sum_{i=0}^{q+1}(-1)^i\p_i
\end{split}
\end{align}
It is known that $b^2=0$. As a result, one obtains the Hochschild complex of the coalgebra $\Hc$ with coefficients in bicomodule $M$. Here, the right comodule defined trivially. The cohomology of $(C^\bullet(\Hc,M),b)$ is denoted by $H^\bullet_{\rm coalg}(H,M)$.

One uses the rest of the operators to define the Connes boundary operator,
\begin{align}
\begin{split}
&B: C^{q}(\Hc,M)\ra C^{q-1}(\Hc,M), \qquad \text{by}\qquad B:=\left(\sum_{i=0}^{q}(-1)^{qi}\tau^{i}\right) \s_{q-1}\tau
\end{split}
\end{align}

It is shown in \cite{C-83} (can be found also in \cite{Connes-Book}) that
 for any cocyclic module  $b^2=B^2=(b+B)^2=0$. As a result, one defines the cyclic cohomology of $\Hc$ with coefficients in   SAYD module $M$, which is  denoted by $HC^\ast(\Hc,M)$, as the total  cohomology of the bicomplex
\begin{align}
C^{p,q}(\Hc,M)= \left\{ \begin{matrix} M\ot \Hc^{\ot q},&\quad\text{if}\quad 0\le p\le q, \\
&\\
0, & \text{otherwise.}
\end{matrix}\right.
\end{align}
One also defines the periodic cyclic cohomology of $\Hc$ with coefficients in $M$, which is denoted by $HP^\ast(\Hc,M)$, as the total cohomology of direct sum total  of the following bicomplex

\begin{align}
C^{p,q}(\Hc,M)= \left\{ \begin{matrix} M\ot \Hc^{\ot q},&\quad\text{if}\quad  p\le q, \\
&\\
0, & \text{otherwise.}
\end{matrix}\right.
\end{align}

It can be seen that the periodic cyclic complex and hence cohomology is $\Zb_2$ graded.

\subsection{Lie algebra cohomology and Hopf cyclic cohomology}
\label{SS-Lie algebra cohomology and Hopf cyclic cohomology}
In this subsection, we recall the Lie algebra cohomology and relate it to the Hopf cyclic cohomology of commutative Hopf algebras.

Let $\Fg $ be a finite dimensional  Lie algebra. Let also  $\{\t^i\}$ and $\{X_i\}$ be a pair of dual basis for $\Fg^\ast$ and $\Fg$. Assume that  $V$ is a right $\Fg$-module. The Chevalley-Eilenberg complex of the $(\Fg,V)$ is defined by
\begin{equation}
\xymatrix{V\ar[r]^{\p_0}&C^1(\Fg,V)\ar[r]^{\p_1}& C^2(\Fg,V)\ar[r]^{\p_2}&\cdots\;,}
\end{equation}
where $C^n(\Fg,V)=\Hom(\wdg^q\Fg^\ast ,V)$ is the vector space of all alternating linear maps on $\Fg^{\ot q}$ with values in $V$.
 \begin{align}
 \begin{split}
& \p_q(\om)(X_0, \ldots,X_q)=\sum_{i<j} (-1)^{i+j}\om([X_i,X_j], X_0\ldots \widehat{X_i}, \ldots, \widehat{X_j}, \ldots, X_q)+\\
& ~~~~~~~~~~~~~~~~~~~~~~~\sum_{i}(-1)^i\om(X_0,\ldots,\widehat{X_i},\ldots X_q)X_i.
\end{split}
 \end{align}

Alternatively, one identifies  $C^q(\Fg,V)$ with $V\ot \wdg^q\Fg^\ast$ and the coboundary $\p_\bullet$ with the following one
\begin{align}
\begin{split}\label{equivalent-CE}
&\p_0(v)= vX_i\ot \t^i,\\
&\p_q(v\ot \om)= vX_i\ot \t^i\wdg\om+ v\ot \p_{\rm dR}(\om).
\end{split}
\end{align}
Here $\p_{dR}: \wdg^q\Fg^\ast\ra \wdg^{q+1}\Fg^\ast$ is the de Rham coboundary which is a derivation of degree $1$  and  recalled here by
\begin{align}
\p_{\rm dR}(\t^k)=\frac{1}{2} C^k_{i,j}\t^i\wdg\t^j.
\end{align}
We denote the cohomology of $(C^\bullet(\Fg,V),\p)$ by $H^\bullet(\Fg,V)$ and refer to it as the Lie algebra cohomology of $\Fg$ with coefficients in $V$.
For a Lie subalgebra $\Fh\subseteq \Fg$ one defines the relative cochains by
\begin{equation}
C^q(\Fg,\Fh,V)=\{\om\in C^q(\Fg,V) |\, \quad \i_X\om=\Lc_X(\om)=0, \quad X\in \Fh\},
\end{equation}
where,
\begin{align}
&\i_X(\om)(X_1,\ldots,X_{q})=\om(X,X_1,\ldots,X_q), \\
&\Lc_X(\om)(X_1,\ldots,X_{q})=\\\notag
&\sum(-1)^i\om([X,X_i], X_1,\ldots, \widehat{X_i},\ldots,X_q)+\om(X_1,\ldots,X_q)X.
\end{align}
One also identifies $C^q(\Fg,\Fh,V)$ with $\Hom_{\Fh}(\wdg^q(\Fg/\Fh),V)$ which is $(V\ot \wdg^q(\Fg/\Fh)^\ast)^\Fh$, where the action of $\Fh$ on $\Fg/\Fh$ is induced  by the  adjoint action of $\Fh$ on $\Fg$.

It is checked in \cite{CE} that the Chevalley-Eilenberg coboundary  $\pi_\bullet$ is well defined on $C^n(\Fg,\Fh,V)$. We denote the cohomology of $(C^\bullet(\Fg,\Fh,V),\p_\bullet)$ by $H^\bullet(\Fg,\Fh,V)$ and refer  to it as the relative Lie algebra cohomology of $\Fh\subseteq \Fg$ with coefficients in $V$.

Let a coalgebra $C$ and an algebra $A$  be in duality, i.e, there is a pairing between $\langle\; ,\rangle: C\ot A\ra \Cb$ compatible with product and coproduct  i.e,
\begin{equation}\label{pairing-property}
\langle c,ab\rangle= \langle c\ps{1},a\rangle\langle c\ps{2},b\rangle, \quad \langle c,1\rangle=\ve(c).
\end{equation}
By using the duality, one turns any bicomodule $V$ over $C$ into a bimodule over $A$ via
\begin{equation}
av:= \langle v\ns{1},a\rangle v\ns{0}, \quad va:=\langle v\ns{-1},a\rangle v\ns{0}.
\end{equation}

Now we define the following map
\begin{align}\label{map-theta-Alg}
 \begin{split}
&\t_{(C,A)}: V\ot C^{\ot q}\ra \Hom(A^{\ot q},V),\\
&\t_{(C,A)}(v\ot c^1\odots c^q)(a_1 \odots a_q)= \langle c^1, a_1\rangle \cdots \langle c^q, a_q\rangle v.
\end{split}
 \end{align}

\begin{lemma} For any algebra $A$ , coalgebra $C$ with a pairing and any $C$-bicomodule $V$, the map $\t_{(C,A)}$ defined in \eqref{map-theta-Alg} is a map of complexes between Hochschild complex of the coalgebra $C$ with coefficients in the bicomodule $V$ and the Hochschild complex of the algebra $A$ with coefficients in the $A$-bimodule induced by $V$.
\end{lemma}
\begin{proof}
The proof is elementary and uses only the pairing property \eqref{pairing-property}.
\end{proof}

Now let  $\Fc$ be a commutative Hopf algebra with a Hopf pairing with $U(\Fg_2)$, the enveloping Hopf algebra of some Lie algebra $\Fg_2$
 \begin{equation}\langle,\rangle:\Fc\ot U(\Fg)\ra \Cb,
  \end{equation}
  satisfying \eqref{pairing-property} and
  \begin{align}
  &\langle f^1f^2,u\rangle= \langle f^1,u\ps{1}\rangle\langle f^2,u\ps{2}\rangle,\quad  \langle 1,u\rangle= \ve(u),\\
 & \langle S(f), u\rangle= \langle f,S(u)\rangle.
  \end{align}

   In addition, we assume that $\Fg_2=\Fh\ltimes \Fl$, where the Lie subalgebra  $\Fh$ is reductive and every finite dimensional representation of $\Fh$ is semisimple, and $\Fl$ is an ideal of $\Fg$.

   For a $\Fg_2$-module $V$, we observe that the Lie algebra inclusion $\Fl\hookrightarrow \Fg_2$ induces a map of Hochschild complexes, where $\Fl$ acts on $V$ by restriction of the action of $\Fg_2$
   \begin{equation}
   \pi_\Fl: \Hom(U(\Fg_2)^{\ot q},V)\ra \Hom(U(\Fl)^{\ot l}, V)
   \end{equation}

   One uses the antisymmetrization map
   \begin{align}\label{map-antisymmetrization}
   &\a: \Hom(U(\Fl)^{\ot q},V)\ra C^q(\Fl,V):= V\ot \wdg^{q}\Fl^\ast,\\\notag
    &\a(\om)(X_1,\ldots,X_q)=\sum_{\s\in {S_q}}(-1)^\s \om(X_{\s(1)},\ldots,X_{\s(q)}).
   \end{align}
   It is known that $\a$ is a map of complexes between Hochschild cohomology of $U(\Fl)$ with coefficients in $V$ and the Lie algebra cohomology of $\Fl$ with coefficients in $V$.

  One then uses the fact that $\Fh$ acts  semisimply  to decompose the complex $C^\bullet(\Fl,V)$ into the weight spaces
   \begin{equation}
   C^\bullet(\Fl,V)= \bigoplus _{\mu\in \Fh^\ast}C^\bullet_\mu(\Fl,V).
   \end{equation}

   Since  $\Fh$ acts on  $\Fl$ by derivations, one observes that each $C_\mu^\bullet(\Fl,V)$ is  a complex for its own and the projection $\pi^\mu: C^{\bullet}(\Fl,V)\ra C^{\bullet}_\mu(\Fl,V)$ is a map of complexes. Composing $\t_{\Fc,U(\Fg)}$, $\pi_\Fl$ and $\pi^\mu$ we get a map of complexes
   \begin{equation}
   \t_{\Fc,U(\Fg),\Fl,\mu}:=\pi_\mu\circ\pi_\Fl\circ\t_{\Fc,U(\Fg)}: C^\bullet_{\rm coalg}(\Fc,V)\ra C^\bullet_\mu(\Fl,V).
   \end{equation}

   \begin{definition}
   Let a commutative Hopf algebra $\Fc$ be in a Hopf pairing with $U(\Fg)$, the enveloping Hopf algebra of $\Fg$. A decomposition of Lie algebras $\Fg=\Fh\ltimes \Fl$ is  called a $\Fc$-Levi decomposition if the map $\t_{\Fc,\Fl,\mu}$ is a quasi isomorphism for $\mu=0$ and any $\Fc$-comodule $V$.
   \end{definition}

\begin{theorem}\label{Theorem-F-Levi}
Let a commutative Hopf algebra $\Fc$ be in duality with the enveloping Hopf algebra of a Lie algebra $\Fg$, and assume that $\Fg=\Fh\ltimes \Fl$ is an $\Fc$-Levi decomposition. Then the map  $\t_{\Fc,\Fl,0}$ induces an isomorphism between Hopf cyclic cohomology of $\Fc$ with coefficients in $V$ and the relative Lie algebra cohomology of $\Fh\subseteq \Fg$ with coefficients in $V$. In other words,
\begin{equation}
HP^\bullet(\Fc,V)\cong \bigoplus_{\bullet=i\;\text{mod 2}}H^i(\Fg,\Fh,V).
\end{equation}
\end{theorem}
\begin{proof}
 First, one uses \cite{Khalkhali-Rangipour-03} to observe that for any commutative Hopf algebra $\Fc$ and trivial comodule, the Connes boundary $B$ vanishes in the level of Hochschild cohomology. The same proof works for any comodule and hence we have
\begin{equation}
HP^\bullet(\Fc,V)\cong \bigoplus_{\bullet= k\;\text{mod 2}} H^{i}_{\rm coalg}(\Fc,V).
\end{equation}
Since $\Fg=\Fh\ltimes \Fl$ is assumed to be a $\Fc$-Levi decomposition, the  map of complexes  $\t_{\Fc,\Fl,0}$ induces an isomorphism in the level of cohomologies.
\end{proof}

\subsection{Hopf cyclic cohomology of $R(G)$}
\label{SS-Hopf cyclic cohomology of R(G)}
In this subsection, we compute the Hopf cyclic cohomology of the commutative Hopf algebra $\Fc:=R(G)$, the Hopf algebra of all representative functions on a Lie group $G$,  with coefficients in a right comodule $V$ over $\Fc$. Indeed, let $V$ be a right comodule over $\Fc$, or equivalently a {\it representative } left $G$-module. Let us recall from \cite{Hochschild-Mostow-62} that a representative $G$-module is a locally finite $G$-module such that for any finite-dimensional $G$-submodule $W$ of $V$, the induced representation of $G$ on $W$ is continuous. The representative $G$-module  $V$ is called  {\it representatively injective } if for every exact sequence
\begin{equation}
\xymatrix{ 0\ar[r]& A\ar[r]^\rho \ar[d]_{\a}& B\ar[r]\ar@{.>}[dl]^{\b}& C\ar[r]&0\\
&V&&&}
\end{equation}
of $G$-module homomorphisms between representative $G$-modules $A$, $B$, and $C$, and for every $G$-module homomorphism $\a: A\ra V$, there is a $G$-module homomorphism  $\b:B\ra V$ such that $\b\circ\rho=\a$. A {\it representatively injective resolution }of the representative  $G$-module $V$ is an exact sequence of $G$-module homomorphisms
\begin{equation}
\xymatrix{ 0\ar[r]& X_0\ar[r] & X_1\ar[r]& \cdots}\;,
\end{equation}
where each $X_i$ is a representatively injective  $G$-module. The {\it representative cohomology group} of $G$ with value in the representative $G$-module $V$ is then defined to be the cohomology of
\begin{equation}
\xymatrix{ X_0^G\ar[r] & X_1^G\ar[r]& \cdots}\;,
\end{equation}
where $X_i^G$ are the elements of $X_i$ which are fixed by $G$.  We denote this cohomology by $H^\ast_{\rm rep}(G,V)$.
\begin{proposition}\label{Propositio-rep-coalgebra}
For any Lie group $G$ and any representative module $V$, the representative cohomology groups of $G$ with value in $V$ coincide with the coalgebra cohomology groups of the coalgebra $R(G)$ with coefficients in the induced comodule by $V$.
\end{proposition}
\begin{proof}
 In \cite{Hochschild-Mostow-61} it is shown  that
\begin{equation}
\xymatrix{ V\ar[r]^{d_{-1}~~} &V\ot \Fc\ar[r]^{d_0~~}&V\ot \Fc^{\ot 2}\ar[r]^{~~d_1}&\cdots }\;,
\end{equation}
is a representatively injective resolution for the representative $G$-module $V$.
Here $G$ acts on $V\ot \Fc^{\ot n}$ by
\begin{equation}
\g(v\ot f^1\odots f^q)= \g v\ot f^1\cdot \g^{-1}\odots f^q\cdot \g^{-1},
\end{equation}
where $G$ acts on $\Fc$ by right translation. One look at $V\ot \Fc^{\ot q}$ as a group cochain with value in the $G$-module $V$ by embedding $V\ot \Fc^{\ot q}$ into $F(G^{\times q},V)$, the space of all continuous maps from $\underset{q~ {\rm times}}{\underbrace{G\times\cdots\times G}}$ to $V$,  by
 \begin{equation}
 (v\ot f^1\odots f^q)(\g_1,\ldots, \g_q)=f^1(\g_1)\cdots f^q(\g_q)v.
 \end{equation}
The coboundaries $d_i$ are defined by
\begin{align}
\begin{split}
&d_{-1}(v)(\g)=v,\\
&d_i(\phi)(\g_1,\ldots,\g_{q+1})= \sum_{i=0}^{q+1} \phi(\g_0,\ldots, \widehat{\g}_i,\ldots, \g_{q}).
\end{split}
\end{align}
One then identifies $(V\ot \Fc^{\ot q})^G$ with $V\ot \Fc^{\ot q-1}$ by
\begin{equation}
v\ot f^1\odots f^q\mapsto \ve(f^1)v\ot f^2 f^3\ps{1}\cdots f^q\ps{1}\ot f^3\ps{2}\cdots f^q\ps{2}\odots f^{q-1}\ps{q}f^q\ps{q}\ot f^q\ps{q+1}.
\end{equation}

The complex of the $G$-fixed part of the resolution is
 \begin{equation}
\xymatrix{  V\ar[r]^{\d_0~~} &V\ot \Fc\ar[r]^{~~\d_1}&\cdots}\;,
\end{equation}
where the coboundaries $\d_i$  are defined by
\begin{align*}
\begin{split}
&\d_0: V\ra V\ot \Fc, \quad \d(v)= v\ns{0}\ot v\ns{1}-v\ot 1,\\
&\d_i: V\ot\Fc^{\ot q}\ra V\ot \Fc^{\ot q+1},\\
&\d_i(v\ot f^1\odots f^q)= v\ns{0}\ot v\ns{1}\ot f^1\odots f^q+\\
&\sum (-1)^i v\ot f^1\odots f^i\ps{1}\ot f^i\ps{2}\odots f^q + (-1)^{q+1}v\ot f^1\odots f^q\ot 1.
\end{split}
\end{align*}
which is the complex who computes the coalgebra cohomology of $\Fc$ with coefficients in $\Fc$-comodule in $V$.
\end{proof}
One of course  has the version of the above proposition for right $G$-module and  corresponding  left $\Fc$-comodule.

\medskip

Let us recall from \cite{Hochschild-Mostow-61} that a {\it nucleus} of a Lie group $G$ is  a simply connected solvable closed normal Lie subgroup $L$ of $G$ such that $G/L$ is reductive. It means that $G/L$ has  a faithful representation and every finite dimensional analytic representation of $G/L$ is semisimple.    In this case one proves that $G=S\ltimes L$, where $S:=G/L$ is reductive.  Let, in addition,  $\Fs\subseteq\Fg$ be Lie algebras of  $S$ and $G$ respectively.

For a representative $G$-module $V$ one defines the following map.
\begin{align}\label{map-D}
\begin{split}
&\Dc_{\rm Gr}: V\ot \Fc^{\ot q}\ra C^q(\Fg,\Fh,V),\\
&\Dc_{\rm Gr}(v\ot f^{1}\odots f^q)(X_1,\ldots,X_q)=\\
 &~~~~~~~~~~\sum_{\mu\in S_q}(-1)^\mu\mdt{1}\cdots\;\mdt{q}f^1(exp(t_1X_{\mu(1)}))\cdots f^q(exp(t_qX_{\mu(q)})v.
\end{split}
\end{align}

\begin{theorem} \label{Theorem-R(G)-cohomology}Let $G$ be a  Lie group, $V$ a representative $G$-module, $L$ a  nucleus of $G$ and $\Fs\subset \Fg$ the Lie algebras of $S:=G/L$ and $G$ respectively. Then the map $\Dc$ defined in \eqref{map-D} induces an isomorphism of vector spaces between the periodic Hopf cyclic cohomology of $R(G)$, the Hopf algebra of all representative functions on $G$, with coefficients in the comodule induced by $V$, and the relative Lie algebra cohomology of $\Fs\subseteq \Fg$ with coefficients in  $\Fg$-module induced by $V$. In other words,
\begin{equation}
HP^\ast(R(G)\;,\;V)\cong \bigoplus_{\ast=i\;\;{\rm mod \; 2}}H^i(\Fg,\Fs,V).
\end{equation}
\end{theorem}
\begin{proof}
 One knows that $R(G)$ and $U(\Fg)$ are in Hopf duality via
   \begin{equation}
   \langle f,X\rangle =\dt f(exp(tX)), \quad X\in\Fg, f\in R(G).
   \end{equation}
   On the other hand it is easy to see that $\Dc_{\rm Gr}$ is nothing but $\t_{R(G),U(\Fg),\Fl,0}$  and hence, by  applying   Proposition \ref{Propositio-rep-coalgebra}, a map of complexes between complex of representative group cochains of $G$ with value in $V$ and Chevalley-Eilenberg complex of the Lie algebras $\Fs\subseteq\Fg$ with  coefficients in the $\Fg$-module induced by $V$. With a slight modification of the same proof as  in \cite[Theorem 10.2]{Hochschild-Mostow-62}, one shows that $\Dc_{\rm Gr}$ induces a quasi-isomorphism. So $\Fg=\Fs\ltimes\Fl$ is a $R(G)$-Levi decomposition and hence the rest follows from Theorem \ref{Theorem-F-Levi}.
\end{proof}

\subsection{Hopf cyclic cohomology of $R(\Fg)$}
\label{SS-Hopf cyclic cohomology of R(g)}
Let $\Fg$ be a  Lie algebra and $R(\Fg)$  the commutative Hopf algebra of representative functions on $U(\Fg)$ recalled in Subsection  \ref{SS-Lie algebra}. Let $V$  be  a locally finite $\Fg$-module or equally an $R(\Fg)$-comodule. In this subsection we compute the Hopf cyclic cohomology of $R(\Fg)$ with coefficients in $V$. To this end,  we need some new cohomology theory similar to the representative cohomology of Lie groups.  We assume all modules are locally finite.
The $\Fg$-module  $V$ is called  {\it injective } if for every exact sequence
\begin{equation}
\xymatrix{ 0\ar[r]& A\ar[r]^\rho \ar[d]_{\a}& B\ar[r]\ar@{.>}[dl]^{\b}& C\ar[r]&0\\
&V&&&}
\end{equation}
of $\Fg$-module homomorphisms between $\Fg$-modules $A$, $B$, and $C$, and for every $\Fg$-module homomorphism $\a: A\ra V$, there is a $\Fg$-module homomorphism  $\b:B\ra V$ such that $\b\circ\rho=\a$. An {\it injective resolution } of the $\Fg$-module $V$ is an exact sequence of $\Fg$-module homomorphisms
\begin{equation}
\xymatrix{ 0\ar[r]& X_0\ar[r] & X_1\ar[r]& \cdots}\;,
\end{equation}
where each $X_i$ is an injective  $\Fg$-module. The {\it representative cohomology groups} of $G$ with value in the $\Fg$-module $V$ is then defined to be the cohomology of
\begin{equation}
\xymatrix{ X_0^\Fg\ar[r] & X_1^\Fg\ar[r]& \cdots}\;,
\end{equation}
where $X_i^\Fg$ are the the coinvariant elements of $X_i$ i.e,
 \begin{equation}
 X_i^\Fg=\{\xi\in X_i\mid X\xi=0\quad \text{ for all}\;\; X\in \Fg \}.
 \end{equation}
   We denote this cohomology by $H^\ast_{\rm rep}(\Fg,V)$.

   Since the category of locally finite $\Fg$-modules and the category of $R(\Fg)$-comodules are equivalent,  we conclude that $H^\ast_{\rm rep}(\Fg,V)$ is the same as ${\rm Cotor}^\ast_{R(\Fg)}(V,\Cb)$ which is by definition   $H^\ast_{\rm calg}(R(\Fg),V)$.

Let $\Fl$ be the solvable radical of $\Fg$, i.e, $\Fl$ is the unique maximal solvable ideal of $\Fg$. Levi decomposition implies that  $\Fg=\Fs\ltimes \Fl$, where $\Fs$ is a  semisimple subalgebra of $\Fg$ called a Levi subalgebra .

We now consider the following map
\begin{align}\label{map-D-alg}
\begin{split}
&\Dc_{\rm Alg}: V\ot R(\Fg)^{\ot q}\ra (V\ot\wedge^{q} \Fl^\ast)^\Fs,\\
&\Dc_{\rm Alg}(v\ot f^1\odots f^q)(X_1,\ldots,X_q)= \sum_{\s\in S_q}(-1)^\s f^1(X_{\s(1)})\dots f^q(X_{\s(q)})v.
\end{split}
\end{align}

\begin{proposition}\label{Proposition-R(g)-cohomology}
Let $\Fg$ be a Lie algebra with $\Fg=\Fs\ltimes \Fl$ as a   Levi decomposition.  Then  for any finite dimensional $\Fg$-module $V$, the map $\Dc_{\rm Alg}$ induces an isomorphism between the representative cohomology of $H_{\rm rep}(\Fg,V)$ and the relative  Lie algebra cohomology $H(\Fg,\Fs,V)$.
\end{proposition}
\begin{proof}
 First one notes that $\Dc_{\rm Alg}$ induces a map of complexes. Now one lets $G$ be the  simply connected Lie group of the Lie algebra $\Fg$. The Levi decomposition  $\Fg=\Fs\ltimes \Fl$ induces a nucleus of $G$ and  $G=S\ltimes L$. Since $G$ is simply connected the representation of $\Fg$ and representation of $G$ coincides and any injective resolution of $\Fg$ is induced by an representatively injective  resolution of $G$. It means that the obvious map $H_{\rm rep}(G,V)\ra H_{\rm rep}(\Fg,V)$ is surjective. Since $V$ is finite dimensional , $\Dc_{Gr}: H^\ast_{\rm rep(G,V)}\ra H(\Fg,\Fs,V)$ is an isomorphism and it factors through $\Dc_{\rm Alg}: H_{\rm rep}(\Fg,V)\ra H^\ast(\Fg,\Fs,V)$, the latter map is an isomorphism.
\end{proof}

Finally we summarize the result of this subsection as the following theorem.

\begin{theorem}\label{theorem-cohomology-R(g)}  Let $\Fg$ be a finite dimensional Lie algebra with a   Levi decomposition $\Fg=\Fs\ltimes \Fl$. Then for any   finite dimensional $\Fg$-module $V$, the map $\Dc_{\rm Alg}$ defined in \eqref{map-D-alg} induces an isomorphism of vector spaces between the periodic Hopf cyclic cohomology of $R(\Fg)$, the Hopf algebra of all representative functions on $\Fg$, with coefficients in the comodule induced by $V$, and the relative Lie algebra cohomology of $\Fs\subseteq \Fg$ with coefficients in $V$. In other words,
\begin{equation}
HP^\ast(R(\Fg)\;,\;V)\cong \bigoplus_{\ast=i\;\;{\rm mod \; 2}}H^i(\Fg,\Fs,V).
\end{equation}
\end{theorem}
\begin{proof}
We know that $R(\Fg)$ and $U(\Fg)$ are in (nondegenerate) Hopf duality via
   \begin{equation}
   \langle f,u\rangle = f(u), \quad u\in U(\Fg), f\in R(\Fg).
   \end{equation}
   On the other hand, it is easy to see that $\Dc_{\rm Alg}$ is  $\t_{R(\Fg),U(\Fg),\Fl,0}$  and hence, by  applying   Proposition \ref{Propositio-rep-coalgebra}, is a map of complexes between the complex of representative group cochains of $G$ with value in $V$ and the Chevalley-Eilenberg complex of the Lie algebras $\Fs\subseteq\Fg$ with coefficients in the $\Fg$-module $V$. By Proposition \ref{Proposition-R(g)-cohomology},   $\Fg=\Fs\ltimes\Fl$ is a $R(\Fg)$-Levi decomposition. Hence the proof is completed  by applying Theorem \ref{Theorem-F-Levi}.
\end{proof}

\subsection{Hopf cyclic cohomology of $\Pc(G)$}
\label{SS-Hopf cyclic cohomology of P(G)}
In this section, we compute the Hopf cyclic cohomology of $\Pc(G)$, the Hopf algebra of all complex polynomial functions of an affine algebraic group $G$ over $\Cb$.

Let $V$ be a finite dimensional polynomial right $G$-module. One defines the polynomial  group cohomology of $G$ with coefficients in $V$ as the cohomology of the following complex
\begin{equation}
\xymatrix{ C^0_{\rm pol}(G,V)\ar[r]^\d& C^1_{\rm pol}(G,V)\ar[r]^{~~~\d} &\cdots }
\end{equation}
Here $C^0_{\rm pol}(G,V)=V$, and

\begin{equation}
C^q_{\rm pol}(G,V)=\{ \phi:\underset{q\; {\rm times}}{\underbrace{G\times \dots \times G}}\ra V\mid \phi\;\; \text{is polynomial}\}
\end{equation}
 The coboundary $\d$ is the usual group cohomology coboundary which is recalled here by
 \begin{align}
 \begin{split}
 &\d:V\ra C^1_{\rm pol}(G,V), \quad \d(v)(\g)=v- v\cdot \g,\\
 &\d: C^q_{\rm pol}(G,V)\ra C^{q+1}_{\rm pol}(G,V), \\
& \d(\phi)(\g_1,\ldots,\g_{q+1})=\d(\phi)(\g_2,\ldots,\g_{q+1})+\\
 &\sum_{i=1}^q (-1)^i \phi(\g_1, \ldots, \g_i\g_{i+1}, \ldots, \g_{q+1})+ (-1)^{q+1}\phi(\g_1,\ldots, \g_q)\cdot\g_{q+1}
 \end{split}
 \end{align}

 One identifies $C^q_{\rm pol}(G,V)$ with $V\ot \Pc(G)^{\ot q}$ via
 \begin{equation}
 v\ot f^1\odots f^q(\g_1, \ldots,\g_q)= f^1(\g_1)\cdots f^q(\g_q)v.
 \end{equation}

 and observe that the coboundary $\d$ is identified with the Hochschild coboundary of the coalgebra $\Pc(G)$ with value in the bicomodule $V$, where the right comodule is trivial and the left comodule is induced by the right $G$-module.

 The  cohomology $(C^\ast_{\rm pol}(G,V),\d)$ is denoted by $H_{\rm pol}(G,V)$.
 One notes that $H_{\rm }(G,V)$ can be also calculated by the means of {\it polynomially injective resolutions} \cite{Hochschild-61}. Let us recall here {\it polynomially injective resolutions}. A  polynomial module $V$ over an affine algebraic group $G$ is called {\it polynomially  injective} if   for any exact sequence of polynomial modules over $G$
\begin{equation}
\xymatrix{ 0\ar[r]& A\ar[r]^\rho \ar[d]_{\a}& B\ar[r]\ar@{.>}[dl]^{\b}& C\ar[r]&0\\
&V&&&}
\end{equation}
 there is a polynomial $G$-module homomorphism  $\b:B\ra V$ such that $\b\circ\rho=\a$.
  A {\it polynomially injective  resolution} for a polynomial module $V$ over $G$ is an exact sequence of {polynomially injective} modules over $G$
    \begin{equation}
\xymatrix{ 0\ar[r]&V\ar[r]& X_0\ar[r] & X_1\ar[r]& \cdots}\;,
\end{equation}

 It is shown in \cite{Hochschild-61} that the $H_{\rm pol}(G,V)$ is computed by the cohomology of the $G$-fixed part of any {\it polynomially injective resolution}  i.e., the following complex
   \begin{equation}
 \xymatrix{  X_0^G\ar[r] & X_1^G\ar[r]& \cdots}\;.
\end{equation}
 The most natural  polynomially injective  resolution of a polynomial $G$-module $V$ is the resolution of differential polynomial forms with value in $V$ which is $V\ot \wdg^\bullet(\Fg^\ast)\ot \Pc(G)$, where $G$ acts by  $\g\cdot (u\ot f)= u\ot (f\cdot \g^{-1})$, and $G$ acts on $\Pc(G)$ by right translations. This yields the following map of complexes

\begin{align}\label{map-D-pol}
\begin{split}
&\Dc_{\rm Pol}: C^q_{\rm pol}(G,V)\ra C^q(\Fg,\Fg^{\rm red},V),\\
&\Dc_{\rm Pol}(v\ot f^1\odots f^q)(X_1,\ldots, X_q)= \\
&~~~~~~~~~~~~~~~~~~~~~~~~~~\sum_{\s\in S_q} (-1)^\s (X_{\s(1)}\cdot f^1)(e)\cdots (X_{\s(q)}\cdot f^q)(e)\;v\;.
\end{split}
\end{align}

Here, we identify $\Fg$ the Lie algebra of $G$ by the  left $G$-invariant derivations of $\Pc(G)$.

Here $G= G^{\rm red}\ltimes G^{\rm u}$, is a Levi decomposition of an affine algebraic group $G$, where $G^{\rm u}$ is the unipotent radical of $G$ and $G^{\rm red}$ is the maximal reductive subgroup of $G$. We also use   $\Fg^{\rm red}$ and $\Fg^{\rm u}$  to denote the Lie algebra of $G^{\rm red}$ and $G^{\rm u}$ respectively.

\begin{theorem}\label{theorem-cohomology-P(G)}
Let $G$ be an affine algebraic group over $\Cb$ and $V$ be a finite dimensional polynomial $G$-module. Let $G= G^{\rm red}\ltimes G^{\rm u}$ be a Levi decomposition of $G$ and $\Fg^{\rm red}\subseteq \Fg$ be the Lie algebras of $G^{\rm red}$ and $G$ respectively. Then the map $\Dc_{\rm Pol}$ defined in \ref{map-D-pol} induces and isomorphism between Hopf cyclic cohomology of $\Pc(G)$, the Hopf algebra of polynomial functions on $G$,  with coefficients in the comodule induced by $V$, and the  Lie algebra cohomology of $\Fg$ relative to $\Fg_{\rm red}$ with coefficients in the $\Fg$-module $V$. In other words
\begin{equation}
HP^\bullet(\Pc(G),V)\cong \bigoplus_{i= \bullet,\text{mod} 2} H^i(\Fg,\Fg^{\rm red},V)
\end{equation}

\end{theorem}

\begin{proof}
It is shown in \cite{Hochschild-61} that $V\ot \wdg^\bullet\Fg^\ast\ot \Pc(G)$ is a {\it polynomially injective resolution} for $V$. The comparison between this resolution and the standard resolution, i.e. $V\ot \Pc(G)^{\ast+1}$, results the map $\Dc_{\rm Pol}$.  The complete proof which shows that  the map $\Dc_{\rm Pol}$ is an isomorphism between  $H^\bullet_{\rm pol}(G,V)$ and $H^\bullet(\Fg,\Fg_{\rm red},V)$ is Theorem 2.2 in \cite{KN}.  On the other hand
 the map $\Dc_{\rm pol}$ equals to $\t_{\Pc(G),U(\Fg),\Fg^{\rm u},0}$, where  we naturally pair $\Pc(G)$ and $U(\Fg)$ by
\begin{equation}
\langle v\,,\, f\rangle = v\cdot f(e).
\end{equation}
This shows that $\Fg=\Fg^{\rm red}\ltimes \Fg^{\rm u}$ is a $\Pc(G)$-Levi decomposition. One then applies the Theorem \ref{Theorem-F-Levi} to finish the proof.
\end{proof}

\section{Hopf cyclic cohomology of noncommutative geometric Hopf algebras}
\label{S4}
We use the machinery developed for computing the Hopf cyclic cohomology of bicrossed product Hopf algebra by Moscovici  and the first author in \cite{Moscovici-Rangipour-09,Moscovici-Rangipour-011} to compute the Hopf cyclic cohomology of  the geometric bicrossed product Hopf algebras we constructed in Subsections \ref{SS-Lie algebra}, \ref{SS-Lie group}, and \ref{SS-algebraic group}.
Since most of the improvements done in \cite{Moscovici-Rangipour-011} are for special cases, we need  first to advance the  machinery  to cover the case of Lie-Hopf algebras in general.

\subsection{Bicocyclic module associated to Lie Hopf algebras}
\label{SS-Bicocyclic module associated to Lie Hopf algebras}

 Let $\Fg$ be a Lie algebra and  $\Fc$ a commutative $\Fg$-Hopf algebra. We denote the bicrossed product Hopf algebra $\Fc\acl U(\Fg)$ by $\Hc$.
 Let the character $\d$ and the group-like $\s$ be the canonical modular pair in involution defined  in \eqref{delta} and \eqref{sigma}. In addition, let $M$ be an induced $(\Fg,\Fc)$-module and $^\s{M}_\d$ be the associated  SAYD module over $\Hc$ defined in \eqref{action-SAYD} and \eqref{coaction-SAYD}.

The Hopf algebra $\Uc:=U(\Fg)$ admits the following right
 action on $\;^\s{M}_\d\ot \Fc^{\ot q}$, which plays
a key role in the definition of the next bicocyclic module:
\begin{equation}
(m\ot \td f)u= \d_\Fg(u\ps{1}) S(u\ps{2})\cdot m\ot S(u\ps{3})\bullet \td f,
\end{equation}

Where $\td f:= f^1\odots f^q$, and the left action of $\Uc$ on $\Fc^{\ot q}$ is defined by
\begin{align} \label{bullet}
\begin{split}
& u\bullet( f^1\odots f^n):=\\
& u\ps{1}\ns{0}\rt f^1\ot u\ps{1}\ns{1}u\ps{2}\ns{0}\rt f^2\odots
u\ps{1}\ns{n-1}\dots u\ps{n-1}\ns{1} u\ps{n}\rt  f^n.
\end{split}
 \end{align}

One then defines a bicocyclic module $C^{\bullet,\bullet}(\Uc,\Fc,M)$, where
\begin{equation}
C^{p,q}(\Uc,\Fc,^s{M}_\d):= ^s{M}_\d\ot \Uc^{\ot p}\ot \Fc^{\ot q}, \qquad p,q\ge 0,
\end{equation}

whose horizontal morphisms are given by
\begin{align}
\begin{split}
&\hd_0(m\ot \td{u}\ot \td{f})= m\ot 1\ot u^1\ot\dots\ot u^p\ot  \td{f}\\
&\hd_j(m\ot \td{u}\ot \td{f})= m\ot u^1\ot\dots\ot
\Delta(u^i)\ot\dots \ot u^p\ot  \td{f}\\
&\hd_{p+1}(m\ot \td{u}\ot \td{f})=m\ot u^1\ot \dots \ot
u^p\ot 1\ot \td{f}\\
 &\hs_j(m\ot \td u\ot \td f)=m\ot u^1\ot \dots \ot \epsilon(u^{j+1})\ot\dots\ot u^p\ot
 \td{f}\\
&\hta(m\ot \td{u}\ot  \td{f})=\\
 &\d_\Fg(u\ps{1})S(u\ps{2})\cdot m\ot S(u^1\ps{4})\cdot(u^2\ot\dots\ot u^p\ot 1)\ot S(u^1\ps{3})\bullet \td{f},
 \end{split}
\end{align}
while the vertical morphisms
 are defined by
\begin{align}
\begin{split}
&\vd_0(m\ot \td{u}\ot \td{f})= m\ot \td u\ot  1\ot \td{f},\\
 &\vd_j(m\ot \td{u}\ot \td{f})= m \ot \td u \ot    f^1\odots\Delta(f^j)\odots f^q,\\
&\vd_{q+1}(m\ot \td{u}\ot \td{f})=m\ns{0} \ot \td u\ns{0}\ot \td{f}\ot S(\td u\ns{1}m\ns{1})\s,\\
 &\vs_j(m\ot \td u\ot  \td f)=m\ot \td u\ot  f^1\odots \epsilon(f^{j+1})\odots f^n, \\
 &\vta(m\ot \td{u}\ot  \td{f})=\\
 &m\ns{0}\ot \td u\ns{0} \ot  S(f^1)\cdot(f^2\odots f^n\ot S(\td u\ns{1}m\ns{1})\s) .
  \end{split}
\end{align}
One notes that, by definition,  a bicocyclic module is  a bigraded module whose  rows and columns form  cocyclic modules which have their own Hochschild coboundary and Connes boundary maps. These boundaries and coboundaries are denoted by $\hB$, $\vB$, $\hb$, and $\vb$, which are demonstrated in the following diagram. We refer the reader to \cite{Moscovici-Rangipour-09,Moscovici-Rangipour-011} for details on the construction of $C^{\bullet,\bullet}(\Uc,\Fc, ^\s{M}_\d)$.

\begin{align}\label{UF}
\begin{xy} \xymatrix{  \vdots\ar@<.6 ex>[d]^{\uparrow B} & \vdots\ar@<.6 ex>[d]^{\uparrow B}
 &\vdots \ar@<.6 ex>[d]^{\uparrow B} & &\\
^\s{M}_\d \ot \Uc^{\ot 2} \ar@<.6 ex>[r]^{\hb}\ar@<.6
ex>[u]^{  \uparrow b  } \ar@<.6 ex>[d]^{\uparrow B}&
  ^\s{M}_\d\ot \Uc^{\ot 2}\ot \Fc   \ar@<.6 ex>[r]^{\hb}\ar@<.6 ex>[l]^{\hB}\ar@<.6 ex>[u]^{  \uparrow b  }
   \ar@<.6 ex>[d]^{\uparrow B}&^\s{M}_\d\ot \Uc^{\ot 2}\ot\Fc^{\ot 2}
   \ar@<.6 ex>[r]^{~~\hb}\ar@<.6 ex>[l]^{\hB}\ar@<.6 ex>[u]^{  \uparrow b  }
   \ar@<.6 ex>[d]^{\uparrow B}&\ar@<.6 ex>[l]^{~~\hB} \hdots&\\
^\s{M}_\d \ot \Uc \ar@<.6 ex>[r]^{\hb}\ar@<.6 ex>[u]^{  \uparrow b  }
 \ar@<.6 ex>[d]^{\uparrow B}&  ^\s{M}_\d \ot \Uc \ot\Fc \ar@<.6 ex>[r]^{\hb}
 \ar@<.6 ex>[l]^{\hB}\ar@<.6 ex>[u]^{  \uparrow b  } \ar@<.6 ex>[d]^{\uparrow B}
 &^\s{M}_\d\ot \Uc \ot \Fc^{\ot 2}  \ar@<.6 ex>[r]^{~~\hb}\ar@<.6 ex>[l]^{\hB}\ar@<.6 ex>[u]^{  \uparrow b  }
  \ar@<.6 ex>[d]^{\uparrow B}&\ar@<.6 ex>[l]^{~~\hB} \hdots&\\
^\s{M}_\d  \ar@<.6 ex>[r]^{\hb}\ar@<.6 ex>[u]^{  \uparrow b  }&
^\s{M}_\d\ot\Fc \ar@<.6 ex>[r]^{\hb}\ar[l]^{\hB}\ar@<.6
ex>[u]^{  \uparrow b  }&^\s{M}_\d\ot\Fc^{\ot 2}  \ar@<.6
ex>[r]^{~~\hb}\ar@<.6 ex>[l]^{\hB}\ar@<1 ex >[u]^{  \uparrow b  }
&\ar@<.6 ex>[l]^{~~\hB} \hdots& ,}
\end{xy}
\end{align}

In the next move,  we  identify the standard Hopf cocyclic module $C^\ast(\Hc,
\;^\s{M}_\d)$ with the diagonal subcomplex  $D^\ast$ of $C^{\bullet,\bullet}$.
This is achieved by means of the map
$\, \Psi_{\acl}:  D^\bullet \longrightarrow  C^\ast(\Hc,\;^\s{M}_\d)$
together with its inverse $\Psi^{-1}_{\acl}:C^\ast(\Hc, \;^\s{M}_\d) \longrightarrow D^{\ast}$ .
They are explicitly defined as follows:
 \begin{align}\label{PSI-1}
 \begin{split}
&\Psi_{\acl}(m\ot u^1\odots u^n\ot f^1\odots f^n)=\\
&m\ot f^1\acl u^1\ns{0}\ot f^2u^1\ns{1}\acl u^2\ns{0}\odots \\
& \odots f^n u^1\ns{n-1} \dots u^{n-1}\ns{1}\acl u^n ,
 \end{split}
\end{align}
respectively
 \begin{align}\label{PSI}
 \begin{split}
&\Psi^{-1}_{\acl}(m\ot   f^1\acl u^1\ot \dots\ot f^n\acl u^n)=\\
& m\ot u^1\ns{0}\odots u^{n-1}\ns{0}\ot u^n\ot f^1\ot\\
&\ot f^2S(u^1\ns{n-1})\ot f^3S(u^1\ns{n-2}u^2\ns{n-2}) \odots
f^nS(u^1\ns{1} \dots u^{n-1}\ns{1})  .
 \end{split}
\end{align}

The bicocyclic module \eqref{UF} can be further reduced to  the bicomplex
\begin{equation}
C^{\bullet, \bullet}(\Fg,\Fc, ^\s{M}_\d):=^\s{M}_\d\ot \wedge^{\bullet}\Fg\ot \Fc^{\ot \bullet} ,
\end{equation}
obtained by replacing the tensor algebra
 of $\Uc (\Fg)$ with the exterior algebra of  $\Fg$. To this end,
 recall the action \eqref{bullet}, which is restricted to $\Fg$ reads
 \begin{align} \label{bullet}
&X\bullet (f^1 \ot \cdots \ot f^q) =  \\ \notag
&X\ps{1}\ns{0}\rt f^1\ot X\ps{1}\ns{1}(X\ps{2}\ns{0}\rt f^2)\odots X\ps{1}\ns{q-1}\dots X\ps{q-1}\ns{1}(X\ps{q}\rt f^q),
\end{align}
and define
 $\,  \p_{\Fg} : \;^\s{M}_\d\ot \wg^p\Fg\ot\Fc^{\ot q} \ra \;^\s{M}_\d\ot \wg^{p-1}\Fg\ot\Fc^{\ot q} $
as the Lie algebra homology boundary with respect to the action of
$\Fg$ on $\;^\s{M}_\d\ot\Fc^{\ot q}$ defined by
\begin{equation} \label{Liebdact}
(m\ot\td f)\lt X= m\d_\Fg(X)\ot \td{f}\; - \; X\cdot m\ot \td{f}\;-\;m\ot X\bullet \td f.
\end{equation}
Via the antisymmetrization map
\begin{align} \label{antsym1}
&\td \a_{\Fg}: \;^\s{M}_\d\ot
\wg^q\Fg \ot \Fc^{\ot p}\ot\ra \;^\s{M}_\d\ot \Uc^{\ot q} \ot \Fc^{\ot p} ,
\qquad \td\a_{\Fg}= \a\ot  \Id, \\ \notag
&\a(m \ot X^1\wdots X^p)= \frac{1}{p!} \sum_{\s\in S_p}(-1)^\s  m \ot X^{\s(1)}\odots X^{\s(p)} ,
\end{align}
the pullback of the vertical $b$-coboundary in \eqref{UF}  vanishes, while
the vertical $B$-coboundary becomes
$\p_{\Fg}$ (\cf~\cite[Proposition 7]{Connes-Moscovici-98}).

On the other hand, since $\Fc$ is commutative, the coaction $\Db:\Fg\ra \Fg\ot \Fc$,
 extends  from $\Fg$ to
 a unique coaction $\Db_{\Fg}:\wg^p\Fg\ra \wg^p\Fg\ot \Fc$ . After tensoring with the right coaction of $^\s{M}_\d$ ,
\begin{align}
\begin{split} \label{wgcoact}
&\Db_{^\s{M}_\d\ot\wdg\Fg}(m\ot X^1\wdots X^q)\,=\\
&\,m\ns{0}\ot X^1\ns{0}\wdots X^q\ns{0}\ot
\s^{-1}m\ns{1}X^1\ns{1}\dots X^q\ns{1} .
\end{split}
\end{align}
\medskip
The coboundary $b_\Fc$ is given by
\begin{align} \label{b-cobd}
\begin{split}
&b_\Fc(m\ot\a\ot   f^1\odots f^p )=m \ot\a\ot  1\ot f^1\odots f^p+\\
&\sum_{m\ge i\ge p} (-1)^i m \ot\a\ot  f^1\odots \D(f^i)\odots f^p+\\
&(-1)^{p+1}m\ns{0}\ot\a\ns{0}\ot   f^1\odots f^p\ot  S(\a\ns{1})S(m\ns{1})\s  ,
\end{split}
\end{align}
while  the $B$-boundary is
\begin{align} \label{B-bd}
& B_\Fc= \left(\sum_{i=0}^{q-1}(-1)^{(q-1)i}\tau_\Fc^{i}\right) \s \tau_\Fc , \qquad \qquad
\text{where} \\ \notag
&\tau_\Fc(m\ot \a\ot f^1\odots f^q)= \\\notag
&m\ns{0}\ot\a\ns{0}\ot
S(f^1)\cdot(f^2\odots f^q\ot S(\a\ns{1})S(m\ns{1})\s) ,\\ \notag
&\s(m\ot\a\ot f^1\odots f^q)= \ve(f^q)m\ot\a\ot f^1\odots f^{q-1} .
\end{align}
Actually, since $\Fc$ is commutative and $\Fc$ acts on $\;^\s{M}_\d\ot \wg^q\Fg$ trivially,
by~\cite[Theorem 3.22]{Khalkhali-Rangipour-03}  $B_\Fc$ vanishes in Hochschild cohomology
and therefore can be omitted.

We arrive at the bicomplex $C^{\bullet, \bullet}(\Fg, \Fc, \;^\s{M}_\d)$, described by
the diagram
\begin{align}\label{UF+}
\begin{xy} \xymatrix{  \vdots\ar[d]^{\p_{\Fg}} & \vdots\ar[d]^{\p_{\Fg}}
 &\vdots \ar[d]^{\p_{\Fg}} & &\\
\;^\s{M}_\d \ot \wg^2\Fg \ar[r]^{b_{\Fc}} \ar[d]^{\p_{\Fg}}& \;^\s{M}_\d\ot \wg^2\Fg\ot\Fc  \ar[r]^{b_{\Fc}} \ar[d]^{\p_{\Fg}}
 &\;^\s{M}_\d\ot \wg^2\Fg\ot\Fc^{\ot 2} \ar[r]^{~~~~~~b_{\Fc}}   \ar[d]^{\p_{\Fg}}& \hdots&\\
\;^\s{M}_\d \ot \Fg \ar[r]^{b_{\Fc}} \ar[d]^{\p_{\Fg}}&  \;^\s{M}_\d\ot \Fg\ot\Fc \ar[r]^{b_{\Fc}} \ar[d]^{\p_{\Fg}}&\;^\s{M}_\d\ot \Fg\ot\Fc^{\ot 2}
   \ar[d]^{\p_{\Fg}} \ar[r]^{~~~~~~~b_{\Fc}}& \hdots&\\
\;^\s{M}_\d \ar[r]^{b_{\Fc}}&  \;^\s{M}_\d\ot\Fc \ar[r]^{b_{\Fc}}&\;^\s{M}_\d\ot\Fc^{\ot 2} \ar[r]^{~~~~~b_{\Fc}} & \hdots&  . }
\end{xy}
\end{align}

Referring to~\cite[Prop. 3.21 and \S 3.3]{Moscovici-Rangipour-09} for more details, we state the conclusion as follows.

\begin{proposition} \label{mixCE}
The map \eqref{antsym1} induces a quasi-isomorphism
between the total complexes  $\Tot C^{\bullet, \bullet}(\Fg, \Fc, \;^\s{M}_\d )$
and $\Tot C^{\bullet, \bullet} (\Uc, \Fc, \;^\s{M}_\d)$.
\end{proposition}

\bigskip

In order to convert the Lie algebra homology into Lie algebra cohomology,
 we shall resort the Poincar\'e isomorphism
\begin{align}\label{theta}
\begin{split}
&\FD_{\Fg} = \Id\ot \Fd_{\Fg} \ot \Id : M\ot\wedge^q\Fg^\ast\ot \Fc^{\ot p} \ra
  \;M\ot  \wedge^{m}\Fg^\ast \ot\wedge^{m-q}\Fg  \ot \Fc^{\ot p} \\
 &\Fd_{\Fg}(\eta) \, = \,   \varpi^\ast\ot \iota(\eta) \varpi,
 \end{split}
\end{align}
where $\varpi$ is a covolume element and
$\varpi^\ast $ is the
dual volume element. The contraction operator is defined as follows:
for $\lambda \in \Fg^\ast $,
$\iota(\lambda) : \wedge^{\bullet}\Fg \ra \wedge^{\bullet-1}\Fg$ is
the unique derivation of degree $-1$ which on $\Fg$ is the evaluation map
\begin{align*}
\iota(\lambda) (X) \, = \, \langle \lambda, v \rangle , \qquad \fl \, X \in \Fg ,
\end{align*}
while for $\eta = \lambda_1 \wdots \lambda_q \in \wedge^{q}\Fg^\ast$,
 $\iota(\eta):\wedge^{\bullet}\Fg \ra \wedge^{\bullet-q}\Fg$ is given by
\begin{align*}
\iota (\lambda_1 \wdots \lambda_q) := \iota (\lambda_q) \circ \ldots
 \circ \iota(\lambda_1) , \qquad \fl \, \lambda_1, \ldots , \lambda_q \in \Fg^\ast .
\end{align*}
Noting that the coadjoint action of $\Fg$ induces on $ \wedge^{m}\Fg^\ast$ the action
\begin{align}
\ad^\ast (X) \varpi^\ast = \d_\Fg(X) \varpi^\ast , \qquad \fl \, X \in \Fg,
\end{align}
we shall identify $\wedge^{m}\Fg^\ast$ with $\;\Cb_\d$ as $\Fg$-modules.

Let $\{X_i\}$ be a basis for $\Fg$ and $\{\t^j\}$ be its dual basis for $\Fg^\ast$. We use  $\Db_\Fg(X_i)=X_j\ot f^j_i$ to define the following left coaction $\Db_\Fg^\ast:\Fg^\ast\ra \Fc \ot \Fg^\ast$ which can be seen as the transpose of the  original right coaction $\Db_\Fg$.
\begin{equation}
\Db_\Fg^\ast(\t^i)= \sum_j f^i_j\ot \t^j.
\end{equation}
Let us check that it is a coaction. We know from \eqref{D-f-j-i} that $\D(f_j^i)= f^i_k\ot f^k_j$.
\begin{equation}
((\Id\ot \Db_{\Fg}^\ast)\circ\Db_\Fg^\ast)(\t^i)= \sum_{j,k}f^i_k\ot f^k_j\ot \t^j=
((\D\ot\Id)\circ\Db_\Fg^\ast)(\t^i).
\end{equation}
We extend this coaction on $\wdg ^\bullet \Fg^\ast$ diagonally and observe that the  result is a left coaction just because $\Fc$ is commutative. For $\a:= \t^{i_1}\wdots \t^{i_k}$, it is recorded below  by
\begin{align}\label{coaction-g-ast}
\begin{split}
\a\ns{-1}\ot \a\ns{0}= \sum_{1\le l_j\le m} f_{l_1}^{i_1}\cdots f_{l_k}^{i_k}\ot \t^{l_1}\wdots \t^{l_k}.
\end{split}
\end{align}

One easily sees that we have $\Db_\Fg^\ast(\varpi^\ast)=\s\ot \varpi^\ast$. In other words, as a right module and left comodule,
\begin{equation}
\wdg^{\dim \Fg}\Fg^\ast= \;^\s{\Cb}_\d.
\end{equation}

One uses  the antipode of $\Fc$ to turn $\Db_\Fg^\ast$ into a right coaction, we denote resulting right coaction by $\Db_{\Fg^\ast}$. We then apply this right coaction to endow $M\ot \wdg^p\Fg^\ast$ with a right coaction as follows,
\begin{align}
\begin{split}
&\Db_{M\ot\wdg\Fg^\ast}: M\ot \wdg^p\Fg^\ast\ra M\ot \wdg^p\Fg^\ast\ot \Fc,\\
&\Db_{M\ot\wdg\Fg^\ast}(m\ot \a)= m\ns{0}\ot \a\ns{1}\ot m\ns{1}S(\a\ns{-1}).
\end{split}
\end{align}

\begin{lemma}
The Poincar\'e isomorphism connects $\Db_{^\s{M}_\d\ot\wdg\Fg}$ and $\Db_{M\ot\wdg\Fg^\ast}$ in the following sense,
\begin{equation}\label{F-coaction-g*}
\Db_{M\ot\wdg\Fg^\ast}(m\ot\om)=\FD_{\Fg}^{-1}\circ\Db_{^\s{M}_\d\ot\wdg\Fg}\circ\FD_{\Fg}(m\ot\om).
\end{equation}
\end{lemma}
\begin{proof}
Without loss of generality let $\om:=\t^{p+1}\wdots\t^m$. We observe that
\begin{equation}
\FD_{\Fg}(m\ot\om)= m\ot\varpi^\ast\ot X_1\wdots X_l.
\end{equation}
Applying $\Db_{^\s{M}_\d\ot\wdg\Fg}$ we get,
\begin{align}\label{proof-FD-Db}
\begin{split}
&\Db_{^\s{M}_\d\ot\wdg\Fg}(\FD_{\Fg}(m\ot\om))= \\
&\underset{1\le l_t,i_s\le m}{\sum m\ns{0}\ot} \t^{i_1}\wdots \t^{i_m}\ot X_{l_1}\wdots X_{l_p}\ot m\ns{1}S(f^1_{i_1})\cdots S(f_{i_m}^m)f^{l_{1}}_1\cdots f^{l_p}_p=\\
&\sum_{1\le l_t,i_s\le m}m\ns{0}\ot \t^{l_1}\wdots\t^{l_p}\wg\t^{i_{p+1}}\wdots \t^{i_m}\ot\\
&~~~~~~~~~~~~~~~~~~~~~~~~~~~~~~~~~~~~~~~ X_{l_1}\wdots X_{l_p}\ot m\ns{1}S(f^{p+1}_{i_{p+1}})\cdots S(f_{i_m}^m)\\
&\sum_{\mu\in S_m}(-1)^\mu m\ns{0}\ot\varpi^\ast \ot X_{\mu(1)}\wdots X_{\mu(p)}\ot m\ns{1}S(f^{p+1}_{\mu(p+1)})\cdots S(f^{m}_{\mu(m)})=\\
&\underset{1\le l_1< \ldots< l_p\le 1, \; \mu\in S_{m-p}}{\sum(-1)^\mu m\ns{0}\ot\varpi^\ast} \ot X_{l_1}\wdots X_{l_p}\ot m\ns{1}S(f^{p+1}_{j_{\mu(1)}})\cdots S(f^{m}_{j_{\mu(m-p)}}).\\
\end{split}
\end{align}
 Here in the last part we denote by $\{j_1<j_2<\ldots <j_{m-p}\}$ the complement of $\{l_1<l_2<\ldots< l_p\}$ in $\{1, \ldots, m\}$.
On the other hand,
\begin{align}
\begin{split}
&\FD_\Fg(\Db_{M\ot\wdg\Fg^\ast}(m\ot\om))= \\
&\sum_{l_{1},\ldots,l_{m-p}}\FD_\Fg( m\ns{0}\ot \t^{l_{1}}\wdots\t^{l_{m-p}}\ot m\ns{1}S(f^{p+1}_{l_1}\cdots f^{m}_{l_{m-p}}))=\\
&\sum_{l_t}\FD_\Fg( m\ns{0}\ot \t^{l_{1}}\wdots\t^{l_{m-p}}\ot m\ns{1}S(f^{p+1}_{l_1})\cdots S(f^{m}_{l_{m-p}})=\\
&\underset{\underset{\;\mu\in S_{m-p}}{1\le j_1<\ldots <j_{m-p}\le m,}}{\sum}(-1)^\mu\FD_\Fg( m\ns{0}\ot \t^{j_1}\wdots\t^{j_{m-p}}\ot m\ns{1}S(f^{p+1}_{j_{\mu(1)}}\cdots f^{m}_{j_{\mu(m-p)}})=\\
&\underset{\underset{\;\mu\in S_{m-p}}{1\le l_1< \ldots< l_p\le 1,} }{\sum}(-1)^\mu m\ns{0}\ot\varpi^\ast \ot X_{l_1}\wdots X_{l_p}\ot m\ns{1}S(f^{p+1}_{j_{\mu(1)}})\cdots S(f^{m}_{j_{\mu(m-p)}}).
\end{split}
\end{align}

\end{proof}

By transfer of structure,  the bicomplex  \eqref{UF+} becomes
$(C^{\bullet,\bullet}(\Fg^\ast,  \Fc, M) ,\p_{\Fg^\ast}, b^\ast_\Fc)$,
 \begin{align}\label{UF+*}
\begin{xy} \xymatrix{  \vdots & \vdots
 &\vdots &&\\
 M\ot \wdg^2\Fg^\ast  \ar[u]^{\p_{\Fg^\ast}}\ar[r]^{b^\ast_\Fc~~~~~~~}&  M\ot \wdg^2\Fg^\ast\ot\Fc \ar[u]^{\p_{\Fg^\ast}} \ar[r]^{b^\ast_\Fc}& M\ot \wdg^2\Fg^\ast\ot\Fc^{\ot 2} \ar[u]^{\p_{\Fg^\ast}} \ar[r]^{~~~~~~~~~b^\ast_\Fc} & \hdots&  \\
 M\ot \Fg^\ast  \ar[u]^{\p_{\Fg^\ast}}\ar[r]^{b^\ast_\Fc~~~~~}& M\ot \Fg^\ast\ot\Fc \ar[u]^{\p_{\Fg^\ast}} \ar[r]^{b^\ast_\Fc}& M\ot  \Fg^\ast\ot \Fc^{\ot 2} \ar[u]^{\p_{\Fg^\ast}} \ar[r]^{~~~~~b^\ast_\Fc }& \hdots&  \\
   M\ar[u]^{\p_{\Fg^\ast}}\ar[r]^{b^\ast_\Fc~~~~~~~}& M\ot \Fc \ar[u]^{\p_{\Fg^\ast}}\ar[r]^{b^\ast_\Fc}& M\ot \Fc^{\ot 2} \ar[u]^{\p_{\Fg^\ast}} \ar[r]^{~~~~~b^\ast_\Fc} & \hdots&. }
\end{xy}
\end{align}
 The vertical coboundary $\p_{\Fg^\ast}:C^{p,q}\ra C^{p,q+1}$ is the Lie algebra
cohomology coboundary of the Lie algebra $\Fg$ with coefficients in $M\ot\Fc^{\ot p}$,
 where the action of $\Fg$ is
\begin{multline}
( m\ot f^1\odots f^p) \blacktriangleleft X=\\
- X\cdot m\ot f^1\odots f^p\;-\; m\ot X\bullet(f^1\odots f^p).
\end{multline}
The horizontal $b$-coboundary $b^*_\Fc$ is precisely defined by
\begin{align}
\begin{split}
&b^\ast_\Fc(m\ot\a\ot f^1\odots f^q)=\\
&m\ot\a\ot 1\ot f^1\odots f^q +\sum_{i=1}^q(-1)^i m\ot\a\ot f^1\odots \D(f^i)\odots f^q+\\
&(-1)^{q+1}m\ns{0}\ot \a\ns{1}\ot f^1\odots f^q\ot S(m\ns{1})\a\ns{-1}.
\end{split}
\end{align}

For future reference, we record the conclusion.

\begin{proposition} \label{mixCE*}
The map \eqref{theta} induces a quasi-isomorphism
between the total complexes  $\Tot C^{\bullet, \bullet} (\Uc, \Fc, \;^\s{M}_\d) $
and $\Tot C^{\bullet, \bullet} (\Fg^\ast,  \Fc) $.
\end{proposition}

\subsection{ Hopf cyclic cohomology of $\Fc\acl U(\Fg_1)$, for $\Fc= R(G_2)$, $R(\Fg_2)$, and $\Pc(G_2)$}
\label{SS-Hopf cyclic cohomology of noncommutative Hopf algebras}
As we have seen for a matched pair of Lie algebras  $(\Fg_1, \Fg_2)$ there  is a right action of $\Fg_1$ on $\Fg_2$ and a left action of $\Fg_2$ on $\Fg_1$
satisfying the compatibility conditions \eqref{mp-L-1}, \ldots,  \eqref{mp-L-4}.

Given such a matched pair, one defines the double crossed sum Lie algebra whose underlying vector space is $\Fg_1\oplus\Fg_2$ by setting
\begin{equation*}
[X\oplus\z,  Z\oplus\x]=([X, Z]+\z\rt Z-\x\rt X)\oplus ([\z,\x]+\z\lt Z-\x\lt X).
\end{equation*}
 Conversely, given a Lie algebra $\Fa$ and
 two Lie subalgebras $\Fg_1$ and $\Fg_2$ so that $\Fa=\Fg_1\oplus\Fg_2$ as vector  spaces,  then
 $(\Fg_1, \Fg_2)$  forms a matched pair of Lie algebras and $\Fa\cong \Fg_1\bowtie \Fg_2$ as Lie algebras.
  In this case, the actions of $\Fg_1$ on $\Fg_2$ and $\Fg_2$ on $\Fg_1$  for $\z\in \Fg_2$ and $X\in\Fg_1$ are uniquely determined  by
\begin{equation*}
[\z,X]=\z\rt X+\z\lt X
\end{equation*}
Let $\Fh\subseteq \Fg_2$ be a $\Fg_1$-invariant subalgebra. Then one  easily sees that $\Fa/\Fh\cong \Fg_1\oplus \Fg_2/\Fh$. In addition, we let $\Fh$ act on $\Fa/\Fk$ by the induced adjoint action i.e.,
\begin{equation}\label{adjoint}
Ad_\z( Z\oplus \bar\x)= \overline{[0\oplus \z, Z\oplus \x]}= \z\rt Z\oplus\overline{ [\z,\x]}.
\end{equation}
For simplicity we denote $\Fg_2/\Fh$ by $\Fl$.
We like to make sure that the Chevalley-Eilenberg coboundary of $\Fg_1$ with coefficients in $M\ot \wdg^q\Fl^\bullet$ is $\Fh$-linear. To do so, we observe the restriction of  the action of $\Fg_2$ on $\Fg_1$ induces  an action of $\Fh$ on $\Fg_1$. We assume that this action of  $\Hc$ on $\Fg_1$  is given by derivations.

We now introduce the following bicomplex

\begin{equation}\label{g-1-g-2-bicomplex}
\xymatrix{\vdots&\vdots&\vdots&\\
(M\ot \wdg^2\Fg_1^\ast)^\Fh \ar[r]^{\hP}\ar[u]^{\vP}&(M\ot \wdg^2\Fg_1^\ast\ot \Fl^\ast)^\Fh\ar[r]^{\hP} \ar[u]^{\vP}& (M\ot \wdg^2\Fg_1^\ast\ot \wdg^2\Fl^\ast)^\Fh\ar[r]^{~~~~~~~~~~~\hP}\ar[u]^{\vP}&\cdots\\
(M\ot \Fg_1^\ast)^\Fh \ar[r]^{\hP}\ar[u]^{\vP}&(M\ot \Fg_1^\ast\ot \Fl^\ast)^\Fh\ar[r]^{\hP}\ar[u]^{\vP} & (M\ot \Fg_1^\ast\ot \wdg^2\Fl^\ast)^\Fh\ar[r]^{~~~~~~~~~~~\hP}\ar[u]^{\vP}&\cdots\\
M^\Fh \ar[r]^{\hP}\ar[u]^{\vP}&(M\ot \Fl^\ast)^\Fh\ar[r]^{\hP} \ar[u]^{\vP}& (M\ot \wdg^2\Fl^\ast)^\Fh\ar[r]^{~~~~~~~~~~\hP}\ar[u]^{\vP}&\cdots}
\end{equation}

Here $\hP$ is the relative Lie algebra cohomoloy of the pair  $(\Fg_2, \Fh)$ with coefficients in $M\ot\wdg^p\Fg_1^\ast$, where $\Fg_2$ acts on $M$ by restriction and on $\Fg_1$ by its natural action. The vertical coboundary $\vP$ is the Lie algebra cohomology of $\Fg_1$ with coefficients on $M\ot\wdg^q\Fl^\ast$, where $\Fg_1$ acts on $M$ in the obvious way, i.e, restriction of the action of $\Fa$ on $M$, and on $\Fl$ by the induced action of $\Fg_1$ on $\Fg_2$.
One notes that since action of $\Fh$ on $\Fg_1$ is given by derivations then the vertical coboundary is well-defined.

One identifies
\begin{equation}
\xymatrix{(M\ot \wdg^s (\Fa/\Fh)^\ast)^\Fh\ar[rr]^{\nr~~~~~~~~~~}&&\bigoplus_{p+q=s} \left(M\ot \wdg^p\Fg_1^\ast\ot \wdg^q \Fl^\ast\right)^\Fh}
\end{equation}

This isomorphism is implemented by the map \\
$\natural:  C^{s}(\Fg_1\bowtie\Fg_2, \Fh, M)\ra \bigoplus_{p+q=s} (M\ot \wdg^p\Fg^\ast\ot \wdg^q\Fl^\ast)^\Fh$,
\begin{align*}
\natural(\omega)(Z_1,\dots, Z_p\mid \z_1, \dots, \z_q)=\omega(Z_1\oplus 0,\dots, Z_p\oplus 0, 0\oplus\z_1, \dots, 0\oplus \z_q) ,
\end{align*}
whose inverse is given by
\begin{align*}
&\natural^{-1}(\mu\ot\nu)(Z_1\oplus\z_1, \dots,Z_{p+q}\oplus\z_{p+q})= \\
&\sum_{\s\in Sh(p,q)}(-1)^{\s}\mu(Z_{\s(1)}, \dots,Z_{\s(p)})\nu(\z_{\s(p+1)}, \dots, \z_{\s(p+q)}) .
\end{align*}
\begin{lemma}\label{lemma-natural}
The map $\nr$ is an isomorphism of complexes.
\end{lemma}
\begin{proof}
 We see that $\nr$ is induced by $\Fa^\ast=\Fg_1^\ast\oplus\Fg_2^\ast$. Then one uses \eqref{adjoint}  to show that the vertical and horizontal coboundaries of \eqref{g-1-g-2-bicomplex} are just restriction of the Chevalley-Eilenberg coboundary of $\Fa$ with coefficients in $V$.
 It is routine to show that the map $\nr$  is an isomorphism of complexes  between the relative Lie algebra cohomology of the pair  $(\Fa, \Fh)$ with coefficients in the $\Fa$-module  $M$ and the total complex of the bicomplex \eqref{g-1-g-2-bicomplex}. However, we refer the reader to Lemma 2.7 in \cite{Moscovici-Rangipour-011} for a proof in a very similar situation.
\end{proof}

\begin{definition}
 Let  a $\Fg_1$-Hopf algebra $\Fc$  be in Hopf duality with $U(\Fg_2)$ for  a matched pair of Lie algebras  $(\Fg_1,\Fg_2)$. Then we say  $\Fc$ is  $(\Fg_1, \Fg_2)$-related  if
 \begin{enumerate}
   \item  The pairing is $U(\Fg_1)$-balanced, i.e
\begin{equation}\label{g-balenced}
\langle v, u\rt f \rangle= \langle v \lt u, f \rangle, \qquad f\in\Fc, \;v\in U(\Fg_2),\; u\in U(\Fg_1).
\end{equation}
\item The action of $U(\Fg_2)$ on $U(\Fg_1)$ is compatible with the coaction of $\Fc$ on $U(\Fg_1)$ via the pairing, i.e
\begin{equation}\label{g-cobalenced}
u\ns{0}\langle v, u\ns{1}\rangle= v\rt u,\qquad u\in U(\Fg_1), \;\; v\in U(\Fg_2).
\end{equation}
 \end{enumerate}
\end{definition}
One uses the right action of $\Fg_1$ on $\Fg_2$ to induce a right action of $U(\Fg_1)$ on $U(\Fg_2)^{\otimes q}$ as follows,
\begin{align}\label{*-action}
\begin{split}
&(v^1\odots v^q)\ast u=\\
&v^1\lt (v^2\ps{1}\cdots v^q\ps{1}\rt u\ps{1})\odots v^{q-1}\ps{q-1}\lt (v^q\ps{q-1}\rt u\ps{q-1})\ot v^q\ps{q}\lt u\ps{q}
\end{split}
\end{align}
\begin{lemma}
The equation \eqref{*-action} defines an action of $U(\Fg_1)$ on $U(\Fg_2)^{\ot q}$.
\end{lemma}
\begin{proof}
We need to prove that for $\td v:= v^1\odots v^q\in U(\Fg_2)^{\ot q}$, and $u^1$, $u^2$ in $U(\Fg_1)$, we have $(\td v\ast u^1)\ast u^2= (\td v)\ast(u^1u^2)$. Indeed, first we use the fact that $U(\Fg_2)$ is $U(\Fg_1)$-module coalgebra and \eqref{mutual-1} to observe that
\begin{align}\label{proof-*-action}
\begin{split}
&((v^1\ot v^2)\ast u^1)\ast u^2= (v^1 \lt(v^2\ps{1}\rt u^1\ps{1})\ot v^2\ps{2} \lt u^1\ps{2})\ast u^2=\\
& (v^1 \lt(v^2\ps{1}\rt u^1\ps{1})\lt ((v^2\ps{2} \lt u^1\ps{2})\ps{1}\rt u^2\ps{1})\ot (v^2\ps{2} \lt u^1\ps{2})\ps{2}\lt u^2\ps{2})=\\
&v^1 \lt((v^2\ps{1}\rt u^1\ps{1})((v^2\ps{2} \lt u^1\ps{2})\rt u^2\ps{1}))\ot (v^2\ps{3} \lt u^1\ps{3}u^2\ps{2})=\\
&v^1\lt ( v^2\ps{1}\rt (u^1\ps{1}u^2\ps{2}))\ot v^2\ps{2}\lt (u^1\ps{2}u^2\ps{2})= (v^1\ot v^2)\ast u^1u^2.
\end{split}
\end{align}
Then  for $\td v\in U(\Fg_2)^{\ot m}$ and $\td w= w^1\odots w^l\in U(\Fg)^{\ot l}$, we observe that
\begin{equation}
(\td v\ot \td w)\ast u= \td v\ast(w^1\ps{1}\dots w^l\ps{1}\rt u\ps{1})\ot (w^1\ps{2}\odots w ^l\ps{2})\ast u\ps{2}.
\end{equation}
This observation and \eqref{proof-*-action} completes the proof.
\end{proof}

\begin{proposition}\label{Proposition-theta-linear}
Let $(\Fg_1, \Fg_2)$ be a matched pair of Lie algebras and $\Fc$ be a $(\Fg_1, \Fg_2)$-related Hopf algebra. Then the map $\t_{\Fc, U(\Fg_2)}$ defined in \eqref{map-theta-Alg} is $U(\Fg_1)$-linear provided $\Fg_1$ acts on $\Fc^{\ot q}$ by $\bullet$ defined in \eqref{bullet}, and on $U(\Fg_2)^{\ot q}$ by $\ast$ defined in \eqref{*-action}.
\end{proposition}
\begin{proof} Without loss of generality we assume that $V=\Cb$. We use the Hopf  pairing properties,  \eqref{g-balenced}, and \eqref{g-cobalenced} to observe that
\begin{align}
\begin{split}
&\t_{\Fc,U(\Fg_2)}(u\bullet (f^1\odots f^q))(v^1\odots v^q)=\\
&\langle v^1,u\ps{1}\ns{0}\rt f^1\rangle \langle v^2,u\ps{1}\ns{1}(u\ps{2}\ns{0}\rt f^1)\rangle\cdots \langle v^q,u\ps{1}\ns{q-1}\cdots u\ps{q-1}\ns{1} (u\ps{q+1}\rt f^q)\rangle=\\
&\langle v^1,  u\ps{1}\ns{0}\rt f^1\rangle \langle v^2\ps{1},u\ps{1}\ns{1}\rangle\langle v^2\ps{2},u\ps{2}\ns{0}\rt f^2\rangle\cdots\\
&\langle v^q\ps{1}\,,\, u\ps{1}\ns{q-1}\rangle\langle v^q\ps{2}\,,\, u\ps{2}\ns{q-2}\rangle\cdots \langle v^q\ns{q-1}\,,\, u\ps{q-1}\ns{1}\rangle\langle v^q\ps{q}\,,\, u\ps{q+1}\rt f^q\rangle=\\
&\langle v^1,  u\ps{1}\ns{0}\rt f^1\rangle \langle v^2\ps{1}\cdots v^q\ps{1},u\ps{1}\ns{1}\rangle \langle v^2\ps{2},u\ps{2}\ns{0}\rt f^2\rangle\langle v^3\ps{2}\cdots v^q\ps{2}\,,\, u\ps{2}\ns{1} \rangle\cdots\\
& \cdots\langle v^{q-1}\ps{q-1},u\ps{q-1}\ns{0}\rt f^{q-1}\rangle\langle v^q\ps{q-1}\,,\, u\ps{q-1}\ns{1} \rangle\langle v^q\ps{q}, u\ps{q}\rt f^q\rangle=\\
&\langle v^1,  (  v^2\ps{1}\cdots v^q\ps{1}\rt  u\ps{1}  )\rt f^1\rangle \langle v^2\ps{2},
( v^3\ps{2}\cdots v^q\ps{2}\rt u\ps{2})\rt f^2\rangle\cdots\\
& \cdots\langle v^{q-1}\ps{q-1}, ( v^q\ps{q-1}\rt u\ps{q-1})\rt f^{q-1}\rangle\langle v^q\ps{q}, u\ps{q}\rt f^q\rangle=\\
&\langle v^1\lt (  v^2\ps{1}\cdots v^q\ps{1} \rt u\ps{1})\,,\, f^1\rangle \langle v^2\ps{2}\lt ( v^3\ps{2}\cdots v^q\ps{2}\rt u\ps{2})\,,\, f^2\rangle\cdots\\
& \cdots\langle v^{q-1}\ps{q-1}\lt( v^q\ps{q-1}\rt u\ps{q-1})\,,\, f^{q-1}\rangle\langle v^q\ps{q}\lt u\ps{q}\,,\, f^q\rangle=\\
&\t_{\Fc,U(\Fg_2)}(f^1\odots f^q)((v^1\odots v^q)\ast u).
\end{split}
\end{align}
\end{proof}

\begin{proposition}\label{proposition-antisymmetriza-linear}
For a matched pair of Lie algebras $(\Fg_1,\Fg_2)$, let $\Fg_1$ act on $U(\Fg_2)^{\ot q}$ by $\ast$ defined in \eqref{*-action}, and on $\wg^q\Fg_2$ by the intrinsic right action of $\Fg_1$ on $\Fg_2$. Then the antisymmetrization map is a  $\Fg_1$-linear map  of complexes between normalized Hochschild cochains of $U(\Fg_2)$ and Lie algebra cochains of $\Fg_2$.
\end{proposition}
\begin{proof}
It is known that the antisymmetrization map $\a:\Hom(U(\Fg_2)^{\ot q}, V)\ra V\ot \wdg^q\Fg_2^\ast$ defined in \eqref{map-antisymmetrization}, is a  map of complexes. In the next proposition we prove that it is actually $\Fg_1$-linear.
One uses the fact that $v\rt 1=\ve(v)$, and that the elements of the Lie algebra is primitive in its enveloping Hopf algebra to see that for any $\xi^1, \ldots, \xi^q \in \Fg_2$, any $X\in \Fg_1$, and any normalized cochain $\phi$ we have
\begin{align}
\begin{split}
&\phi((\xi^1\odots \z^q)\ast X)=\phi(\xi^1\lt (\xi^2\ps{1}\cdots \xi^q\ps{1}\rt X)\ot \xi^2\ps{2}\lt (\xi^3\ps{2}\cdots \xi^q\ps{2}\rt 1)\ot\cdots\\
&~~~\cdots\ot \xi^{q-1}\ps{q-1}\lt (\xi^q\ps{q-1}\rt 1)\ot \xi^q\ps{q}+\\
&~~~~~~~\xi^1\lt (\xi^2\ps{1}\cdots \xi^q\ps{1}\rt 1)\ot \xi^2\ps{2}\lt (\xi^3\ps{2}\cdots \xi^q\ps{2}\rt X)\ot\cdots\\
&~~~~~~~~~~~~~\cdots\ot \xi^{q-1}\ps{q-1}\lt (\xi^q\ps{q-1}\rt 1)\ot \xi^q\ps{q}+\cdots\\
&~~~~~~~~~~~~~~~~~~~\cdots + \xi^1\lt (\xi^2\ps{1}\cdots \xi^q\ps{1}\rt 1)\ot \xi^2\ps{2}\lt (\xi^3\ps{2}\cdots \xi^q\ps{2}\rt 1)\ot\cdots\\
&~~~~~~~~~~~~~~~~~~~~~~~~~~~~~~~~\cdots\ot \xi^{q-1}\ps{q-1}\lt (\xi^q\ps{q-1}\rt 1)\ot \xi^q\ps{q}\lt X)=\\
&\sum _{i=1}^q  \phi(\xi^1\ot \xi^2\odots \xi^i\lt (\xi^{i+1}\ps{1}\cdots \xi^q\ps{1}\rt X)\ot \xi^{i+1}\ps{2}\odots \xi^q\ps{2})=\\
&\underset{\xi^{i+1}\ps{1}=\cdots=\xi^{q}\ps{1}=1}{\sum _{i=1}^q  \phi(\xi^1\ot \xi^2\odots} \xi^i\lt (\xi^{i+1}\ps{1}\cdots \xi^q\ps{1}\rt X)\ot \xi^{i+1}\ps{2}\odots \xi^q\ps{2})+\\
&\underset{\xi^{j}\ps{2}=1~~\text{for some }\; i+1\le j\le q}{\sum _{i=1}^q \phi( \xi^1\ot \xi^2\odots} \xi^i\lt (\xi^{i+1}\ps{1}\cdots \xi^q\ps{1}\rt X)\ot \xi^{i+1}\ps{2}\odots \xi^q\ps{2})=\\
&\sum_{i=1}^q \phi(\xi^1\ot \xi^2\odots \xi^i\lt  X\ot \xi^{i+1}\odots \xi^q )+0.
\end{split}
\end{align}
\end{proof}

Now let $(\Fg_1, \Fg_2)$ be a  matched pair of Lie algebras and $\Fc$  a    $(\Fg_1,\Fg_2)$-related Hopf algebra. In addition, let $\Fg_2= \Fh\ltimes \Fl$ be a semi-crossed-sum  of Lie algebras, where $\Fh$ is reductive and every $\Fh$-module is semisimple. We define the following map between the bicomplexes \eqref{UF+*} and  \eqref{g-1-g-2-bicomplex}

\begin{align}\label{VE}
&\Vc:= \t_{\Fc,\Fl,0}: M\ot \wdg^q\Fg_1^\ast\ot \Fc^{\ot q}\ra (M\ot\wdg^p \Fg_1\ot \wdg^q\Fl^\ast)^\Fh
\end{align}
In other words,
\begin{align}
\begin{split}
&\Vc(m\ot\om\ot f^1\odots f^q)( X^1,\ldots,X^p\mid \xi^1,\ldots, \xi^q)=\\
&\om(X^1,\ldots X^p)\sum_{\s\in S_q}(-1)^\s \langle \xi^{\s(1)}\,,\, f^1\rangle \cdots
\langle \xi^{\s(q)}\,,\, f^q\rangle m.
\end{split}
\end{align}

\begin{theorem}\label{Theorem-main}
Let $(\Fg_1,\Fg_2)$ be a matched pair of Lie algebras and $\Fc$ a $(\Fg_1,\Fg_2)$-related Hopf algebra. Assume that $\Fg_2=\Fh\ltimes \Fl$ is a $\Fc$-Levi decomposition such that $\Fh$ is $\Fg_1$-invariant and the natural action of $\Fh$ on $\Fg_1$ is given by derivations. Then for any $\Fc$-comodule and $\Fg_1$-module $M$, the map $\Vc$ defined in \eqref{VE}, is a map of bicomplexes and induces an isomorphism between Hopf cyclic cohomology of $\Fc\acl U(\Fg_1)$ with coefficients in $^\s{M}_\d$ and the Lie algebra cohomology of $\Fa:=\Fg_1\bowtie \Fg_2$ relative to $\Fh$ with coefficients in the $\Fa$-module induced by $M$. In other words,
\begin{equation}
HP^\bullet(\Fc\acl U(\Fg_1), ^\s{M}_\d)\cong \bigoplus_{i=\bullet\;\text{mod 2}} H^i(\Fg_1\bowtie \Fg_2, \Fh,\;M).
\end{equation}
\end{theorem}
\begin{proof}
First we have to prove that $\Vc$ commutes with both of the boundaries of our bicomplexes. The commutation of $\Vc$ with horizontal coboundaries is  guaranteed by the fact that $\t_{\Fc,\Fl, \mu}$ is a complex map for $V=M\ot \wdg^p\Fg_1^\ast$. Since $\Vc$ does  not have any affect on  $M\ot\wdg^p\Fg_1^\ast$, by the equivalent definition of Chevalley-Eilenberg coboundary in \eqref{equivalent-CE}, to prove that $\Vc$ commutes with the vertical coboundary it is necessary and   sufficient  to show that  $\Vc$ is $\Fg_1$-linear. Since $\Vc$ is made of antisymmetrization map $\a$ and $\t_{\Fc, U(\Fg_2)}$, Proposition \ref{Proposition-theta-linear} and Proposition \ref{proposition-antisymmetriza-linear} prove that $\Vc$ commutes with vertical coboundaries which are both
 Lie algebra cohomology coboundaries of $\Fg_1$. Finally, one uses the assumption that  $\Fg_2=\Fh\ltimes \Fl$ is a  $\Fc$-Levi decomposition which implies, by definition, that $\Vc$ induces an isomorphism  in the first term of the spectral sequence of the bicomplexes. We then conclude that $\Vc$ induces an  isomorphism in the level of total cohomology of  total complexes. The proof is complete since  total complex of the \eqref{UF+*} computes the Hopf cyclic cohomology by Proposition \ref{mixCE*} and the total complex of \eqref{g-1-g-2-bicomplex} computes the relative Lie algebra cohomology by Lemma \ref{lemma-natural}.
\end{proof}

\begin{corollary}Let $(\Fg_1,\Fg_2)$ be a matched pair of finite dimensional Lie algebras. Assume that  $\Fg_2=\Fh\ltimes \Fl$ is a Levi decomposition of $\Fg_2$ such that $\Fh$ is $\Fg_1$-invariant and the action of $\Fh$ on $\Fg_1$ is given by derivations. Then for any finite dimensional $\Fg_1\bowtie \Fg_2$-module $M$ we have
\begin{equation}
HP^\bullet(R(\Fg_2)\acl U(\Fg_1),\; ^\s{M}_\d)\;\;\cong \bigoplus_{i=\bullet\text{~mod ~}2} H^i(\Fg_1\bowtie\Fg_2,\Fh,M).
\end{equation}
\end{corollary}
\begin{proof}
The main task here to prove that the criteria of Theorem \ref{Theorem-main} are satisfied for $\Fc:=R(\Fg_2)$. It is shown in Proposition \ref{Proposition-matched-Lie-Hopf-Lie} that $R(\Fg_2)$ is a $\Fg_1$-Hopf algebra.

We know that $\Fc$  and $U(\Fg_2)$ are in a Hopf pairing via \eqref{F-V-pairing}, i.e,
\begin{equation}\label{pairing-R(g)-U(g)}
\langle v\,,\,  f\rangle_{\rm Alg} =f(v), \quad f\in R(\Fg_2), \;\;v\in U(\Fg_2).
\end{equation}
The pairing \eqref{pairing-R(g)-U(g)} is by definition, as it is defined in \eqref{actio-U-Uo},   $U(\Fg_1)$-balanced. The equation \eqref{coact-act-comp} shows that the coaction of $R(\Fg_2)$ on $U(\Fg_1)$ is compatible with the pairing.  So $R(\Fg_2)$ is $(\Fg_1,\Fg_2)$-related.  Finally  Theorem \ref{theorem-cohomology-R(g)} shows that any Levi decomposition of Lie algebra $\Fg_2$ implies  a $R(\Fg_2)$-Levi decomposition.
\end{proof}

\begin{corollary}
Let $(G_1,G_2)$ be a matched pair of  Lie groups.  Assume that $L$ is a nucleus of $G$. Let   $\Fh$, $ \Fg_1$ and $\Fg_2$ denote the Lie algebras of $H:=G/L$, $G_1$ and $G_2$ respectively. Let also assume that $\Fh$ is $\Fg_1$-invariant and the natural action of $\Fh$ on $\Fg_1$ is given by derivations. Then for any representative $G_1\bowtie G_2$ module $M$, we have
\begin{equation}
HP^\bullet(R(G_2)\acl U(\Fg_1),\; ^\s{M}_\d)\;\;\cong \bigoplus_{i=\bullet\text{~mod~}2} H^i(\Fg_1\bowtie\Fg_2,\Fh,M).
\end{equation}
\end{corollary}
\begin{proof}
The Hopf algebra map  $\t:R(G_2)\ra R(\Fg_2)$ defined in \eqref{map-R(G)->R(g)} and the Hopf duality between $R(\Fg_2)$ and  $U(\Fg_2)$ defined in \eqref{F-V-pairing} guarantee the desired  Hopf duality between  $R(G_2)$ and $U(\Fg_2)$ as it is recalled here by
\begin{equation}\label{pairing-R(G)-U(g)}
\langle f\,,\, \xi\rangle_{\rm Gr} = \langle\t( f)\,,\, \xi\rangle_{\rm Alg}=\dt f(exp(t\xi)),\qquad \xi\in \Fg_2,\; f\in R(G_2).
\end{equation}
By  \eqref{theta-is-g-linear}, the map $\t$ is $U(\Fg_1)$-linear and hence the pairing \eqref{pairing-R(G)-U(g)} is $U(\Fg_1)$-balanced since the pairing \eqref{pairing-R(g)-U(g)} is $U(\Fg_1)$ balanced. Let us use the Sweedler notation $u\ns{0}\ot u\ns{1}$ for the coaction of $R(G_2)$ on $U(\Fg_1)$,  and $u\sns{0}\ot u\sns{1}$ for the coaction of $R(\Fg_2)$ on $U(\Fg_1)$. Now one uses the commutativity of the diagram \eqref{diagram-R(G)-R(g)}, and the compatibility of the coaction of $R(\Fg_2)$ on $U(\Fg_2)$ with the pairing $\langle\,,\,\rangle_{\rm Alg}$ to observe that
\begin{equation}
u\ns{0}\langle v\,,\, u\ns{1}\rangle_{\rm Gr} = u\ns{0} \langle v\,,\, \t(u\ns{0})\rangle_{\rm Alg}= u\sns{0} \langle v\,,\, u\sns{1}\rangle_{\rm Alg}= v\rt u.
\end{equation}
So far we have proved  that $R(G_2)$ is $(\Fg_1,\Fg_2)$-related. Finally, Theorem \ref{Theorem-R(G)-cohomology} shows that $\Fg_2=\Fh\ltimes\Fl$ is a $R(G_2)$-Levi decomposition. Here $\Fl\subseteq\Fg_2$ is the Lie algebra of $L$. We are now ready to apply Theorem \ref{Theorem-main}.
\end{proof}

\begin{corollary}
Let $(G_1,G_2)$ be a matched pair of connected affine algebraic groups and
  $G_2= G_2^{\rm red}\rtimes G_2^{\rm u}$ be a Levi decomposition of $G_2$. Let $\Fg_1$,  $\Fg_2^{\rm red}\subseteq \Fg_2$ be the Lie algebras of $G_1$, $G_2^{\rm red}$ and $G_2$ respectively. We assume that $\Fg_2^{\rm red}$ is $\Fg_1$ invariant and the natural action of $\Fg_2^{\rm red}$ on $\Fg_1$ is given by derivations. Then for any finite dimensional polynomial module  $M$ over $G_1\bowtie G_2$, we have
  \begin{equation}
HP^\bullet(\Pc(G_2)\acl U(\Fg_1),\; ^\s{M}_\d)\;\;\cong \bigoplus_{i=\bullet\text{~mod~}2} H^i(\Fg_1\bowtie\Fg_2,\Fg_2^{\rm red},M).
\end{equation}
\end{corollary}
\begin{proof}
We need to prove that the criteria of Theorem \ref{Theorem-main} are satisfied.  To this end, we first  observe that $\Pc(G_2)$  is in a Hopf duality with  $U(\Fg_2)$ via the pairing defined in \eqref{pairing-polynomial} recalled here by
\begin{equation}
\langle v\,,\, f\rangle_{\rm pol} = f^*(v)=(v \cdot f)(e),\qquad f\in\Pc(G_2),\; v\in U(\Fg_2).
\end{equation}
By Proposition \ref{starmapisg1linear}, the pairing $\langle\,,\, \rangle_{\rm pol}$ is $U(\Fg_1)$-balanced.  Since the coaction of  $\Pc(\Fg_2)$ on $U(\Fg_1)$ is obtained by the action of $G_2$ on $U(\Fg_1)$, we conclude that the Hopf algebra $\Pc(G_2)$ is  $(\Fg_1,\Fg_2)$-related. Finally, $\Fg_2=\Fg_2^{\rm red}\ltimes \Fg_2^{\rm u}$ is a $\Pc(G_2)$-Levi decomposition by Theorem \ref{theorem-cohomology-P(G)}. Here $\Fg_2^{\rm u}$ is the Lie algebra of $G_2^{\rm u}$.
\end{proof}


\begin{thebibliography}{9}

\bibitem{Bump}
Bump, D., {Lie Groups}, Springer, 2000.
\bibitem{CE} Chevalley, C. and  Eilenberg, S.,
Cohomology theory of Lie groups and Lie algebras.  Trans. Amer. Math. Soc.  {\bf 63},  (1948). 85–124.

 \bibitem{C-83}Connes, A.,  Cohomologie cyclique et foncteur $Ext^n$,
\textit{C.R. Acad. Sci. Paris}, Ser. I  Math., \textbf{296} (1983),
953-958.


 \bibitem{Connes-Book}
 Connes, A., {Noncommutative differential geometry}. Inst. Hautes Etudes Sci.
 Publ. Math. No. {\bf 62} (1985), 257--360.

\bibitem{Connes-Moscovici-98} Connes, A. and Moscovici, H., {Hopf algebras, cyclic
cohomology and the transverse index theorem}, \textit{Commun. Math. Phys.} 198 (1998), 199-246.


\bibitem{Connes-Moscovici-99} Connes, A. and Moscovici, H., {Cyclic cohomology and Hopf
algebras}.
 \textit{Lett. Math. Phys.}  48  (1999),  no. 1, 97--108.

\bibitem{Connes-Moscovici-00}
 Connes, A. and Moscovici, H., {Cyclic cohomology and Hopf algebra
symmetry.} Conference Mosh\'e Flato 1999 (Dijon).
 \textit{Lett. Math. Phys.} {\bf 52} (2000), no. 1, 1--28.


\bibitem{vEst} van Est, W. T., {Une application d'une m\'ethode de {C}artan-{L}eray}.
 \textit{Indag. Math.}  {\bf 17}  (1955),  542--544.

\bibitem{FultHarr}
Fulton, W. and  Harris, J., {Representation Theory, A First Course}, Springer, 2000.

\bibitem{Hadfield-Majid-07}
Hadfield, T. and  Majid, S., Bicrossproduct approach to the Connes-Moscovici Hopf algebra.  J. Algebra 312 (2007), no. 1, 228--256.

 \bibitem{Hajac-Khalkhali-Rangipour-Sommerhauser-04-1}
Hajac, P. M. , Khalkhali, M. , Rangipour, B. and Sommerh\"auser Y., {Stable anti-Yetter-Drinfeld modules.}  C. R. Math. Acad. Sci. Paris 338  (2004),  no. 8,
587--590.

\bibitem{Hajac-Khalkhali-Rangipour-Sommerhauser-04-2}
Hajac, P. M. , Khalkhali, M. , Rangipour, B. and Sommerh\"auser Y., {Hopf-cyclic homology and cohomology with coefficients.}  C. R. Math. Acad. Sci. Paris  338
(2004), no. 9, 667--672.

\bibitem{HC} Harish-Chandra On representations of Lie algebras. Ann. of Math. (2) {\bf 50}, (1949). 900–915.
\bibitem{Hoch2}
 Hochschild, G., {Basic Theory of Algebraic Groups and Lie Algebras}, Springer, 1981.

\bibitem{Hoch}
 Hochschild, G., {Algebraic Lie Algebras and Representative Functions},
 \textit{Illinoþs J. of Math.}, {\bf 4} (1960), 609--618.

\bibitem{Hoch3}
Hochschild, G., {Introduction To Affine Algebraic Groups}, Holden-Day, Inc., 1971.

\bibitem{Hochschild-61}
Hochschild, G., {Cohomology of Algebraic Linear Groups}, Illinois J. of Math. 5 (1961) 492--519

\bibitem{Hochschild-60}
Hochschild, G., {Algebraic Lie Algebras and Representative Functions},
 \textit{Illinoþs J. of Math.}, {\bf 4} (1960), 609--618.


\bibitem{Hochschild-Mostow-62}    Hochschild, G. and  Mostow, G. D.,  Cohomology of Lie groups.  Illinois J. Math. {\bf 6}  1962 367--401.

\bibitem{Hochschild-Mostow-61}  Hochschild, G.  and  Mostow,G. D.,  On the algebra of representative functions of an analytic group, Amer. J. Math., {\bf 83} (1961), 111--136. MR 25:5129.

\bibitem{Hochschild-Mostow-57}
Hochschild, G. and Mostow, G. D., {Representations and Representative Functions of Lie Groups}, \textit{Ann. of Math.}, {\bf 66} (1957), 495--542.



\bibitem{Khalkhali-Rangipour-06}
Khalkhali, M., and Rangipour, B.,Introduction to Hopf-Cyclic Cohomology, in: Noncommutative Geometry and Number Theory, in: Aspects Math., vol. E 37, Vieweg, 2006, pp. 155–178,

\bibitem{Khalkhali-Rangipour-04}
Khalkhali, M., and Rangipour, B.,  On the generalized cyclic Eilenberg-Zilber
  theorem,  Canad. Math. Bull. {\bf 47} (2004), no. 1, 38--48.

\bibitem{Khalkhali-Rangipour-03} Khalkhali, M., and Rangipour, B.,
  Invariant cyclic homology. $K$-Theory {\bf 28} (2003), no. 2, 183--205.


\bibitem{KN} S. Kumar, and  K. Neeb, Extensions of algebraic groups.  Studies in Lie theory,  365–376, Progr. Math., {\bf 243}, Birkhäuser Boston, Boston, MA, 2006.


\bibitem{Majid-book}
Majid, S., {Foundations of quantum group theory.} Cambridge University Press, Cambridge, 1995.


\bibitem{Majid-90}
 Majid, S., {Matched Pairs of Lie Groups Associated to Solutions of The Yang-Baxter Equations},
 \textit{Pacific J. of Math.}, {\bf 141} (1990), 311--332.

\bibitem{Moscovici-Rangipour-07} Moscovici, H. and  Rangipour, B.,  Cyclic cohomology of Hopf
algebras of transverse symmetries in codimension 1, \textit{ Adv. Math.},  \textbf{210} (2007), 323--374.
\bibitem{Moscovici-Rangipour-09} Moscovici, H. and  Rangipour, B.,  Hopf algebras of primitive Lie pseudogroups and Hopf cyclic cohomology.  Adv. Math. \textbf{ 220}  (2009),  no. 3,
    706--790.
\bibitem{Moscovici-Rangipour-011}
Moscovici, H. and Rangipour, B.,Hopf cyclic cohomology and transverse characteristic classes. arXiv:1009.0955.


\bibitem{TauvYu}
Tauvel, P. and  Yu, R. W. T., {Lie Algebras and Algebraic Groups}, Springer, 2000.





\end{thebibliography}
\end{document}